\documentclass{amsproc}
\usepackage{breqn,float,amssymb,pdflscape,rotating,hyperref,xcolor,setspace,graphicx}
\makeatletter
\@namedef{subjclassname@2020}{%
  \textup{2020} Mathematics Subject Classification}
\makeatother
\newtheorem{theorem}{Theorem}[section]

\theoremstyle{definition}

\newtheorem{example}[theorem]{Example}

\theoremstyle{remark}

\numberwithin{equation}{section}
\allowdisplaybreaks
\begin{document}
%
\doublespacing
\title{Derivation of some definite integrals}


\author{Robert Reynolds}
\address[Robert Reynolds]{Department of Mathematics and Statistics, York University, Toronto, ON, Canada, M3J1P3}
\email[Corresponding author]{milver73@gmail.com}
\thanks{}




\date{}

\dedicatory{}

\begin{abstract}
In this work derivations of definite integrals listed in Prudnikov volume I, Gradshteyn and Ryzhik and a few other tables are produced along with errata. Special cases of these integrals in terms of fundamental constants are also evaluated. The method used in these derivations is contour integration.
\end{abstract}

\maketitle
\section{Preliminaries}
Throughout this work we will use the Hurwitz-Lerch zeta function $\Phi(z,s,a)$ given in [DLMF,\href{https://dlmf.nist.gov/25.14}{25.14}], and its special cases; the Hurwitz zeta function $\zeta(s,a)$, [DLMF,\href{https://dlmf.nist.gov/25.11}{25.11}] and the polylogarithm $Li_{s}(z)$, [DLMF,\href{https://dlmf.nist.gov/25.12#ii}{25.12(ii)}], the log-gamma function given in [DLMF,\href{https://dlmf.nist.gov/25.11#vi.info}{25.11.18}], Catalan's constant $C$, [DLMF,\href{https://dlmf.nist.gov/26.6.E12}{25.11.40}], Glaisher's constant $A$, [DLMF,\href{https://dlmf.nist.gov/5.17.E5}{5.17.5}], Apery's constant $\zeta(3)$, [Wolfram,\href{https://mathworld.wolfram.com/AperysConstant.html}{1}], the Pochhammer's symbol $(a)_{n}$, [DLMF,\href{https://dlmf.nist.gov/5.2.iii}{5.2.5}], Stieltjes constant $\gamma_{n}$, [DLMF,\href{https://dlmf.nist.gov/25.2.E5}{25.2.5}]
\begin{equation}
\Phi'(i,0,u)=\log \left(\frac{\Gamma \left(\frac{u}{4}\right)}{2 \Gamma
   \left(\frac{u+2}{4}\right)}\right)+i \log \left(\frac{\Gamma \left(\frac{u+1}{4}\right)}{2 \Gamma
   \left(\frac{u+3}{4}\right)}\right)
\end{equation}
where $u\in\mathbb{C}$.
\begin{equation}
\Phi'(-i,0,a)=\log \left(\frac{\Gamma \left(\frac{a}{4}\right)}{2 \Gamma
   \left(\frac{a+2}{4}\right)}\right)-i \log \left(\frac{\Gamma \left(\frac{a+1}{4}\right)}{2 \Gamma
   \left(\frac{a+3}{4}\right)}\right)
\end{equation}
where $u\in\mathbb{C}$.
\begin{equation}
\text{Li}_s(z)=(2 \pi )^{s-1} \Gamma (1-s) \left(i^{1-s} \zeta \left(1-s,\frac{1}{2}-\frac{i \log (-z)}{2 \pi
   }\right)+i^{s-1} \zeta \left(1-s,\frac{i \log (-z)}{2 \pi }+\frac{1}{2}\right)\right)
\end{equation}
where $Re(s)>0$.
\begin{equation}\label{eq:hurwitz}
(2 i \pi )^{-s} \Gamma (s) \left((-1)^s \text{Li}_s\left(\frac{1}{z}\right)+\text{Li}_s(z)\right)=\zeta
   \left(1-s,\frac{\pi -i \log (-z)}{2 \pi }\right)
\end{equation}
where $Re(s)>0$.
\section{Introduction}
Tables of definite integrals, series and products have been  published since the 1800's and play a major role giving researchers summarized literature of mathematical formulae. Having books listing a vast number of formulae is a great resource, and having formal derivations is also of importance to researchers in the fields of science and engineering. 
Derivations aid researchers in deriving other formulae as other ideas could be combined with these derivations to produce other formulae thus expanding current literature. Derivations for tables of integrals can be traced back to the books which comprise the volume by Bierens de Haan \cite{bdh} up to the present day with the work by Moll \cite{moll1,moll2}, [Moll,\href{http://www.math.tulane.edu/~vhm/Table.html}{Table}] and Gallas \cite{gallas}. \\\\
The main theorem derived in this work could be described as an ssymptotic expansion of a definite integral. This type of integral representation has been studied by Handelsman \cite{handelsman}, Barnes \cite{barnes} and Wright \cite{wright} for reference. In 1943, Cooper \cite{cooper} published a book on asymptotic expansion of a function definned by a definite integral or contour integral. In this work we apply the contour integral method in \cite{reyn4} to a definite integral to derive an infinite series involving the incomplete gamma function and a generalized quotient polynomial which simplifies to the Hurwitz-Lerch zeta function for a special case of the second parameter in the incomplete gamma function. These characteristics of this theorem make it versatile as we can evaluate the definite integral in terms two special functions and there special cases.\\\\
The contour integral representation for the generalized Hippisley definite integral listed in equation (6.703) in \cite{Hippisley}, equation 2(13) in \cite{bdh} and equation (3.241.1) in \cite{grad} where the definition of the $\beta(z)$-function is given in equation (8.372.1) in \cite{grad} is given by;
\begin{equation}\label{eq:1.1}
\frac{1}{2\pi i}\int_{C}\int_{0}^{1}\frac{a^w w^{-k-1} t^{m+w-1}}{c t^b+1}dxdw=\frac{1}{2\pi i}\int_{C}\sum_{j=0}^{\infty}\frac{a^w (-c)^j w^{-k-1}}{b j+m+w}dw
\end{equation}
where $|Re(m)|<1,w\in\mathbb{C},Re(b)>0,|Re(c)|<1,k\in\mathbb{C}$.\\\\
In this paper we derive the definite integral given by
\begin{equation}\label{eq:1.2}
\int_0^1 \frac{t^{-1+m} (-\log (a t))^k}{1+c t^b} \, dt=\sum _{j=0}^{\infty } \frac{(-c)^j \Gamma (1+k,-((b
   j+m) \log (a)))}{a^{b j+m} (b j+m)^{k+1}}
\end{equation}
where $Re(m)>0$. and the parameters $k,b,c$ are general complex numbers. This definite integral has a wider range of evaluation for the parameters relative to previous results in the latter. This integral can be evaluated in terms of the Hurwitz-Lerch zeta and incomplete gamma function. The derivation involves two definite integrals and follows the method used by us in~\cite{reyn4}. This method involves using a form of the generalized Cauchy's integral formula given by
\begin{equation}\label{intro:cauchy}
\frac{y^k}{\Gamma(k+1)}=\frac{1}{2\pi i}\int_{C}\frac{e^{wy}}{w^{k+1}}dw.
\end{equation}
where $C$ is in general an open contour in the complex plane where the bilinear concomitant has the same value at the end points of the contour. We then multiply both sides by a function of $x$, then take a definite integral of both sides. This yields a definite integral in terms of a contour integral. Then we multiply both sides of Equation~(\ref{intro:cauchy})  by another function of $y$ and take the infinite sum of both sides such that the contour integral of both equations are the same.
\section{Derivations of the contour integral representations}
\subsection{Left-hand side contour integral}
Using a generalization of Cauchy's integral formula \ref{intro:cauchy}, we form the definite integral by replacing $y$ by $\log{ax}$ and multiply both sides by $\frac{t^{m-1}}{c t^b+1}$ to get;
\begin{multline}\label{eq:2.1}
\int_{0}^{1}\frac{t^{m-1} \log ^k(a x)}{k! \left(c t^b+1\right)}dt
=\frac{1}{2\pi i}\int_{0}^{1}\int_{C}\frac{w^{-k-1} t^{m-1} (a x)^w}{c t^b+1}dwdt\\
=\frac{1}{2\pi i}\int_{C}\int_{0}^{1}\frac{w^{-k-1} t^{m-1} (a x)^w}{c t^b+1}dtdw
\end{multline}
We are able to switch the order of integration over $t$ and $w$ using Fubini's theorem for multiple integrals see page 178 in \cite{gelca}, since the integrand is of bounded measure over the space $\mathbb{C} \times [0,1]$.
\subsubsection{The Incomplete Gamma~Function}
The incomplete gamma functions are given in equation [DLMF,\href{https://dlmf.nist.gov/8.4.E13}{8.4.13}], $\gamma(a,z)$ and $\Gamma(a,z)$, are defined by
\begin{equation}
\gamma(a,z)=\int_{0}^{z}t^{a-1}e^{-t}dt
\end{equation}
and
\begin{equation}
\Gamma(a,z)=\int_{z}^{\infty}t^{a-1}e^{-t}dt
\end{equation}
where $Re(a)>0$. The~incomplete gamma function has a recurrence relation given by
\begin{equation}
\gamma(a,z)+\Gamma(a,z)=\Gamma(a)
\end{equation}
where $a\neq 0,-1,-2,..$. The~incomplete gamma function is continued analytically by
\begin{equation}
\gamma(a,ze^{2m\pi i})=e^{2\pi mia}\gamma(a,z)
\end{equation}
and
\begin{equation}\label{eq:7}
\Gamma(a,ze^{2m\pi i})=e^{2\pi mia}\Gamma(a,z)+(1-e^{2\pi m i a})\Gamma(a)
\end{equation}
where $m\in\mathbb{Z}$, $\gamma^{*}(a,z)=\frac{z^{-a}}{\Gamma(a)}\gamma(a,z)$ is entire in $z$ and $a$. When $z\neq 0$, $\Gamma(a,z)$ is an entire function of $a$ and $\gamma(a,z)$ is meromorphic with simple poles at $a=-n$ for $n=0,1,2,...$ with residue $\frac{(-1)^n}{n!}$. These definitions are listed in [DLMF,\href{https://dlmf.nist.gov/8.2.i}{8.2(i)}] and [DLMF,\href{https://dlmf.nist.gov/8.2.ii}{8.2(ii)}].
The incomplete gamma functions are particular cases of the more general hypergeometric and Meijer G functions see section (5.6) and equation (6.9.2) in \cite{erd}. 
Some Meijer G representations we will use in this work are given by;
\begin{equation}\label{g1}
\Gamma (a,z)=\Gamma (a)-G_{1,2}^{1,1}\left(z\left|
\begin{array}{c}
 1 \\
 a,0 \\
\end{array}
\right.\right)
\end{equation}
and
\begin{equation}\label{g2}
\Gamma (a,z)=G_{1,2}^{2,0}\left(z\left|
\begin{array}{c}
 1 \\
 0,a \\
\end{array}
\right.\right)
\end{equation}
from equations (2.4) and (2.6a) in \cite{milgram}.  We will also use the derivative notation given by;
\begin{equation}\label{g3}
\frac{\partial \Gamma (a,z)}{\partial a}=\Gamma (a,z) \log (z)+G_{2,3}^{3,0}\left(z\left|
\begin{array}{c}
 1,1 \\
 0,0,a \\
\end{array}
\right.\right)
\end{equation}
from equations (2.19a) in \cite{milgram}, (9.31.3) in \cite{grad} and equations (5.11.1), (6.2.11.1) and (6.2.11.2) in \cite{luke}, and (6.36) in \cite{aslam}.
\subsubsection{Incomplete gamma function in terms of the contour integral}
In this section, we will once again use Cauchy's generalized integral formula, equation (\ref{intro:cauchy}), and take the infinite integral to derive equivalent sum representations for the contour integrals. We proceed using equation~(\ref{intro:cauchy}) and replace $y$ by $\log (a)+x$ and multiply both sides by $e^{m x}$ and replace $m\to m+bj$ and simplify. Next, take the infinite integral over $x\in[0,\infty)$ and simplify in terms of the incomplete gamma function using equation (3.382.4) in \cite{grad} and multiply both sides by $-(-c)^j$ finally take the infinite sum of both sides over $j\in[0,\infty)$ to obtain
\begin{multline}\label{eq:2.10}
-\sum_{j=0}^{\infty}\frac{(-c)^j a^{-b j-m} (-b j-m)^{-k-1} \Gamma (k+1,-((b j+m) \log (a)))}{k!}\\
=\frac{1}{2\pi i}\sum_{j=0}^{\infty}\int_{C}\frac{a^w (-c)^j w^{-k-1}}{b j+m+w}dw\\
=\frac{1}{2\pi i}\int_{C}\sum_{j=0}^{\infty}\frac{a^w (-c)^j w^{-k-1}}{b j+m+w}dw
\end{multline}
We are able to switch the order of integration and summation over $w$ using Tonellii's theorem for  integrals and sums see page 177 in \cite{gelca}, since the summand is of bounded measure over the space $\mathbb{C} \times [0,\infty)$
\begin{theorem}
For all $Re(b)>0,Re(m)>0,|Re(c)|<1$ then,
\begin{equation}\label{eq:2.11}
\int_0^1 \frac{t^{-1+m} (-\log (a t))^k}{1+c t^b} \, dt=\sum _{j=0}^{\infty } \frac{(-c)^j \Gamma (1+k,-((b
   j+m) \log (a)))}{a^{b j+m} (b j+m)^{k+1}}
\end{equation}
\end{theorem}
\begin{proof}
Since the right-hand sides of equations (\ref{eq:2.1}) and (\ref{eq:2.10}) are equal relative to equation (\ref{eq:1.1}), we may equate the left-hand sides and simplify the gamma function to yield the stated result.
\end{proof}
\section{Table of definite integrals, products and series}
\subsection{Table 4.327 in Gradshteyn and I.M. Ryzhik}
\begin{example}
Equation (4.327.1) in \cite{grad}. Use equation (\ref{eq:2.11}) and set $k=-1,a\to e^a$ then take the indefinite integral with respect to $a$. Next take the first partial derivative with respect to $c$ and set $m=1,t\to x, b=2,c=1$. Next form a second equation by replacing $-a$ and add these two equations and simplify. We also include the derivation listed in equation (4.327).
\begin{multline}\label{eq:3.1}
\int_0^1 \frac{\log \left(a^2+\log ^2(x)\right)}{1+x^2} \, dx\\
=\sum _{j=0}^{\infty } \frac{(-1)^j e^{-i a (1+2 j)}
   \left(E_1(-i a (1+2 j))+e^{2 i a (1+2 j)} E_1(i (a+2 a j))+e^{i (a+2 a j)} 2 \log (a)\right)}{1+2 j}\\
=\pi  \log
   \left(\frac{2 \Gamma \left(\frac{2 a+3 \pi }{4 \pi }\right)}{\Gamma \left(\frac{2 a+\pi }{4 \pi
   }\right)}\right)+\frac{1}{2} \pi  \log \left(\frac{\pi }{2}\right)
\end{multline}
where $Re(a)>0$.
\end{example}
\begin{example}
Here we use the right-hand sides of equation (\ref{eq:3.1}), equate and take the exponential of both sides and simplify.
\begin{multline}
\prod _{j=0}^{\infty } \exp \left(\frac{(-1)^j e^{-i a (1+2 j)} \left(E_1(-i a (1+2 j))+e^{2 i a (1+2 j)} \Gamma (0,i
   (a+2 a j))\right)}{\pi +2 j \pi }\right)\\
=\frac{\sqrt{\frac{1}{a}} \sqrt{2 \pi } \Gamma \left(\frac{3}{4}+\frac{a}{2 \pi
   }\right)}{\Gamma \left(\frac{2 a+\pi }{4 \pi }\right)}
\end{multline}
where $Re(a)>0$.
\end{example}
\begin{example}
Equation (4.327.2) in \cite{grad}. Use equation (\ref{eq:3.1}) and set $a\to a/2$ and simplify.
\begin{multline}\label{eq:3.3}
\int_0^1 \frac{\log \left(a^2+4 \log ^2(x)\right)}{1+x^2} \, dx\\
=\sum _{j=0}^{\infty } \left(\frac{(-1)^j
   e^{-\frac{1}{2} i a (1+2 j)} E_1\left(-\frac{1}{2} i a (1+2 j)\right)}{1+2 j}+\frac{(-1)^j e^{\frac{1}{2} i a (1+2 j)}
   E_1\left(i \left(\frac{a}{2}+a j\right)\right)}{1+2 j}\right)\\+\frac{1}{2} \pi  \log (a)\\
=\pi  \log \left(\frac{2 \Gamma
   \left(\frac{a+3 \pi }{4 \pi }\right)}{\Gamma \left(\frac{a+\pi }{4 \pi }\right)}\right)+\frac{1}{2} \pi  \log (\pi
   )
\end{multline}
where $Re(a)>0$.
\end{example}
\begin{example}
Here we take the exponential function of both sides of equation (\ref{eq:3.3}) and simplify.
\begin{multline}
\prod _{j=0}^{\infty } \exp \left(\frac{2 (-1)^j e^{-\frac{1}{2} i a (1+2 j)} \left(E_1\left(-\frac{1}{2} i a (1+2
   j)\right)+e^{i a (1+2 j)} E_1\left(\frac{1}{2} i (a+2 a j)\right)\right)+(1+2 j) \pi  \log (a)}{2 (\pi +2 j \pi
   )}\right)\\
=\frac{2 \sqrt{\pi } \Gamma \left(\frac{1}{4} \left(3+\frac{a}{\pi }\right)\right)}{\Gamma \left(\frac{a+\pi
   }{4 \pi }\right)}
\end{multline}
where $Re(a)>0$.
\end{example}
\begin{example}
Derivation of equation (4.327.3) in \cite{grad}. Errata. Here we use (\ref{eq:2.11}) and set $b=0$ and simplify the infinite series. Next set $k=0,a\to e^a$. Next form a second equation by replacing $a\to -a$ and add these two equations and simplify.
\begin{multline}
\int_0^1 t^{-1+m} \log \left(a^2+\log ^2(t)\right) \, dt=\frac{e^{-i a m} E_1(-i a m)+e^{i a m} E_1(i a m)+2 \log
   (a)}{m}
\end{multline}
where $Re(a)>0$.
\end{example}
\subsection{Table (4.325) in Gradshteyn and I.M. Ryzhik}
\begin{example}
Equation (4.325.1). Use equation (\ref{eq:2.11}) and set $b=1,c=1,a=1,m=1$ and simplify the series in terms of the Riemann zeta function. Next take the first partial derivative with respect to $k$ and apply l'Hopital's rule as $k\to 0$ and simplify using equations [DLMF\href{https://dlmf.nist.gov/25.11.E2}{25.11.2}] and [DLMF,\href{https://dlmf.nist.gov/25.6.E11}{25.6.11}].
\begin{equation}
\int_0^1 \frac{\log (-\log (t))}{1+t} \, dx=-\frac{1}{2} \log ^2(2)
\end{equation}
\end{example}
\begin{example}
Equation (4.325.2) Extended form. Use equation (\ref{eq:2.11}) and set $b\to 1,c\to e^{-i \lambda },a\to 1,m\to 1$ and simplify the series in terms of the polylogarithm function. Next take the first partial derivative with respect to $k$ and set $k=0$ and simplify using equations [DLMF,\href{https://dlmf.nist.gov/25.14.E3}{25.14.3}] and [DLMF,\href{https://dlmf.nist.gov/25.12.E13}{25.12.12}].
\begin{equation}
\int_0^1 \frac{\log (-\log (t))}{e^{i \lambda }+t} \, dt=-\gamma  \log \left(1+e^{-i \lambda
   }\right)-Li'_{1}\left(-e^{-i \lambda }\right)
\end{equation}
where $\lambda\in\mathbb{C}$.
\end{example}
\begin{example}
Equation (4.325.3). Use equation (\ref{eq:2.11}) and set $b\to 1,a\to 1,m\to 0$ and simplify theseries in terms of the Riemann zeta function. Next take the first partial derivative with respect to $c$ and set $c=1$ and simplify using  [DLMF\href{https://dlmf.nist.gov/25.11.E2}{25.11.2}],  [DLMF,\href{https://dlmf.nist.gov/25.6.E11}{25.6.11}] and [DLMF,\href{https://dlmf.nist.gov/5.2.E3}{5.2.3}].
\end{example}
\begin{example}
Equation (4.325.4). Use equation (\ref{eq:2.11}) and set $b\to 2,c\to 1,a\to 1,m\to 1$ and simplify the series in terms of the Hurwitz zeta function. Next take the first partial derivative with respect to $k$ and apply l'Hopital's rule as $k\to 0$ and replace $t\to x$ and simplify using  [DLMF,\href{https://dlmf.nist.gov/25.2#E5}{25.2.5}] and [DLMF,\href{https://dlmf.nist.gov/5.2.E3}{5.2.3}].
\begin{multline}
\int_0^1 \frac{\log (-\log (x))}{1+x^2} \, dx
=\frac{1}{4} \left(-\pi  (\gamma +\log (4))-\gamma
   _1\left(\frac{1}{4}\right)+\gamma _1\left(\frac{3}{4}\right)\right)\\
=\frac{1}{2} \pi  \log \left(\frac{\sqrt{2 \pi }
   \Gamma \left(\frac{3}{4}\right)}{\Gamma \left(\frac{1}{4}\right)}\right)
\end{multline}
\end{example}
\begin{example}
Equation (4.325.5). Use equation (\ref{eq:2.11}) and set $b\to 1,a\to 1,t\to x,m\to 0$ and simplify the series in terms of the polylogarithm function. Next take the first partial derivative with respect to $k$ and apply l'Hopital's rule as $k\to 0$. Next form to equations by replacing $c\to -(-1)^{2/3}, c\to (-1)^{1/3}$ and take their difference and simplify using [DLMF,\href{https://dlmf.nist.gov/25.14.E3}{25.14.3}] and [DLMF,\href{https://dlmf.nist.gov/25.12.E13}{25.12.12}].
\begin{multline}
\int_0^1 \frac{\log (-\log (x))}{1+x+x^2} \, dx=-\frac{i \left(-\frac{1}{3} i \gamma  \pi
   -Li_{1}\left(-\sqrt[3]{-1}\right)+Li'
_{1}\left((-1)^{2/3}\right)\right)}{\sqrt{3}}
   \\
=\frac{\pi  \log \left(\frac{\sqrt[3]{2 \pi } \Gamma \left(\frac{2}{3}\right)}{\Gamma
   \left(\frac{1}{3}\right)}\right)}{\sqrt{3}}
\end{multline}
\end{example}
\begin{example}
Equation (4.325.6). Use equation (\ref{eq:2.11}) and set $a=b=1$ and simplify the series in terms of the Hurwitz-Lerch zeta function. Next take the first partial derivative with respect to $k$ and set $t\to x, k=0$. Then form two equations by replacing $c\to (-1)^{2/3}$ and multiply both sides by $\frac{\sqrt[3]{-1}}{1+\sqrt[3]{-1}}$ and for the second $c\to -(-1)^{1/3}$ and multiply both sides by $\frac{1}{1+\sqrt[3]{-1}}$ and add these two equations with $m=1$ and simplify using [DLMF,\href{https://dlmf.nist.gov/25.14.i}{25.14.1}], [DLMF,\href{https://dlmf.nist.gov/25.14.E3}{25.14.3}] and [DLMF,\href{https://dlmf.nist.gov/25.2.E5}{25.2.5}].
\begin{multline}
\int_0^1 \frac{\log (-\log (x))}{1-x+x^2} \, dx=-\frac{2 \gamma  \pi +3 i
   Li'_{1}\left(\sqrt[3]{-1}\right)-3 i Li'_{1}\left(-(-1)^{2/3}\right)}{3
   \sqrt{3}}\\
=\frac{(2 \pi ) \left(\frac{5}{6} \log (2 \pi )-\log \left(\Gamma
   \left(\frac{1}{6}\right)\right)\right)}{\sqrt{3}}
\end{multline}
\end{example}
\begin{example}
Equation (4.325.7). Use equation (\ref{eq:2.11}) and set $a=b=m=1$ and simplify the series in terms of the polylogarithm function. Next take the first partial derivative with respect to $k$ and set $k=0,c\to 1/c,t\to x$. Next form two equations by replacing $c\to e^{i t}$ and $c\to e^{-it}$ then take their difference and simplify by multiplying both sides by $\frac{1}{2} i \csc (t)$ and simplifying using [DLMF,\href{https://dlmf.nist.gov/25.2.E5}{25.2.5}] and [DLMF,\href{https://dlmf.nist.gov/25.14.E3}{25.14.3}].
\begin{multline}
\int_0^1 \frac{\log (-\log (x))}{1+x^2+2 x \cos (t)} \, dx
=\frac{1}{2} i \csc (t) \left(i \gamma 
   t-Li'_{1}\left(-e^{-i t}\right)+Li'_{1}\left(-e^{i t}\right)\right)\\
=\frac{\pi  \log
   \left(\frac{(2 \pi )^{t/\pi } \Gamma \left(\frac{1}{2}+\frac{t}{2 \pi }\right)}{\Gamma \left(\frac{1}{2}-\frac{t}{2\pi }\right)}\right)}{2 \sin (t)}
\end{multline}
where $t\in\mathbb{C}$.
\end{example}
\begin{example}
Equation (4.325.8). Use equation (\ref{eq:2.11}) and set $a=1,b=0$ and simplify the series. Next take the first partial derivative with respect to $k$ and replace $k=0,t\to x,m \to v$ and simplify using [DLMF,\href{https://dlmf.nist.gov/25.2.E5}{25.2.5}].
\begin{equation}
\int_0^1 x^{-1+v} \log (-\log (x)) \, dx=-\frac{\gamma +\log (v)}{v}
\end{equation}
where $|Re(v)|<1$.
\end{example}
\begin{example}
Equation (4.325.9). Extended form. Use equation (\ref{eq:2.11}) and set $a=1$ and simplify the series in terms of the Hurwitz-Lerch zeta function. Next set $b\to 2n,c=-1,t\to x$. Next replace $m\to m+1$ and form two equations by replacing $m\to n-2,m\to n$ and take their difference. Next take the first partial derivative with respect to $k$ and set $k=0$ and simplify using [DLMF,\href{https://dlmf.nist.gov/25.2.E5}{25.2.5}], [DLMF,\href{https://dlmf.nist.gov/25.14.E2}{25.14.2}], [DLMF,\href{https://dlmf.nist.gov/25.2.E5}{25.2.4}], and [DLMF,\href{https://dlmf.nist.gov/25.14.i}{25.14.1}].
\begin{multline}
\int_0^1 \frac{x^{-2+\alpha } \left(-1+x^2\right) \log \left(\log \left(\frac{1}{x}\right)\right)}{-1+x^{2 \alpha
   }} \, dx\\
=\frac{(\gamma +\log (2 \alpha )) \left(\psi ^{(0)}\left(\frac{-1+\alpha }{2 \alpha }\right)-\psi
   ^{(0)}\left(\frac{1+\alpha }{2 \alpha }\right)\right)-\gamma _1\left(\frac{-1+\alpha }{2 \alpha }\right)+\gamma
   _1\left(\frac{1+\alpha }{2 \alpha }\right)}{2 \alpha }
\end{multline}
where $\alpha\in\mathbb{C}$.
\end{example}
\begin{example}
Equation (4.325.10). Use equation (\ref{eq:2.11}) and set $c=1,b=2,a=1$ and simplify the series in terms of the Hurwitz zeta function. Next take the first partial derivative with respect to $k$ and set $k=-1/2,m=1,t\to x$ and simplify using [DLMF,\href{https://dlmf.nist.gov/25.14.E2}{25.14.2}] and 
[DLMF,\href{https://dlmf.nist.gov/25.2.E5}{25.2.5}].
\begin{multline}
\int_0^1 \frac{\log (-\log (x))}{\left(1+x^2\right) \sqrt{-\log (x)}} \, dx\\
=\frac{1}{2} \sqrt{\pi }
   \left(-\left(\left(\zeta \left(\frac{1}{2},\frac{1}{4}\right)-\zeta \left(\frac{1}{2},\frac{3}{4}\right)\right)
   (\gamma +\log
   (16))\right)\right. \\ \left.
+\zeta'\left(\frac{1}{2},\frac{1}{4}\right)-\zeta'\left(\frac{1}{
   2},\frac{3}{4}\right)\right)
\end{multline}
\end{example}
\begin{example}
Equation (4.325.11). Use equation (\ref{eq:2.11}) and set $b=0,a->1$ and simplify the series. Next take the first partial derivative with respect to $k$ and set $k=-1/2,t\to x,m\to v$ and simplify.
\begin{equation}
\int_0^1 \frac{x^{-1+v} \log (-\log (x))}{\sqrt{-\log (x)}} \, dx=-\frac{\sqrt{\pi } (\gamma +\log (4)+\log
   (v))}{\sqrt{v}}
\end{equation}
where $|Re(v)|<1$.
\end{example}
\begin{example}
Equation (4.325.12). Use equation (\ref{eq:2.11}) and set $b=0,a=1$ and simplify the series. Next take the first partial derivative with respect to $k$ and set $k\to u-1,t\to x,m\to v$ and simplify.
\begin{multline}
\int_0^1 x^{-1+v} (-\log (x))^{-1+u} \log (-\log (x)) \, dx=-v^{-u} \Gamma (u) (\log (v)-\psi ^{(0)}(u))
\end{multline}
where $|Re(v)|<1$.
\end{example}
\subsection{Derivation of some definite integrals}
\begin{example}
Equation (4.325.12) in \cite{grad}. Errata and generalized form. Use equation (\ref{eq:2.11}) and set $k=-1,a\to e^a,m\to m+b$. Next form a second equation by replacing $a\to -a$ and take their difference and simplify.
\begin{multline}
\int_0^1 \frac{t^{-1+m} \log \left(1+c t^b\right)}{a^2+\log ^2(t)} \, dt\\
=\sum _{j=0}^{\infty } \frac{i (-1)^j
   c^{1+j} e^{-i a (b+b j+m)} \left(-\Gamma (0,-i a (b+b j+m))+e^{2 i a (b+b j+m)} \Gamma (0,i a (b+b j+m))\right)}{2 a
   (1+j)}
\end{multline}
where $|Re(m)|<1,Re(a)>0,Re(b)>0$.
\end{example}
\begin{example}
Equation (2.6.18.2) in \cite{prud1} extended form. Use equation (\ref{eq:2.11}) and set $b=0$ and simplify the series in terms of the incomplete gamma function. Next replace $a\to e^a,k=-1$ and take the $h$-th partial derivative with respect to $a$ and simplify with $h\to n-1,t\to x$ and [DLMF,\href{https://dlmf.nist.gov/15.5.i}{15.5}].
\begin{equation}
\int_0^1 \frac{x^{m-1}}{(a+\log (x))^n} \, dx=-a^{1-n} e^{-a m} E_n(-a m)
\end{equation}
where $Im(a)>0$.
\end{example}
\begin{example}
Equation (2.6.18.6) in \cite{prud1}. Alternate form. Use equation (\ref{eq:2.11}) and set $k=-1,a\to e^{a},c=1$ and simplify the series. Next form two equations by replacing $a\to -a$ and take their difference and simplify with replacing $a\to i a,b\to u,m\to \frac{u}{2},t\to x$ and using equation (11.3.1) on pp.775 in \cite{prud1}.
\begin{multline}
\int_0^1 \frac{x^{-1+\frac{u}{2}}}{\left(1+x^u\right) \left(a^2+\log ^2(x)\right)} \, dx\\
=\sum _{j=0}^{\infty }
   \frac{i (-1)^j e^{-\frac{1}{2} i a (u+2 j u)} \left(-\Gamma \left(0,-\frac{1}{2} i a (u+2 j u)\right)+e^{i a (u+2 j
   u)} \Gamma \left(0,\frac{1}{2} i a (u+2 j u)\right)\right)}{2 a}\\
=\frac{-\psi ^{(0)}\left(\frac{\pi +a u}{4 \pi
   }\right)+\psi ^{(0)}\left(\frac{1}{4} \left(3+\frac{a u}{\pi }\right)\right)}{4 a}
\end{multline}
where $|Re(u)|<1,Re(a)>0,Re(u)>0$.
\end{example}
\begin{example}
Equation (2.6.18.16) in \cite{prud1}. Alternate form. Use equation (\ref{eq:2.11}) and take the indefinite integral with respect to $c$ set $k=-1,a\to e^a$. Next form a second equation by replacing $a\to -a$ and take their difference. Replace $a\to ai,m\to m+b$ and simplify.
\begin{multline}
\int_0^1 \frac{t^{-1+m} \log \left(1+c t^b\right)}{a^2+\log ^2(t)} \, dt\\
=\sum _{j=0}^{\infty } \frac{i (-1)^j
   c^{1+j} e^{-i a (b+b j+m)} \left(-\Gamma (0,-i a (b+b j+m))+e^{2 i a (b+b j+m)} \Gamma (0,i a (b+b j+m))\right)}{2 a
   (1+j)}
\end{multline}
where $|Re(m)|<1,Re(a)>0,Re(b)>0$.
\end{example}
\subsection{Bierens de Haan Table 147}
\begin{example}
Equation 147(1). Use equation (\ref{eq:2.11}) and set $b=0,a=1$ then take the first partial derivative with respect to $k$ and set $k=0,t\to x,m\to q$ and simplify using [Wolfram,\href{https://mathworld.wolfram.com/Euler-MascheroniConstant.html}{47}].
\begin{equation}
\int_0^1 x^{-1+q} \log \left(\log \left(\frac{1}{x}\right)\right) \, dx=-\frac{\gamma +\log (q)}{q}
\end{equation}
where $|Re(q)|<1$.
\end{example}
\begin{example}
Equation 147(2). Use equation (\ref{eq:2.11}) and set $b=0,m\to q,t\to x,k\to p,a=1$, then take the first partial derivative with respect to $p$ and simplify using [Wolfram,\href{https://mathworld.wolfram.com/Euler-MascheroniConstant.html}{48}].
\begin{multline}
\int_0^1 x^{-1+q} \log ^p\left(\frac{1}{x}\right) \log \left(\log \left(\frac{1}{x}\right)\right) \, dx=-q^{-1-p}
   \Gamma (1+p) (\log (q)-\psi ^{(0)}(1+p))
\end{multline}
where $|Re(q)|<1$.
\end{example}
\begin{example}
Equation 147(3). Use equation (\ref{eq:2.11}) and set $b=0,m\to q,t\to x,k\to p, a=1$, then take the first partial derivative with respect to $p$ set $p=-1/2$ and simplify using [Wolfram,\href{https://mathworld.wolfram.com/Euler-MascheroniConstant.html}{48}].
\begin{equation}
\int_0^1 \frac{x^{-1+q} \log \left(\log \left(\frac{1}{x}\right)\right)}{\sqrt{\log \left(\frac{1}{x}\right)}} \,
   dx=-(\gamma +\log (4 q)) \sqrt{\frac{\pi }{q}}
\end{equation}
where $|Re(q)|<1$.
\end{example}
\begin{example}
Equation 147(4). Use equation (\ref{eq:2.11}) and $b=2,c=1,t\to x,a=1$ and simplify the series in terms of the Hurwitz zeta function. Next take the first partial derivative with respect to $k$ and set $k=-1/2,m=1$ and simplify using [DLMF,\href{https://dlmf.nist.gov/25.14.E2}{25.14.2}] and [Wolfram,\href{https://mathworld.wolfram.com/Euler-MascheroniConstant.html}{49}].
\begin{multline}
\int_0^1 \frac{\log \left(\log \left(\frac{1}{x}\right)\right)}{\left(1+x^2\right) \sqrt{\log
   \left(\frac{1}{x}\right)}} \, dx=\frac{1}{2} \sqrt{\pi } \left(-\left(\left(\zeta
   \left(\frac{1}{2},\frac{1}{4}\right)-\zeta \left(\frac{1}{2},\frac{3}{4}\right)\right) (\gamma +\log
   (16))\right)\right. \\ \left.
+\zeta'\left(\frac{1}{2},\frac{1}{4}\right)-\zeta'\left(\frac{1}{2
   },\frac{3}{4}\right)\right)
\end{multline}
\end{example}
\begin{example}
Equation 147(5). Use equation (\ref{eq:2.11}) and set $b=2,c=1,t\to x,m\to m+1$ and simplify the series in terms of the Hurwitz zeta function. Next take the first partial derivative with respect to $k$ and apply l'Hopital's rule as $k\to 0$ and replace $m\to p$ and simplify using [DLMF,\href{https://dlmf.nist.gov/25.11.E1}{25.11.1}] and [DLMF,\href{https://dlmf.nist.gov/25.2.E5}{25.2.4}].
\begin{multline}
\int_0^1 \frac{x^{-p} \left(1+x^{2 p}\right) \log \left(\log \left(\frac{1}{x}\right)\right)}{1+x^2} \,
   dx\\
=\frac{1}{4} \left(-2 \pi  (\gamma +\log (4)) \sec \left(\frac{p \pi }{2}\right)-\gamma
   _1\left(\frac{1-p}{4}\right)+\gamma _1\left(\frac{3-p}{4}\right)-\gamma _1\left(\frac{1+p}{4}\right)\right. \\ \left.+\gamma
   _1\left(\frac{3+p}{4}\right)\right)
\end{multline}
where $|Re(p)|<1$.
\end{example}
\begin{example}
Equation 147(6). Use equation (\ref{eq:2.11}) and set $b=2,c=-1,t\to x,a=1,m\to m+1$ and simplify the series in terms of the Hurwitz zeta function. Next form a second equation by replacing $m\to -m$ and take their difference. Next take the first partial derivative with respect to $k$ and apply l'Hopital's rule as $k\to 0$ and simplify using [DLMF,\href{https://dlmf.nist.gov/25.2.E5}{25.2.5}] and [DLMF,\href{https://dlmf.nist.gov/25.11.i}{25.11.1}].
\begin{multline}
\int_0^1 \frac{\log \left(\log \left(\frac{1}{x}\right)\right) \left(x^p-x^{-p}\right)}{1-x^2} \, dx\\
=\frac{1}{2}
   \left(\gamma _1\left(\frac{1-p}{2}\right)-\gamma _1\left(\frac{1+p}{2}\right)+\pi  (\gamma +\log (2)) \tan
   \left(\frac{p \pi }{2}\right)\right)
\end{multline}
where $|Re(p)|<1$.
\end{example}
\begin{example}
Equation 147(7). Use equation (\ref{eq:2.11}) and set $b=1,t\to x,a=1,m=0$ and simplify the series in terms of the polylogarithm function. Next take the first partial derivative with respect to $c$ and set $c=1$, then take the first partial derivative with respect to $k$ and set $k=0$ and simplify using [DLMF,\href{https://dlmf.nist.gov/25.2.E5}{25.2.5}] and [DLMF,\href{https://dlmf.nist.gov/25.12#ii}{25.14.3}].
\begin{equation}
\int_0^1 \frac{\log \left(\log \left(\frac{1}{x}\right)\right)}{(1+x)^2} \, dx=\frac{1}{2} \left(-\gamma +\log
   \left(\frac{\pi }{2}\right)\right)
\end{equation}
\end{example}
\begin{example}
Equation 147(8). Errata  Use equation (\ref{eq:2.11}) and set $b=1,t\to x,a=1$ and simplify the series in terms of the Hurwitz-Lerch zeta function. Next we form two equations by replacing $c\to -(1)^{1/3},c\to -(-1)^{2/3}$ and take their difference. Next take the first partial derivative with respect to $k$ and set $k=-1/2$ and simplify using [DLMF,\href{https://dlmf.nist.gov/25.14.i}{25.14.1}] and [DLMF,\href{https://dlmf.nist.gov/25.12#ii}{25.14.3}].
\begin{multline}
\int_0^1 \frac{\log \left(\log \left(\frac{1}{x}\right)\right)}{\left(1+x+x^2\right) \sqrt{\log
   \left(\frac{1}{x}\right)}} \, dx\\
=i \sqrt{\frac{\pi }{3}} \left(i (\gamma +\log (4)) \left(\zeta
   \left(\frac{1}{2},\frac{1}{3}\right)-\zeta
   \left(\frac{1}{2},\frac{2}{3}\right)\right)+Li'_{1/2}\left(-\sqrt[3]{-1}\right)-Li'_{1/2}\left((-1)^{2/3}\right)\right)
\end{multline}
\end{example}
\begin{example}
Equation 147(9). Use equation (\ref{eq:2.11}) and set $b=1,t\to x,a=1$ and simplify the series in terms of the Hurwitz-Lerch zeta function. Next we form two equations by replacing $c\to e^{i\lambda},c\to e^{-i\lambda}$ and take their difference set $m=0$ and convert the Hurwitz-Lerch zeta function to the polyloagarithm function using  [DLMF,\href{https://dlmf.nist.gov/25.14.i}{25.14.1}]  and simplify using [DLMF,\href{https://dlmf.nist.gov/25.12#ii}{25.14.3}].
\begin{equation}
\int_0^1 \frac{\log \left(\log \left(\frac{1}{x}\right)\right)}{1+x^2+2 x \cos (\lambda )} \, dx=\frac{1}{2} \pi
    \csc (\lambda ) \log \left(\frac{(2 \pi )^{\lambda /\pi } \Gamma \left(\frac{\pi +\lambda }{2 \pi }\right)}{\Gamma
   \left(\frac{\pi -\lambda }{2 \pi }\right)}\right)
\end{equation}
where $|Re(\lambda)|<\pi/2$
\end{example}
\begin{example}
Equation 147(10). Use equation (\ref{eq:2.11}) and set $c=1,b=2,m=1,a\to e^a$ and form a second equation by replacing $a\to -a$ and adding the two equations and simplifying the term on the left-hand side.
\begin{multline}
\int_0^1 \frac{\log \left(q^2+\log ^2(x)\right)}{1+x^2} \, dx\\
=\sum _{j=0}^{\infty } \frac{(-1)^j \left(e^{-i (1+2
   j) q} E_1(-i (1+2 j) q)+e^{i (q+2 j q)} E_1(i (q+2 j q))+2 \log (q)\right)}{1+2 j}\\
=\pi  \log \left(\frac{2 \Gamma
   \left(\frac{2 q+3 \pi }{4 \pi }\right)}{\Gamma \left(\frac{2 q+\pi }{4 \pi }\right)}\right)+\frac{1}{2} \pi  \log
   \left(\frac{\pi }{2}\right)
\end{multline}
where $|Re(q)|<1$.
\end{example}
\begin{example}
Equation 147(11). Use equation (\ref{eq:2.11}) and set $b=2,t\to x,a\to e^{a},m\to m+1,c=1$ then take the first partial derivative with respect to $k$ and set $k=0$ and simplify. Next form a second equation by replacing $a\to -a$ and add the two equations and simplify with replacing $a\to ai$. Next form another equation by replacing $m\to s$ and add these two equations, finally replace $m\to -b/a, s\to b/a$ and simplify.
\begin{multline}\label{eq:3.32}
\int_0^1 \frac{\left(x^{-\frac{b}{a}}+x^{b/a}\right) \log \left(q^2+\log ^2(x)\right)}{1+x^2} \, dx\\
=\sum
   _{j=0}^{\infty } (-1)^j a \left(\frac{e^{-i \left(1-\frac{b}{a}+2 j\right) q} E_1\left(-i \left(1-\frac{b}{a}+2
   j\right) q\right)}{a-b+2 a j}+\frac{e^{-\frac{i (a+b+2 a j) q}{a}} E_1\left(-\frac{i (a+b+2 a j) q}{a}\right)}{a+b+2
   a j}\right. \\ \left.
+\frac{e^{i \left(1-\frac{b}{a}+2 j\right) q} E_1\left(i \left(1-\frac{b}{a}+2 j\right) q\right)+2 \log
   (q)}{a-b+2 a j}+\frac{e^{\frac{i (a+b+2 a j) q}{a}} E_1\left(\frac{i (a+b+2 a j) q}{a}\right)+2 \log (q)}{a+b+2 a
   j}\right)
\end{multline}
where $Re(q)>0$.
\end{example}
\begin{example}
Equation 147(12). Use equation (\ref{eq:3.32}) and set $q=a/2$ and simplify.
\begin{multline}
\int_0^1 \frac{\left(x^{-\frac{b}{a}}+x^{b/a}\right) \log \left(\frac{a^2 \pi ^2}{4}+\log ^2(x)\right)}{1+x^2}
   \, dx\\
=\sum _{j=0}^{\infty } (-1)^j a \left(\frac{e^{-\frac{1}{2} i a \left(1-\frac{b}{a}+2 j\right) \pi }
   E_1\left(-\frac{1}{2} i a \left(1-\frac{b}{a}+2 j\right) \pi \right)}{a-b+2 a j}\right. \\ \left.
+\frac{e^{-\frac{1}{2} i (a+b+2 a j)
   \pi } E_1\left(-\frac{1}{2} i (a+b+2 a j) \pi \right)}{a+b+2 a j}\right. \\ \left.
+\frac{e^{\frac{1}{2} i a \left(1-\frac{b}{a}+2
   j\right) \pi } E_1\left(\frac{1}{2} i a \left(1-\frac{b}{a}+2 j\right) \pi \right)+2 \log \left(\frac{a \pi
   }{2}\right)}{a-b+2 a j}\right. \\ \left.+\frac{e^{\frac{1}{2} i (a+b+2 a j) \pi } E_1\left(\frac{1}{2} i (a+b+2 a j) \pi \right)+2
   \log \left(\frac{a \pi }{2}\right)}{a+b+2 a j}\right)
\end{multline}
where $Re(b)>0,Re(a)>0$.
\end{example}
\begin{example}
Equation 147(13). Errata. Use equation (\ref{eq:2.11}) and set $m=1,c=1,b=2,\to x,a\to e^{a}$ then take the first partial derivative with respect to $k$ and set $k=0$. Next form a second equation by replacing $a\to -a$ and add these two equations and simplify.
\begin{multline}
\int_0^1 \frac{\log \left(\frac{\pi ^2}{4}+\log ^2(x)\right)}{1+x^2} \, dx\\
=\sum _{j=0}^{\infty } \frac{i
   \left(-\Gamma \left(0,-\frac{1}{2} i (1+2 j) \pi \right)+\Gamma \left(0,\frac{1}{2} i (1+2 j) \pi
   \right)\right)}{1+2 j}+\frac{1}{2} \pi  \log \left(\frac{\pi }{2}\right)
\end{multline}
\end{example}
\begin{example}
Equation 147(14). Use equation (\ref{eq:2.11}) and set $c=-1,b=2,a\to e^{a},t\to x,m\to m+1$ then take the first partial derivative with respect to $k$ and set $k=0$. Next form a second equation by replacing $a\to -a$ and add these two equations and simplify. Next form another equation by replacing $m\to s$ and take their difference.
\begin{multline}\label{eq:3.35}
\int_0^1 \frac{x^{-\frac{b}{a}} \left(-1+x^{\frac{2 b}{a}}\right) \log \left(q^2+\log ^2(x)\right)}{-1+x^2} \,
   dx\\
=\sum _{j=0}^{\infty } a \left(\frac{e^{-i \left(1-\frac{b}{a}+2 j\right) q} E_1\left(-i \left(1-\frac{b}{a}+2
   j\right) q\right)}{a-b+2 a j}+\frac{e^{i \left(1-\frac{b}{a}+2 j\right) q} E_1\left(i \left(1-\frac{b}{a}+2 j\right)
   q\right)}{a-b+2 a j}\right. \\ \left.
-\frac{e^{-\frac{i (a+b+2 a j) q}{a}} E_1\left(-\frac{i (a+b+2 a j) q}{a}\right)+e^{\frac{i
   (a+b+2 a j) q}{a}} E_1\left(\frac{i (a+b+2 a j) q}{a}\right)-\frac{2 b \log \left(q^2\right)}{a-b+2 a j}}{a+b+2 a
   j}\right)
\end{multline}
where $Re(q)>0,Re(a)>0,Re(b)>0$.
\end{example}
\begin{example}
Equation 147(15). Use equation (\ref{eq:3.35}) and replace $q\to a/2$ and simplify.
\begin{multline}
\int_0^1 \frac{\left(x^{b/a}-x^{-\frac{b}{a}}\right) \log \left(\frac{a^2 \pi ^2}{4}+\log ^2(x)\right)}{-1+x^2}
   \, dx\\
=\sum _{j=0}^{\infty } a \left(\frac{e^{-\frac{1}{2} i (a-b+2 a j) \pi } \left(E_1\left(-\frac{1}{2} i (a-b+2 a
   j) \pi \right)+e^{i (a-b+2 a j) \pi } E_1\left(\frac{1}{2} i (a-b+2 a j) \pi \right)\right)}{a-b+2 a
   j}\right. \\ \left.
-\frac{e^{-\frac{1}{2} i (a+b+2 a j) \pi } E_1\left(-\frac{1}{2} i (a+b+2 a j) \pi \right)}{a+b+2 a
   j}-\frac{e^{\frac{1}{2} i (a+b+2 a j) \pi } E_1\left(\frac{1}{2} i (a+b+2 a j) \pi \right)}{a+b+2 a
   j}\right)\\+\frac{1}{2} \pi  \log \left(\frac{a^2 \pi ^2}{4}\right) \tan \left(\frac{b \pi }{2 a}\right)
\end{multline}
where $Re(a)>0,Re(b)>0$.
\end{example}
\begin{example}
Equation 147(16). Use equation (\ref{eq:2.11}) and set $c=1,b=1,a\to e^{q},m\to m+1$ then take the first partial derivative with respect to $k$ and set $k=0$ Next form a second equation by replacing $q\to -q$ and add these two equations and simplify by replacing $q\to qi$.
\begin{multline}\label{eq:3.36}
\int_0^1 \frac{\log \left(q^2+\log ^2(x)\right)}{\sqrt{x} (1+x)} \, dx\\
=\sum _{j=0}^{\infty } \frac{2 (-1)^j
   e^{-i \left(\frac{1}{2}+j\right) q} \left(E_1\left(-i \left(\frac{1}{2}+j\right) q\right)+e^{i (1+2 j) q} E_1\left(i
   \left(\frac{1}{2}+j\right) q\right)+2 e^{i \left(\frac{1}{2}+j\right) q} \log (q)\right)}{1+2 j}\\
=2 \pi  \log
   \left(\frac{2 \Gamma \left(\frac{q+3 \pi }{4 \pi }\right)}{\Gamma \left(\frac{q+\pi }{4 \pi }\right)}\right)+\pi 
   \log (\pi )
\end{multline}
where $Re(q)>0$.
\end{example}
\begin{example}
Infinite product derivation using the exponential function of the right-hand sides of equations 147(16) and (\ref{eq:3.36}).
\begin{multline}
\prod _{j=0}^{\infty } \exp \left(\frac{(-1)^j e^{-i \left(\frac{1}{2}+j\right) q} \left(E_1\left(-i
   \left(\frac{1}{2}+j\right) q\right)+e^{i (1+2 j) q} E_1\left(i \left(\frac{1}{2}+j\right) q\right)+2 e^{i
   \left(\frac{1}{2}+j\right) q} \log (q)\right)}{\pi +2 j \pi }\right)\\
=\frac{2 \sqrt{\pi } \Gamma \left(\frac{1}{4}
   \left(3+\frac{q}{\pi }\right)\right)}{\Gamma \left(\frac{\pi +q}{4 \pi }\right)}
\end{multline}
where $Re(q)>0$.
\end{example}
\begin{example}
Equation 147(17). Generalized form of this equation. Errata. Use equation (\ref{eq:2.11}) and set $a\to e^{a},b=1,m\to m+1$ and take the first partial derivative with respect to $k$ and set $k=0$ and simplify. Next form a second equation by replacing $a\to -a$ and add these two equations. Next replace $a\to qi,t\to x$ and simplify.
\begin{multline}
\int_0^1 \frac{x^m \log \left(q^2+\log ^2(x)\right)}{1+c x} \, dx\\
=\sum
   _{j=0}^{\infty } \frac{(-1)^j c^j e^{-i (1+j+m) q} \left(\Gamma (0,-i (1+j+m)
   q)+e^{2 i (1+j+m) q} \Gamma (0,i (1+j+m) q)\right)}{1+j+m}\\
+\Phi (-c,1,1+m)
   \log \left(q^2\right)
\end{multline}
where $Re(q)>0$.
\end{example}
\subsection{Table 129 in Bierens de Haan}
We will derive a general form definite integral which can be used to evaluate all entries in this Table.
\begin{theorem}
For all $|Re(m)|<1,Re(c)>0,Re(b)>0$ then,
\begin{multline}\label{eq:3.39}
\int_0^1 \frac{x^m \log ^n(x)}{\left(1+c x^b\right) \left(q^2+\log ^2(x)\right)} \, dx\\
=\sum _{j=0}^{\infty }
   \frac{1}{2} (-1)^j c^j e^{-i (1+b j+m) q} (1+b j+m) \left((1+b j+m)^2 q^2\right)^{-n} \\
\left((-i q)^n (-i (1+b j+m)
   q)^n E_n(-i (1+b j+m) q)\right. \\ \left.
+e^{2 i (1+b j+m) q} (i q)^n (i (1+b j+m) q)^n E_n(i (1+b j+m) q)\right) \Gamma (n)
\end{multline}
\end{theorem}
\begin{proof}
Use equation (\ref{eq:2.11}) and set $t\to x,a\to e^{a},k=-1$. Next form a second equation by replacing $q\to -a$ and add the two equations. Next using this new equation we take the $n$-th partial derivative with respect to $m$ and replace $q\to qi,m\to m+1,n\to n-1$ and simplify using [DLMF,\href{https://dlmf.nist.gov/15.5.i}{15.5}].
\end{proof}
\subsection{Table 4.229 in Gradshteyn and I.M. Ryzhik}
\begin{example}
Equation (4.229.1). Use equation (\ref{eq:2.11}) and set $b=0,a=1m=1$, then take the first partial derivative with respect to $k$ and set $k=0$ and simplify using [Wolfram,\href{https://mathworld.wolfram.com/Euler-MascheroniConstant.html}{49}].
\begin{equation}
\int_0^1 \log \left(\log \left(\frac{1}{t}\right)\right) \, dx=-\gamma
\end{equation}
\end{example}
\begin{example}
Equation (4.229.3). Use equation (\ref{eq:2.11}) and set $b=0,a=m=1,t\to x$ then take the first partial derivative with respect to $k$ and set $k=-1/2$ and simplify using [Wolfram,\href{https://mathworld.wolfram.com/Euler-MascheroniConstant.html}{49}].
\begin{equation}
\int_0^1 \frac{\log \left(\log \left(\frac{1}{x}\right)\right)}{\sqrt{\log \left(\frac{1}{x}\right)}} \,
   dx=\sqrt{\pi } (-\gamma -\log (4))
\end{equation}
\end{example}
\begin{example}
Equation (4.229.11). Use equation (\ref{eq:2.11}) and set set $b=0,a=m=1,t\to x$ then take the first partial derivative with respect to $k$ and set $k=u-1$ and simplify using [Wolfram,\href{https://mathworld.wolfram.com/Euler-MascheroniConstant.html}{49}].
\begin{equation}
\int_0^1 \log ^{-1+u}\left(\frac{1}{x}\right) \log \left(\log \left(\frac{1}{x}\right)\right) \, dx=\Gamma (u)
   \psi ^{(0)}(u)
\end{equation}
where $|Re(u)|<1$.
\end{example}
\begin{example}
Equation (4.229.5(7)). Errata.  Use equation (\ref{eq:2.11}) and set $b=0,a\to e^{a},m=1$ then take the first partial derivative with respect to $k$ and set $k=0$ and simplify.
\begin{multline}\label{eq:3.44}
\int_0^1 \log (a+\log (x)) \, dx=i e^{-a} \pi -e^{-a} \text{Ei}(a)+\log (a)=\log (a)-\exp (-a)
   \text{Ei}(a)
\end{multline}
where $Im(a)\geq 0$.
\end{example}
\begin{example}
Equation (4.229.6). Errata. Use equation (\ref{eq:3.44}) and form a second equation by replacing $a\to -a$ and factor and evaluate the $\log(-1)$ on the left-hand side and simplify.
\begin{multline}
\int_0^1 \log (a-\log (x)) \, dx=\begin{cases}
			-i \pi +i e^a \pi -e^a \text{Ei}(-a)+\log (-a), & \text{ $Im(a)<0$}\\
            \log (a)-e^a \text{Ei}(-a), & \text{$Re(a)>0, Im(a)=0$}
		 \end{cases}
\end{multline}
\end{example}
\subsection{Table 4.213 in  Gradshteyn and I.M. Ryzhik}
The entries in this Table can be derived using special cases of equation (\ref{eq:3.45}) . These include (4.213.1); setting $m=0$, (4.213.2); setting $m=0,a\to ai$, (4.213.3) taking the first partial derivative with respect to $m$ then set $m\to 0$, (4.213.4(7)); taking the first partial derivative with respect to $m$ then set $m\to 0,a\to ai$, (4.213.5); taking the first partial derivative with respect to $a$ then setting $m=0$, (4.213.6); taking the first partial derivative with respect to $a$ then set $m\to 0,a\to ai$, (4.213.7); taking the first partial derivatives with respect to $a$ and $m$ then setting $m=0$, and (4.213.8); taking the first partial derivatives with respect to $a$ and $m$ then setting $m=0,a\to ai$
\begin{theorem}
For all $|Re(m)|<1,Re(a)>0$ then,
\begin{multline}\label{eq:3.45}
\int_0^1 \frac{x^m}{a^2+\log ^2(x)} \, dx=\frac{i e^{-i a (1+m)} \left(-\Gamma (0,-i a (1+m))+e^{2 i a (1+m)}
   \Gamma (0,i a (1+m))\right)}{2 a}
\end{multline}
\end{theorem}
\begin{proof}
Use equation (\ref{eq:2.11}) and set $b=0,t\to x,a\to e^a,m\to m+1$, then form a second equation by replacing $a\to -a$ and take their difference and simplify with $a\to ai$.
\end{proof}
\subsection{Table 4.214 in Gradshteyn and I.M. Ryzhik}
In this derivation the definite integral is evaluated using the Cauchy principal value, where there exists a singularity at $x=ai$. 
\begin{theorem}
For all $Re(a)>0,|Re(m)|<1$ then,
\begin{multline}\label{eq:3.46}
\int_0^1 \frac{x^{-1+m}}{a^4-\log ^4(x)} \, dx\\
=-\frac{e^{-a m} \Gamma (0,-a m)+i e^{-i a m} \left(\Gamma (0,-i a
   m)-e^{2 i a m} \Gamma (0,i a m)\right)-e^{a m} \Gamma (0,a m)}{4 a^3}
\end{multline}
\end{theorem}
\begin{proof}
Use equation (\ref{eq:2.11}) and set $t\to x,a\to e^a, k=-1,b=0$, then form a second equation by replacing $a\to -a$ and take their difference which forms an integral with a quadratic form in the denominator. Next using this equation we form another equation by replacing $a\to ai$ and taking their difference and simplify.
\end{proof}
\begin{example}
Equation  (4.214.1). Use equation (\ref{eq:3.46}) and set $m=1$.
\end{example}
\begin{example}
Equation  (4.214.2). Use equation (\ref{eq:3.46}) and take the first partial derivative with respect to $m$ and set $m=1$.
\end{example}
\begin{example}
Equation  (4.214.3). Use equation (\ref{eq:3.46}) and take the second partial derivative with respect to $m$ and set $m=1$.
\end{example}
\begin{example}
Equation  (4.214.4). Use equation (\ref{eq:3.46}) and take the third partial derivative with respect to $m$ and set $m=1$.
\end{example}
\subsection{Miscellaneous entries and example}
\begin{example}
An example in terms of Catalan's constant $C$. Use equation (\ref{eq:2.11}) and set $a=1$ and simplify the series in terms of the Hurwitz-Lerch zeta function. Next set $c\to z,k\to s-1$, then set $z=1,s=2,m\to b/2$ and simplify using [Wolfram,\href{https://mathworld.wolfram.com/LerchTranscendent.html}{7}].
\begin{equation}\label{eq:3.47}
\int_0^1 \frac{t^{-1+s} \log \left(\frac{1}{t}\right)}{1+t^{2 s}} \, dt=\frac{C}{s^2}
\end{equation}
\end{example}
\begin{example}
Use equation (\ref{eq:3.47}) and take the infinite sum of both sides $s\in[1,\infty)$ and simplify in terms of the theta function using [Wolfram,\href{https://mathworld.wolfram.com/JacobiThetaFunctions.html}{4}].
\begin{equation}
\int_0^1 \frac{\left(-1+\vartheta _3(0,t){}^2\right) \log \left(\frac{1}{t}\right)}{4 t} \, dt=\frac{C \pi ^2}{6}
\end{equation}
\end{example}
\begin{example}
Use equation (\ref{eq:2.11}) and take the $n$-th derivative with respect to $c$ and simplify the gamma function. This formula can be used to derive entries (4.272.1-9,11-14,17,19) in Gradshteyn and Ryzhik \cite{grad} using parameter substitutions and algebraic methods.
\begin{multline}
\int_0^1 \frac{t^{-1+m} (-\log (a t))^k}{\left(1+c t^b\right)^{n+1}} \, dt=\sum _{j=0}^{\infty } \frac{a^{-m+b (-j+n)}
   (-c)^{j-n} j^{(n)} \Gamma (1+k,-((m+b (j-n)) \log (a)))}{n! (m+b (j-n))^{k+1}}
\end{multline}
where $Re(m)>0$.
\end{example}
\begin{example}
Equation (4.285) in \cite{grad}. in this example we extend the given definite integral originally published by Bierens de Haan to the complex plane for the parameter $q$. Use equation (\ref{eq:2.11}) and set $b=0$ and simplify the series. Next set $k=-1,a\to e^a,t\to x,m\to p$. Then take the $n$-th partial derivative with respect to $a$ and simplify the gamma function.
\begin{multline}
\int_0^1 \frac{x^{-1+p}}{(q+\log (x))^n} \, dx=\begin{cases}
			-(-1)^{-n} e^{-p q} (-p)^n \left(\frac{1}{p q}\right)^n q E_n(-p
   q), & \text{if $q\in\mathbb{C}$}\\
            \frac{p^{n-1} \exp (-p q) \text{Ei}(p q)}{(n-1)!}-\frac{\sum _{k=1}^{n-1} (n-k-1)! (p q)^{k-1}}{(n-1)!
   q^{n-1}}, & \text{$q<0$}
		 \end{cases}
\end{multline}
where $|Re(p)|<1,Re(q)<0$
\end{example}
\begin{example}
Equation (2.6.4.11) in \cite{prud1}. Use equation (\ref{eq:2.11}) and set $c=1,k\to 2k,a=1$ and take the $p$-th derivative with respect to $c$, then replace $m\to m-bp$ and simplify. Next replace $p\to p-1$ and simplify. Next replace $m\to up/2,t\to x,b\to u,k\to n,j\to k$ and simplify.
\begin{multline}
\int_0^1 \frac{x^{-1+\frac{p u}{2}} \log ^{2 n}(x)}{\left(1+x^u\right)^p} \, dx=-\frac{2^{1+2 n} u^{-1-2 n}
   (-1)^{-2 n-p} \Gamma (1+2 n)}{\Gamma
   (p)}\sum _{k=0}^{\infty } \frac{(-1)^k (2+k-p)_{-1+p}}{(2+2 k-p)^{2 n+1}}
\end{multline}
where $|Re(m)|<1,Re(p)>0,Re(u)>0$.
\end{example}
\begin{example}
Use equation (\ref{eq:2.11}) and set $m\to m+1,a=1c=-1,b=1$ then replace $m\to m+p$. Next multiply both sides by $(-1)^{a-l}$ and replace $p\to p(a-l)$ then multiply both sides by $\binom{a}{l}$ then take the sum of both sides over $l\in[0,a]$ and simplify. Next form another equation by replacing $m\to q$ and take the difference. Next we apply l'Hopital's rule as $k\to -1$ and simplify.
\begin{equation}
\int_0^1 \frac{\left(1-x^p\right)^a \left(-x^m+x^q\right)}{(-1+x) x \log (x)} \, dx=\sum _{j=0}^{\infty } \sum
   _{l=0}^a (-1)^{a-l} \binom{a}{l} \log \left(\frac{j+m+a p-l p}{j+a p-l p+q}\right)
\end{equation}
where $Re(m)>0,Re(q)>0$.
\end{example}
\begin{example}
Use equation (\ref{eq:2.11}) and set $a=1,c=1$ and simplify in terms of the Hurwitz-Lerch zeta function. Next form a second equation by replacing $m\to s$ and take their difference. Next apply l'Hopital's rule as $k\to -1$ and $m\to m+1,s\to s+1$ using [Wolfram,\href{https://mathworld.wolfram.com/HurwitzZetaFunction.html}{17}].
\begin{equation}
\int_0^1 \frac{t^m-t^s}{t \left(1+t^b\right) \log (t)} \, dt=\log \left(\frac{\Gamma \left(\frac{b+m}{2
   b}\right) \Gamma \left(\frac{s}{2 b}\right)}{\Gamma \left(\frac{m}{2 b}\right) \Gamma \left(\frac{b+s}{2
   b}\right)}\right)
\end{equation}
where $Re(m)>0,Re(s)>0$.
\end{example}
\begin{example}
Use equation (\ref{eq:2.11}) and set $k=-1/2,a\to e^a$ and take the indefinite integral with respect to $c$ and replace $t\to x,m\to m+b+1$ and simplify. Next take the $n$-th derivative with respect to $a$ and simplify.
\begin{multline}\label{eq:3.54}
\int_0^1 \frac{x^m \log \left(1+c x^b\right)}{(a+\log (x))^{n+\frac{1}{2}}} \, dx\\
=-\sum _{j=0}^{\infty }
   \frac{(-1)^{j+n} a^{\frac{1}{2}-n} c^{1+j} e^{-a (1+b+b j+m)} E_{\frac{1}{2}+n}(-a (1+b+b j+m)) \sec (n \pi
   )}{1+j}
\end{multline}
where $Re(a)>0,|Re(c)|<1,|Re(m)|<1$ and singularity at $x=\pm ai$.
\end{example}
\begin{example}
Use equation (\ref{eq:3.54}) and set $a=1,b=1,n=0,c=1/2,m=0$ and simplify.
\begin{multline}
\int_0^1 \frac{\log (x+2)}{\sqrt{1+\log (x)}} \, dx=\sum _{j=0}^{\infty }
   \frac{\left(-\frac{1}{2}\right)^{1+j} e^{-2-j} E_{\frac{1}{2}}(-2-j)}{1+j}+\frac{\sqrt{\pi } (i+\text{erfi}(1))
   \log (2)}{e}
\end{multline}
\end{example}
\begin{example}
Use equation (\ref{eq:2.11}) and set $a=1$ and simplify the series in terms of the Hurwitz-Lerch zeta function using [DLMF,\href{https://dlmf.nist.gov/25.14.i}{25.14.1}] and replace $c\to z, k\to s,t\to x$.
\begin{equation}\label{eq:3.57}
\int_0^1 \frac{x^{-1+m} \log ^s\left(\frac{1}{x}\right)}{1+x^b z} \, dx=b^{-1-s} \Gamma (1+s) \Phi
   \left(-z,1+s,\frac{m}{b}\right)
\end{equation}
where $|Re(m)|<1,Re(b)>0,Re(uz>0$.
\end{example}
\begin{example}
Fermi-Dirac Distribution. Generalized form for equations (2.6.4.4-5) in \cite{prud1}. This example could also be used to expand and derive the definite integrals listed in equations (2.6.4.6-16),  (2.6.5.6-18), (2.6.6.1), and (2.6.6.4-24) in \cite{prud1}. Use equation (\ref{eq:3.57}) and set $m\to b$ and simplify using [DLMF,\href{https://dlmf.nist.gov/25.14.E2}{25.14.3}].
\begin{equation}
\int_0^1 \frac{x^{-1+b} \log ^s\left(\frac{1}{x}\right)}{1+e^z x^b} \, dx=-b^{-1-s} e^{-z} \Gamma (1+s)
   \text{Li}_{1+s}\left(-e^z\right)
\end{equation}
where $|Re(b)|<1,Re(z)>0$.
\end{example}
\begin{example}
Use equation (\ref{eq:3.54}) and set $b=0$ and simplify using [DLMF,\href{https://dlmf.nist.gov/8.4.E13}{8.4.13}]  and [Wolfram,\href{http://functions.wolfram.com/06.34.03.0006.01}{01}].
\begin{multline}
\int_0^1 \frac{x^m}{(a+\log (x))^{n+\frac{1}{2}}} \, dx\\
=\frac{\left((-1)^{n+1} a^{\frac{1}{2}-n} e^{-a
   (1+m)}\right) E_{n+\frac{1}{2}}(-a (1+m))}{\cos (n \pi )}\\
=(-1)^{1+n} a^{\frac{1}{2}-n} e^{-a (1+m)}
   \left(\frac{(-1)^n (-a (1+m))^{-\frac{1}{2}+n} \sqrt{\pi } \text{erfc}\left(\sqrt{-a
   (1+m)}\right)}{\left(\frac{1}{2}\right)_n}\right. \\ \left.
-\sum _{k=0}^{n-1} \frac{e^{a (1+m)} (-a
   (1+m))^k}{\left(\frac{1}{2}-n\right)_{1+k}}\right) \sec (n \pi )
\end{multline}
where $|Re(m)|<1,Re(a)>0,n\in\mathbb{Z_{+}}$ and a singularity exists at $x=ai$.
\end{example}
\begin{example}
Use equation (\ref{eq:2.11}) and set $k=-1,a\to e^a,t\to x,m\to s$ then take the $m$-th derivative with respect to $c$ and $n$-th derivative with respect to $a$ and simplify the gamma function.
\begin{multline}
\int_0^1 \frac{x^{-1+b (-1+m)+s}}{\left(1+c x^b\right)^m (a+\log (x))^n} \, dx\\
=\sum _{j=0}^{\infty }
   \frac{(-1)^{j-m} a^{1-n} c^{1+j-m} e^{-a (b j+s)} E_n(-a (b j+s)) \Gamma (1+j)}{\Gamma (2+j-m) \Gamma (m)}
\end{multline}
where $|Re(s)|<1,Re(a)>0,m,n\in\mathbb{Z_{+}}$ and a singularity exists at $x=ai$.
\end{example}
\begin{example}
Derivation of Equation (3.2.4.109) in \cite{brychkov}. Use equation (\ref{eq:3.39}) and replace $m\to m-1$ and take the $n$-th derivative with respect to $c$ and simplify the gamma function. Next replace $c\to 1/z$ and simplify. Next we split the infinite integral in equation (3.2.4.109) in \cite{brychkov} over $[0,1]$ and $[1,\infty)$. Then we substitute our current equation and simplify.
\begin{multline}
\sum _{j=0}^{\infty } i e^{-2 i n \pi } z^{-1-j-n} (-1-j)^{(n)}
   \left(e^{i n \pi } \Gamma (0,-i (1+j) \pi )-e^{i (2 j+n) \pi } \Gamma (0,i
   (1+j) \pi )\right. \\ \left.
+z^{1+2 j+n} \left(-\Gamma (0,-i (j+n) \pi )+e^{2 i (j+n) \pi }
   \Gamma (0,i (j+n) \pi )\right)\right)\\
=2 (-1)^n \pi  \frac{\partial
   ^n}{\partial z^n}\frac{1}{\log (z)}-2 \pi  (-1+z)^{-1-n} n!
\end{multline}
where $|\arg(z)|<\pi,z\neq 1$.
\end{example}
\begin{example}
Derivation of equation (4.1.5.126) in \cite{brychkov}. Use equation  (\ref{eq:2.11}) and set $k=a=b=m=1$ and take the indefinite integral with respect to $c$. Next form a second equation by replacing $c\to -c$ and take their difference and simplify using [DLMF,\href{https://dlmf.nist.gov/25.12.i}{25.12.1}] and replacing $c\to 1/a, t\to x$.
\begin{equation}
\int_0^1 \frac{\log (x) \log \left(\frac{a+x}{a-x}\right)}{x} \,
   dx=\text{Li}_3\left(-\frac{1}{a}\right)-\text{Li}_3\left(\frac{1}{a}\right)
\end{equation}
where $Re(a)>1$.
\end{example}
\begin{example}
The hypergeometric function. Use equation  (\ref{eq:2.11}) and set $a=1$ and simplify in terms of the Hurwitz-Lerch zeta function using [DLMF,\href{https://dlmf.nist.gov/25.14.i}{25.14.1}]. Next replace $c\to c^b$ and multiply both sides by $c^{m-1}$ and take the indefinite integral with respect to $c$. Next replace $t\to x, c\to z, k\to s$
\begin{multline}\label{eq:3.62}
\int_0^1 x^{-1+m} \, _2F_1\left(1,\frac{m}{b};1+\frac{m}{b};-x^b z^b\right) \log ^s\left(\frac{1}{x}\right) \,
   dx=b^{-2-s} m \Gamma (1+s) \Phi \left(-z^b,2+s,\frac{m}{b}\right)
\end{multline}
where $|Re(s)|<1$.
\end{example}
\begin{example}
Use equation  (\ref{eq:3.62}) take the first partial derivative with respect to $s$ and set $s=0,z=1,b=1,m=1/4$ and simplify using [Wolfram,\href{https://mathworld.wolfram.com/HurwitzZetaFunction.html}{20}] and [Wolfram,\href{https://mathworld.wolfram.com/Euler-MascheroniConstant.html}{1}].
\begin{multline}
\int_0^1 \frac{\, _2F_1\left(\frac{1}{8},1;\frac{9}{8};-x^2\right) \log \left(\log
   \left(\frac{1}{x}\right)\right)}{x^{3/4}} \, dx\\
=\frac{1}{64} \left(-\left((\gamma +\log (2)) \left(\psi
   ^{(1)}\left(\frac{1}{16}\right)-\psi ^{(1)}\left(\frac{9}{16}\right)\right)\right)+4
   \Phi'\left(-1,2,\frac{1}{8}\right)\right)
\end{multline}
\end{example}
\begin{example}
Use equation (\ref{eq:3.62}) and replace $b\to m$ and simplify in terms of the polylogarithm function using [DLMF,\href{https://dlmf.nist.gov/25.14.E2}{25.14.3}]. Next replace the polylogarithm function using [DLMF,\href{https://dlmf.nist.gov/25.12.E13}{25.12.13}] and simplify the gamma function.
\begin{multline}\label{eq:3.64}
\int_0^1 \frac{z^{-m} \log ^s\left(\frac{1}{x}\right) \log \left(1+x^m z^m\right)}{x} \, dx\\
=\frac{i 2^{1+s}
   e^{-\frac{1}{2} i \pi  s} m^{-1-s} \pi ^{2+s} z^{-m} \csc (\pi  s) \left(\zeta \left(-1-s,\frac{\pi -i m \log (z)}{2
   \pi }\right)-e^{i \pi  s} \zeta \left(-1-s,\frac{\pi +i m \log (z)}{2 \pi }\right)\right)}{1+s}
\end{multline}
where $|Re(m)|<1,Re(z)>0$.
\end{example}
\begin{example}
Use equation (\ref{eq:3.64}) take the first partial derivative with respect to $s$ and apply l'Hopital's rule as $s\to 0$ and simplify.
\begin{multline}\label{eq:3.65}
\int_0^1 \frac{z^{-m} \log \left(1+x^m z^m\right) \log \left(\log \left(\frac{1}{x}\right)\right)}{x} \,
   dx\\
=\frac{z^{-m} }{12
   m}\left((-1+\log (2)-\log (m)+\log (\pi )) \left(\pi ^2+3 m^2 \log ^2(z)\right)\right. \\ \left.
-12 \pi  \left((\pi -2 i
   \log (m)+2 i (-1+\log (2)+\log (\pi ))) \zeta'\left(-1,\frac{\pi -i m \log (z)}{2 \pi }\right)\right.\right. \\ \left.\left.
+(2
   i+\pi +2 i \log (m)-2 i \log (2 \pi )) \zeta'\left(-1,\frac{\pi +i m \log (z)}{2 \pi }\right)\right.\right. \\ \left.\left.
-i
   \left(\zeta''\left(-1,\frac{\pi -i m \log (z)}{2 \pi
   }\right)-\zeta''\left(-1,\frac{\pi +i m \log (z)}{2 \pi }\right)\right)\right)\right)
\end{multline}
where $Re(m)>0$.
\end{example}
\begin{example}
Use equation (\ref{eq:3.65}) set $z=1$ and simplify using [Wolfram,\href{https://mathworld.wolfram.com/LerchTranscendent.html}{8}].
\begin{equation}
\int_0^1 \frac{\log \left(1+x^m\right) \log \left(\log \left(\frac{1}{x}\right)\right)}{x} \, dx=\frac{\pi ^2
  }{12 m} \log \left(\frac{4 \pi }{A^{12} m}\right)
\end{equation}
where $Re(m)>0$.
\end{example}
\begin{example}
Use equation (\ref{eq:3.65}) and set $z=-1,m=1$ and simplify using [Wolfram,\href{https://mathworld.wolfram.com/HurwitzZetaFunction.html}{4}] and [Wolfram,\href{https://mathworld.wolfram.com/Glaisher-KinkelinConstant.html}{3}]. Note that in Figure 1, the singularity lies outside the domain of evaluation.
\begin{multline}
\int_0^1 \frac{\log (1-x) \log \left(\log \left(\frac{1}{x}\right)\right)}{x} \, dx=\pi ^2 \log
   \left(\frac{A^2}{\sqrt[6]{2 \pi }}\right)+i \pi  \zeta ''(-1)-i \pi  \zeta''(-1,0)
\end{multline}
\end{example}
\begin{figure}[h]
\caption{Plot of $f(x)=Re\left(\frac{\log (1-x) \log \left(\log \left(\frac{1}{x}\right)\right)}{x}\right)$}
\centering
\includegraphics[width=0.6\textwidth]{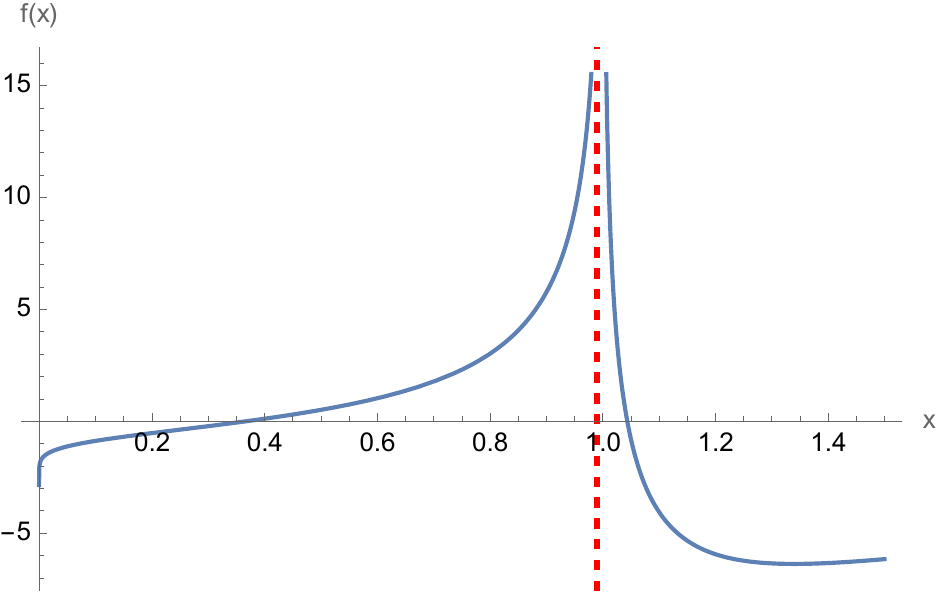}
\end{figure}
\begin{example}
Use equation (\ref{eq:3.64}) and replace the Hurwitz zeta function using [DLMF,\href{https://dlmf.nist.gov/25.12.E13}{25.12.13}] and simplify the gamma function.
\begin{multline}
\int_0^1 \frac{z^{-m} \log ^s\left(\frac{1}{x}\right) \log \left(1+x^m z^m\right)}{x} \, dx\\
=-\frac{i m^{-1-s}
   z^{-m} \Gamma (2+s)}{2 (1+s)} \left((-i+\cot (\pi  s)) \text{Li}_{2+s}\left(\frac{\pi -i m \log (z)}{2 \pi }\right)\right. \\ \left.
-(i+\cot
   (\pi  s)) \text{Li}_{2+s}\left(\frac{2 \pi }{\pi +i m \log (z)}\right)\right. \\ \left.
+\csc (\pi  s)
   \left(\text{Li}_{2+s}\left(\frac{2 \pi }{\pi -i m \log (z)}\right)-\text{Li}_{2+s}\left(\frac{\pi +i m \log (z)}{2
   \pi }\right)\right)\right)
\end{multline}
where $Re(m)>0$.
\end{example}
%
%
\begin{example}
Use equation (\ref{eq:2.11}) and take the indefinite integral with respect to $c$ and set $k=-1,t\to x,a\to e^a,m\to m+b+1$. Next take the first partial derivative with respect to $m$ and set $m=0,a=b=c=1$ and simplify. Note the singularity lies inside the domain of evaluation.
\begin{equation}
\int_0^1 \frac{\log (x) \log (1+x)}{1+\log (x)} \, dx=\log (4)-1+\frac{1}{e^{2}}\sum _{j=0}^{\infty } \frac{ \Gamma (0,-2-j)}{1+j}\left(-\frac{1}{e}\right)^j
\end{equation}
where there exists a singularity at $x=-i$.
\end{example}
\begin{figure}[H]
\caption{Plot of $f(x)=Re\left(\frac{\log (x) \log (x+1)}{\log (x)+1}\right)$}
\centering
\includegraphics[width=0.6\textwidth]{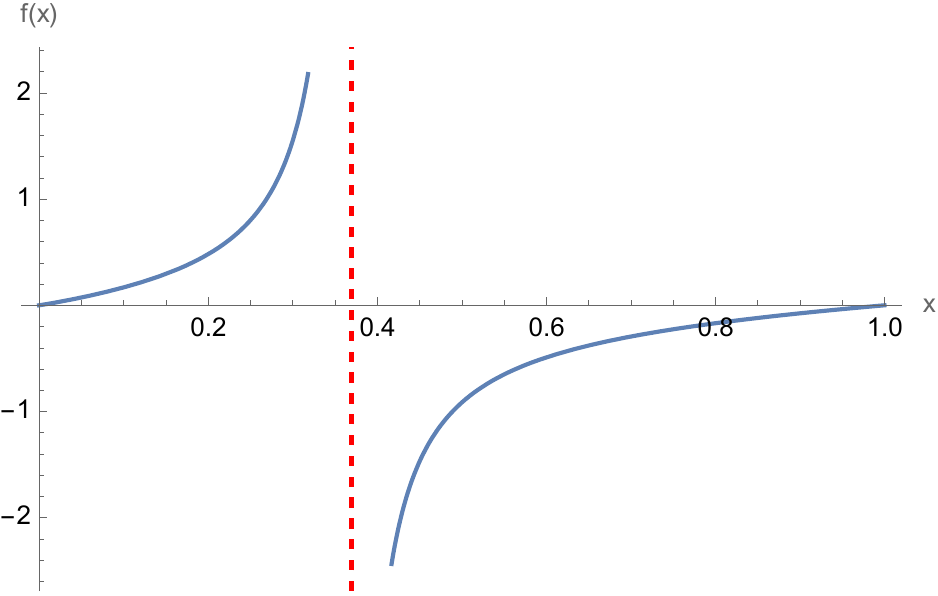}
\end{figure}
\begin{example}
Use equation (\ref{eq:3.62}) and replace $z\to z^{1/b},x\to x^{1/b}$. Next replace $z\to -z,m\to bu,s\to s-2,u\to a$ and simplify.
\begin{multline}\label{eq:3.70}
\int_0^1 x^{-1+a} \, _2F_1(1,a;1+a;x z) \log ^{s-2}(x) \, dx=-(-1)^{1+s} a \Gamma (s-1) \Phi (z,s,a)
\end{multline}
where $Re(s)>1,|Re(a)|<1$.
\end{example}
\begin{example}
Use equation (\ref{eq:3.70}) and set $z=1,s\to n+1$ and simplify using [Wolfram,\href{https://mathworld.wolfram.com/PolygammaFunction.html}{4}] and [Wolfram,\href{http://functions.wolfram.com/10.06.03.0020.02}{02}].
\begin{equation}
\int_0^1 \frac{B_x(a,0) \log ^{-1+n}(x)}{x} \, dx=\frac{\psi ^{(n)}(a)}{n}
\end{equation}
where $Re(a)>1,n\in\mathbb{Z_{+}}$.
\end{example}
\begin{example}
Use equation (\ref{eq:2.11}) and set $a=1,t\to x,m\to m+p-q$ and multiply both sides by $\binom{p}{q} d^{p-q}$ and take the finite sum over $q\in[0,p]$ and simplify.
\begin{multline}
\int_0^1 \frac{x^{m-1} (1+d x)^p \log ^k\left(\frac{1}{x}\right)}{1+c x^b} \, dx=d^p b^{-1-k} \Gamma (1+k)
   \sum _{q=0}^p d^{-q} \Phi \left(-c,1+k,\frac{m+p-q}{b}\right) \binom{p}{q}
\end{multline}
where $|Re(m)|<1,Re(b)>0,Re(c)>0,Re(d)>0,p\in\mathbb{Z_{+}}$.
\end{example}
\section{Tables 4.293 - 4.318 in Gradshteyn and Ryzhik}
The derivation of these entries is achieved by algebraic manipulations and parameter substitutions of equation (\ref{eq:4.1}). The entries are (4.293.1-10), (4.292.1), (4.292.12-13), (4.295.2,3,9,11,13,37(12),39,40), (4.296.1-3,7), (4.298.6,7,21,22), (4.299.1,4), (4.311.1(11)), (4.313.1,4,5), (4.314.1(11)), (4.315), and (4.318)  in the case of infinite integrals in Gradshteyn and Ryzhik, we can split these integrals over $[0,1]$ and use the below formula. Where $|Re(c)| \leq 1, Re(m)>-1/6$. 
\begin{theorem}
For all $|Re(c)|\leq 1,Re(b)>0,Re(m)>-1/6$ then,
\begin{multline}\label{eq:4.1}
\int_0^1 x^{-1+m} (-\log (a x))^k \log \left(1+c x^b\right) \, dx\\
=\sum _{j=0}^{\infty } \frac{a^{-b (1+j)-m} (-c)^j c (b+b j+m)^{-1-k} \Gamma (1+k,-((b+b j+m) \log
   (a)))}{1+j}
\end{multline}
\end{theorem}
\begin{proof}
Use equation (\ref{eq:2.11}) and take the indefinite integral with respect to $c$ and simplify.
\end{proof}
\subsection{Derivation of equation (4.382.5) in Gradshteyn and Ryzhik, its generalized form and evaluations}
The contour integral representation for equation (4.382.5) is given by;
\begin{multline}\label{eq:4.2}
\frac{1}{2\pi i}\int_{0}^{\infty}\int_{C}a^w w^{-k-1} \sin (x (m+w)) \log \left(\frac{\gamma ^2+(\beta +x)^2}{\gamma ^2+(x-\beta )^2}\right)dwdx\\
=\frac{1}{2\pi i}\int_{C}\frac{2 \pi 
   a^w w^{-k-1} e^{-(\gamma  (m+w))} \sin (\beta  (m+w))}{m+w}dwdx
\end{multline}
\subsubsection{Right-hand side contour integral}
Using a generalization of Cauchy's integral formula \ref{intro:cauchy}, we replace $y\to y+xi+\log(a)$ and multiply both sides by $e^{mxi}$. Then form a second equation by replacing $a\to -a$ and take their difference to get;
 \begin{multline}\label{eq:4.3}
\frac{i e^{-i m x} \left((\log (a)-i x+y)^k-e^{2 i m x} (\log (a)+i x+y)^k\right)}{2 k!}\\
=\frac{1}{2\pi i}\int_{C}a^w w^{-k-1} e^{w y}
   \sin (x (m+w))dw
\end{multline}
Next we take the indefinite integral with respect to $y$ to get;
\begin{multline}\label{eq:4.4}
-\frac{i (-1)^{k+1} a^{-m} m^{-k-1} (\Gamma (k+1,-m (-i x+y+\log (a)))-\Gamma (k+1,-m (i x+y+\log (a))))}{2 k!}\\
=\frac{1}{2\pi i}\int_{C}\frac{a^w w^{-k-1} e^{y (m+w)} \sin (x (m+w))}{m+w}dw
\end{multline}
from equation equation (2.325.6(12i)) in \cite{grad}. Finally we replace $y\to -\gamma, x\to \beta$, multiply both sides by $2\pi$ and simplify to get;
\begin{multline}\label{eq:4.5}
\frac{i \pi  (-1)^{k+1} a^{-m} m^{-k-1} (\Gamma (k+1,m (-i \beta +\gamma -\log (a)))-\Gamma (k+1,m (i \beta
   +\gamma -\log (a))))}{k!}\\
=\frac{1}{2\pi i}\int_{C}\frac{2 \pi  a^w w^{-k-1} e^{-(\gamma  (m+w))} \sin (\beta  (m+w))}{m+w}dw
\end{multline}
\subsubsection{Left-hand side contour integral}
Using a equation (\ref{eq:4.3}) we set $y=0$ multiply both sides by $\log \left(\frac{\gamma ^2+(\beta +x)^2}{\gamma ^2+(x-\beta )^2}\right)$ and simplify to get;
\begin{multline}\label{eq:4.6}
\int_{0}^{\infty}\frac{i e^{-i m x} \log \left(\frac{\gamma ^2+(\beta +x)^2}{\gamma ^2+(x-\beta )^2}\right) \left((\log (a)-i
   x)^k-e^{2 i m x} (\log (a)+i x)^k\right)}{2 k!}dx\\
=\frac{1}{2\pi i}\int_{0}^{\infty}\int_{C}a^w w^{-k-1} \sin (x (m+w)) \log \left(\frac{\gamma ^2+(\beta
   +x)^2}{\gamma ^2+(x-\beta )^2}\right)dwdx\\
   =\frac{1}{2\pi i}\int_{C}\int_{0}^{\infty}a^w w^{-k-1} \sin (x (m+w)) \log \left(\frac{\gamma ^2+(\beta
   +x)^2}{\gamma ^2+(x-\beta )^2}\right)dxdw
\end{multline}
We are able to switch the order of integration over $x$ and $w$ using Fubini's theorem for multiple integrals see page 178 in \cite{gelca}, since the integrand is of bounded measure over the space $\mathbb{C} \times [0,\infty)$.
\begin{theorem}
For all $Im(m)=0,Re(a)>0,Re(k)<0$ then,
\begin{multline}\label{eq:4.7}
\int_0^{\infty } e^{-i m x} \left((a-i x)^k-e^{2 i m x} (a+i x)^k\right) \log \left(\frac{(x+\beta
   )^2+\gamma ^2}{(x-\beta )^2+\gamma ^2}\right) \, dx\\
=2 (-1)^{-k} e^{-a m} m^{-1-k} \pi  \left(\Gamma (1+k,m (-a+i
   \beta +\gamma ))-(-1)^{2 k} \Gamma (1+k,m (-a-i \beta +\gamma ))\right)
\end{multline}
\end{theorem}
\begin{proof}
Since the right-hand sides of equations (\ref{eq:4.5}) and (\ref{eq:4.6}) are equal relative to equation (\ref{eq:4.2}) we can equate the left-hand sides and simplify the gamma function to yield the stated result.
\end{proof}
\begin{example}
Derivation of equation (4.382.5) in \cite{grad}. Use equation (\ref{eq:4.7}) and set $k=0$ and simplify.
\begin{equation}
\int_0^{\infty } \log \left(\frac{(x+\beta )^2+\gamma ^2}{(x-\beta )^2+\gamma ^2}\right) \sin (m x) \,
   dx=\frac{2 e^{-m \gamma } \pi  \sin (m \beta )}{m}
\end{equation}
where $Re(\gamma) >0,|Im(\beta) < Re(\gamma),Re(b)>0$.
\end{example}
\begin{example}
In this example we use (\ref{eq:4.7})  and replace $a\to e^a,\gamma=0,k=-1$ and simplify. Next take the $n$-th partial derivative with respect to $a$ and simplify.
\begin{multline}
\int_0^{\infty } \log \left(\left(\frac{x+\beta }{x-\beta }\right)^2\right) \left(\frac{e^{i m x}}{(-a-i
   x)^{n+1}}-\frac{e^{-i m x}}{(-a+i x)^{n+1}}\right) \, dx\\
=\frac{2 \pi  (-1)^n }{e^{a m}}\left(\frac{E_{1+n}(-m
   (a+i \beta ))}{(a+i \beta )^n}-\frac{E_{1+n}(-m (a-i \beta ))}{(a-i \beta )^n}\right)
\end{multline}
where $Im(m)=0,Re(a)>0$.
\end{example}
\section{Definite integrals involving the product of logarithm and a polynomial}
In this example we look at a definite integral with nicely convergent properties. The main theorem derived in this section will be used to derive entries in the book of Gradshteyn and Ryzhik along with some other definite integrals which canbe added to current literature. The contour integral representation for the definite integral to be derived is given by;
\begin{multline}\label{eq:5.1}
\frac{1}{2\pi i}\int_{0}^{1}\int_{C}a^w w^{-k-1} (b x+1)^n x^{m+w}dxdw=\frac{1}{2\pi i}\sum_{j=1}^{n+1}\int_{C}\frac{a^w b^{j-1} w^{-k-1} \binom{n}{j-1}}{j+m+w}dw
\end{multline}
where $Re(m)>0,n>0$.
\subsection{The left-hand contour integral representation}
Using a generalization of Cauchy's integral formula \ref{intro:cauchy} we replace $y\to \log(ax)$ and multiply both sides by $x^m (1+bx)^n$ and simplify to get;
\begin{multline}\label{eq:5.2}
\int_{0}^{1}\frac{x^m (b x+1)^n \log ^k(a x)}{k!}dx=\frac{1}{2\pi i}\int_{0}^{1}\int_{C}w^{-k-1} x^m (a x)^w (b x+1)^ndwdx\\
=\frac{1}{2\pi i}\int_{C}\int_{0}^{1}w^{-k-1} x^m (a x)^w (b x+1)^ndwdx
\end{multline}
We are able to switch the order of integration over $x$ and $w$ using Fubini's theorem for multiple integrals see page 178 in \cite{gelca}, since the integrand is of bounded measure over the space $\mathbb{C} \times [0,\infty)$.
\subsection{The right-hand contour integral representation}
Using a generalization of Cauchy's integral formula \ref{intro:cauchy} we replace $y\to x+\log(a)$ and multiply both sides by $e^{mx}$ then take the infinite integral over $x\in[0,\infty)$ and simplify using [DLMF,\href{https://dlmf.nist.gov/8.6.E5}{8.6.5}]. Next we take the finite sum of both sides over $j\in[1,n+1]$ to get;
\begin{multline}\label{eq:5.3}
-\sum_{j=1}^{n+1}\frac{b^{j-1} a^{-j-m} \binom{n}{j-1} (-j-m)^{-k-1} \Gamma (k+1,-((j+m) \log (a)))}{k!}\\
=\frac{1}{2\pi i}\sum_{j=1}^{n+1}\int_{C}\frac{a^w b^{j-1}
   w^{-k-1} \binom{n}{j-1}}{j+m+w}dw\\
   =\frac{1}{2\pi i}\int_{C}\sum_{j=1}^{n+1}\frac{a^w b^{j-1}
   w^{-k-1} \binom{n}{j-1}}{j+m+w}dw
\end{multline}
We are able to switch the order of integration and summation over $w$ using Tonellii's theorem for  integrals and sums see page 177 in \cite{gelca}, since the summand is of bounded measure over the space $\mathbb{C} \times [0,\infty)$
\begin{theorem}
For all $Re(m)>0,b\in\mathbb{C},a\in\mathbb{C},k\in\mathbb{C},n>0$ then,
\begin{multline}\label{eq:5.4}
\int_0^1 x^m (1+b x)^n \log ^k\left(\frac{1}{a x}\right) \, dx=\sum _{j=1}^{n+1}\frac{b^{j-1}}{a^{j+m} } \frac{ \Gamma (1+k,-((j+m) \log (a)))
  }{(j+m)^{k+1}} \binom{n}{j-1}
\end{multline}
\end{theorem}
\begin{proof}
Since the right-hand sides of equations (\ref{eq:5.2}) and (\ref{eq:5.3}) are equal relative to (\ref{eq:5.1}) we may equate the left-hand sides and simplify the gamma function to yield the stated result.  
\end{proof}
\begin{example}
Use equation (\ref{eq:5.4}) and set $k=-1,a\to e^a$. Next form a second equation by replacing $a\to -a$ and take their difference. Next replace $a\to ai$ and simplify.
\end{example}
\begin{example}
Use equation (\ref{eq:5.4}) 
\begin{multline}
\int_0^1 \frac{x^m (1+b x)^n}{a^2+\log ^2(x)} \, dx\\
=\frac{i }{2 a b}\sum _{j=1}^{n+1} b^j \binom{n}{j-1} \left(e^{i a
   (j+m)} \Gamma (0,i a (j+m))-e^{-i a (j+m)} \Gamma (0,-i a (j+m))\right)
\end{multline}
where $n>0$.
\end{example}
\begin{example}
Use equation (\ref{eq:5.4}) and set $a=1,m\to m+p$ and take the infinite sum of both sides over $p\in[0,\infty)$ and simplify in terms of the Hurwitz-Lerch zeta function using [DLMF,\href{https://dlmf.nist.gov/25.14.i}{25.14.2}];
\begin{multline}\label{eq:5.6}
\int_0^1 \frac{x^m (1+b x)^n \log ^k\left(\frac{1}{x}\right)}{1+x z} \, dx=\frac{\Gamma (1+k) }{b}\sum _{j=1}^{n+1}
   b^j \binom{n}{j-1} \Phi (-z,1+k,j+m)
\end{multline}
where $Re(z)\geq -1$.
\end{example}
\begin{example}
Use equation (\ref{eq:5.6}) and take the $q$-th derivative with respect to $z$ then replace $m\to m-q$ and simplify.
 \begin{multline}
\int_0^1 \frac{x^m (1+b x)^n \log ^k\left(\frac{1}{x}\right)}{(1+x z)^{q+1}} \, dx=\frac{\Gamma (1+k)}{b (-1)^q q!} \sum
   _{j=1}^{n+1} b^j \binom{n}{-1+j} \frac{\partial ^q}{\partial z^q}\Phi (-z,1+k,j+m-q)
\end{multline}
where $Re(z)\geq -1$.
\end{example}
\begin{example}
Use equation (\ref{eq:5.4}) and replace $a\to e^a$ then take the first partial derivative with respect to $k$ and set $k=0$ and simplify using [DLMF,\href{https://dlmf.nist.gov/15.2.i}{15.2.1}].
\begin{multline}\label{eq:5.8}
\int_0^1 x^m (1+b x)^n \log (a+\log (x)) \, dx
=\frac{1}{b}\sum _{j=1}^{n+1} \frac{b^{j} e^{-a (j+m)} }{j+m}\binom{n}{-1+j}
   E_1(-a (j+m))\\
+\frac{\, _2F_1(1+m,-n;2+m;-b) \log (a)}{1+m}
\end{multline}
where $Re(a)>0,n>0$. Note when $Im(a)=0$ there is a singularity at $x=ai$.
\end{example}
\begin{example}
Use equation (\ref{eq:5.8}) and form another equation by $a\to -a$ and add these two equations and simplify the logarithm on the left-hand side.
 \begin{multline}
\int_0^1 x^m (1+b x)^n \log \left(a^2+\log ^2(x)\right) \, dx\\
=\frac{1}{b}\sum _{j=1}^{n+1} \frac{b^j 
  }{j+m}\binom{n}{j-1} \left(e^{-i a (j+m)} E_1(-i a (j+m))+e^{i a (j+m)} E_1(i a (j+m))\right)\\
+\frac{2 \, _2F_1(1+m,-n;2+m;-b)
   \log (a)}{1+m}
\end{multline}
where $n>0$.
\end{example}
\begin{example}
Use equation (\ref{eq:5.4}) and set $a=1$ then take the first partial derivative with respect to $k$ and set $k=0$ and simplify using [DLMF,\href{https://dlmf.nist.gov/15.2.i}{15.2.1}].
\begin{multline}
\int_0^1 x^m (1+b x)^n \log \left(\log \left(\frac{1}{x}\right)\right) \, dx=-\frac{\gamma  \,
   _2F_1(1+m,-n;2+m;-b)}{1+m}\\
-\frac{1}{b}\sum _{j=1}^{n+1}b^{j} \binom{n}{-1+j} \frac{ \log (j+m)}{j+m}
\end{multline}
where $n>0$.
\end{example}
\begin{example}
Use equation (\ref{eq:5.4}) set $a=1$ and form a second equation by replacing $m\to s$ and take their difference and simplify. Next apply l'Hopitals' rule as $k\to -1$ and simplify.
\begin{equation}
\int_0^1 \frac{(1+b x)^n \left(x^s-x^m\right)}{\log \left(\frac{1}{x}\right)} \, dx=\frac{1}{b}\sum _{j=1}^{n+1} b^j
   \binom{n}{j-1} \log \left(\frac{j+m}{j+s}\right)
\end{equation}
where $n>0$.
\end{example}
\begin{example}
Use equation (\ref{eq:5.4}) set $a=1$ and form a second equation by replacing $m\to s$ and take their difference and simplify. Next take the first partial derivative with respect to $k$ and apply l'Hopital's rule as $k\to -1$ and simplify.
\begin{multline}
\int_0^1 \frac{(1+b x)^n \left(-x^m+x^s\right) \log \left(\log \left(\frac{1}{x}\right)\right)}{\log
   \left(\frac{1}{x}\right)} \, dx\\
=-\frac{1}{2b}\sum _{j=1}^{n+1} \frac{b^{j} \Gamma (1+n)  }{ \Gamma (j) \Gamma (2-j+n)}\log \left(\frac{j+m}{j+s}\right)(2
   \gamma +\log ((j+m) (j+s)))
\end{multline}
where $n>0$.
\end{example}
\begin{example}
Use equation (\ref{eq:5.4}) and take the $q$-the derivative with respect to $b$, next form a second equation by replacing $b\to -b$ and take their difference and simplify.
\begin{multline}
\int_0^1 \frac{x^{m-q} \left((1-b x)^{n+q}-(1+b x)^{n+q}\right) \log
   ^k\left(\frac{1}{a x}\right)}{(1+n)_q} \, dx\\
=-\frac{b^{q-1} }{a^m}\sum _{j=1}^{n+1}
   \frac{\left((-1)^{j+q}+1\right) \left(\frac{b}{a}\right)^j \binom{n}{-1+j}
   \Gamma (j) \Gamma (1+k,-((j+m) \log (a)))}{\Gamma (j+q)
   (j+m)^{k+1}}
\end{multline}
where $q < 0, |q| < n$ and a singularity at $x=1/a$.
\end{example}
\begin{example}
Derivation of equation (4.272.10) in \cite{grad}. Errata.  Use equation (\ref{eq:5.4}) and set $k\to u-1,m\to a-1,a=1,b=-1$ and simplify. Next set $m\to n-1,a\to a+1$ for the first equation and multiply both sides by $n$ for the second equation. Then add these two equations and simplify.
\begin{multline}
\int_0^1 \log ^{u-1}\left(\frac{1}{x}\right) (x-1)^n \left(a+\frac{n x}{x-1}\right) x^{a-1} \, dx\\
=(-1)^n \Gamma
   (u) \sum _{j=1}^{n+1} (-1)^j \left(\frac{n }{(a+j)^u}\binom{n-1}{j-1}-\frac{a
   }{(a+j-1)^u}\binom{n}{j-1}\right)
\end{multline}
where $Re(p)>0,Re(q)>0$.
\end{example}
\begin{example}
Generalized form representation for equations (4.272.15) and (4.272.16) in \cite{grad} where $Re(p)>0, Re(q)>0$. Use equation (\ref{eq:5.4}) and set $a=1$. Next replace $x\to x^q$ and simplify. Next replace $m\to (m-q)/q$ and simplify. Next replace $j\to k+1,k\to n,m\to p, n\to m$ and simplify.
\begin{equation}
\int_0^1 x^{p-1} \left(1+b x^q\right)^m \log ^n\left(\frac{1}{x}\right) \, dx=\Gamma (1+n) \sum _{k=0}^m\binom{m}{k}\frac{ b^k}{(p+k q)^{n+1}}
\end{equation}
where $Re(p)>0,Re(q)>0$.
\end{example}
\begin{example}
Use equation (\ref{eq:5.4}) and set $k=-1,a\to e^a$. Next form a second equation by replacing $a\to -a$ and add the two equations and simplify with $a\to ai$.
\begin{multline}\label{eq:5.16}
\int_0^1 \frac{x^m (1+b x)^n \log (x)}{a^2+\log ^2(x)} \, dx\\
=-\frac{1}{2b}\sum _{j=1}^{n+1}  b^{j} e^{-i a
   (j+m)} \binom{n}{j-1} \left(\Gamma (0,-i a (j+m))+e^{2 i a (j+m)} \Gamma (0,i a (j+m))\right)
\end{multline}
where $Re(m)>0$.
\end{example}
\begin{example}
Use equation (\ref{eq:5.16}) and form a second equation by replacing $a\to ai$ and take their difference and simplify with $a\to a\sqrt{i}$.
\begin{multline}
\int_0^1 \frac{x^m (1+b x)^n \log (x)}{a^4+\log ^4(x)} \, dx\\
=-\sum _{j=1}^{n+1} \frac{i b^{-1+j} e^{-i \sqrt{2}
   a (j+m)} }{4 a^2}\binom{n}{-1+j} \left(-e^{(-1)^{3/4} a (j+m)} E_1\left(-\sqrt[4]{-1} a (j+m)\right)\right. \\ \left.
-e^{\sqrt{-4+3 i} a
   (j+m)} E_1\left(\sqrt[4]{-1} a (j+m)\right)+e^{\sqrt[4]{-1} a (j+m)} \left(E_1\left(-(-1)^{3/4} a (j+m)\right)\right.\right. \\ \left.\left.
+e^{2
   (-1)^{3/4} a (j+m)} E_1\left((-1)^{3/4} a (j+m)\right)\right)\right)
\end{multline}
where $Re(m)>0$.
\end{example}
\begin{example}
Use equation (\ref{eq:5.4}) and set $a=1,m\to m+p$ and multiply both sides by $(-z)^p$ then take the infinite sum of both sides over $p\in[0,\infty)$ and simplify using [DLMF,\href{https://dlmf.nist.gov/25.14.i}{25.14.2}]. Next set $z=1$ and simplify using [DLMF,\href{https://dlmf.nist.gov/25.11.iv}{25.111.8}]. Next set $b=0$ and apply l'Hopital's rule as $k\to 0$ and simplify using [DLMF,\href{https://dlmf.nist.gov/25.11.E12}{25.11.12}].
\begin{multline}\label{eq:5.18}
2 (-1)^{-m} B_{-1}(1+m,n)=\sum _{j=1}^{1+n} \binom{n}{j-1} \left(\psi ^{(0)}\left(\frac{j+m}{2}\right)-\psi
   ^{(0)}\left(\frac{1}{2} (1+j+m)\right)\right)
\end{multline}
where $m\in\mathbb{C}$.
\end{example}
\begin{example}
Use equation (\ref{eq:5.18}) and take the exponential of both sides and simplify.
\begin{multline}
\exp \left(\int_{-z}^z 2 (-1)^{-m} B_{-1}(1+m,n) \, dm\right)=\prod _{j=1}^{n+1} \left(\frac{\Gamma
   \left(\frac{1}{2} (1+j-z)\right) \Gamma \left(\frac{j+z}{2}\right)}{\Gamma \left(\frac{j-z}{2}\right) \Gamma
   \left(\frac{1}{2} (1+j+z)\right)}\right)^{2 \binom{n}{-1+j}}
\end{multline}
where $m\in\mathbb{C}$.
\end{example}
\begin{example}
Use equation (\ref{eq:5.4}) and set $a=1$, then form a second equation by replacing $m\to s$ and take their difference. Next replace $m\to m+p$ and multiply both sides by $(-z)^p$ and take the infinite sum of both sides over $p\in[0,\infty)$ and simplify using [DLMF,\href{https://dlmf.nist.gov/25.14.i}{25.14.2}]. Next take the first partial derivative with respect to $k$ and apply l'hopital's rule as $k\to -1$ and simplify using [DLMF,\href{https://dlmf.nist.gov/25.11.E8}{25.11.8}].
\begin{multline}\label{eq:5.20}
\int_0^1 \frac{(1+b x)^n \left(-2 x^m+x^s (1+x)\right) \log \left(\log \left(\frac{1}{x}\right)\right)}{(1+x)
   \log \left(\frac{1}{x}\right)} \, dx\\
=\frac{\Gamma (1+n)}{2b}\sum _{j=1}^{n+1} \frac{b^{j}  }{ \Gamma (j) \Gamma (2-j+n)}\left(-\log ^2(2)-\gamma  \log
   (4)+4 (\gamma +\log (2)) (\log (-2+j+m)\right. \\ \left.
-\log (-1+j+m))+2 \gamma  \log (j+s)+\log ^2(j+s)+4 (\gamma +\log (2))
   \left(\text{log$\Gamma $}\left(\frac{1}{2} (-2+j+m)\right)\right.\right. \\ \left.\left.
-\text{log$\Gamma $}\left(\frac{1}{2}
   (-1+j+m)\right)\right)-2 \zeta''\left(0,\frac{j+m}{2}\right)+2
   \zeta''\left(0,\frac{1}{2} (1+j+m)\right)\right)
\end{multline}
where $|Re(m)|<1,Re(s)|<1$.
\end{example}
\begin{example}
Use equation (\ref{eq:5.20}) and set $m=1/2,s=-1/2,b=1,n=0,j=1$ and simplify.
\begin{multline}
\int_0^1 \frac{(1-x) \log \left(\log \left(\frac{1}{x}\right)\right)}{\sqrt{x} (1+x) \log
   \left(\frac{1}{x}\right)} \, dx\\
=2 i \pi  (\gamma +\log (2))-\gamma  \log (4)+2 (\gamma +\log (2))
   \left(\text{log$\Gamma $}\left(-\frac{1}{4}\right)-\text{log$\Gamma
   $}\left(\frac{1}{4}\right)\right)\\-\zeta''\left(0,\frac{3}{4}\right)+\zeta''\left(0,\frac{5}{4}\right)
\end{multline} 
\end{example}
\begin{example}
Use equation (\ref{eq:5.4}) and multiply both sides by $(-\beta)^m$, then replace $m\to m+l$ and take the infinite sum over $l\in[1,\infty)$ and simplify.
\begin{multline}\label{eq:5.22}
\int_0^1 x^m (1+b x)^n \log ^k\left(\frac{1}{a x}\right) \log (1+x \beta ) \, dx\\
=-\sum _{l=1}^{\infty } \sum
   _{j=1}^{n+1} \frac{(-1)^l a^{-j-l-m} b^{-1+j} (j+l+m)^{-1-k} \beta ^l \binom{n}{-1+j} \Gamma (1+k,-((j+l+m) \log
   (a)))}{l}
\end{multline}
where $n>0$.
\end{example}
\begin{example}
Use equation (\ref{eq:5.22}) and set $k=-1,a\to e^a$ and form a second equation by replacing $a\to -a$ and add these two equations and simplify the logarithm terms on the left-hand side with $a\to ai$.
\begin{multline}
\int_0^1 \frac{x^m (1+b x)^n \log (x) \log (1+x \beta )}{a^2+\log ^2(x)} \, dx\\
=\sum _{l=1}^{\infty } \sum
   _{j=1}^{n+1} \frac{(-1)^l b^{-1+j} e^{-i a (j+l+m)} }{2 l}\beta ^l \binom{n}{-1+j} \\\left(\Gamma (0,-i a (j+l+m))+e^{2 i a
   (j+l+m)} \Gamma (0,i a (j+l+m))\right)
\end{multline}
where $n>0,Re(\beta)\leq 1$.
\end{example}
\begin{example}
Use equation (\ref{eq:5.22}) and set $a=1$ then take the first partial derivative with respect to $k$ and set $k=0$ and simplify.
\begin{multline}
\int_0^1 x^m (1+b x)^n \log (1+x \beta ) \log \left(\log \left(\frac{1}{x}\right)\right) \, dx\\
=\sum
   _{l=1}^{\infty } \frac{(-1)^l \gamma  \beta ^l \, _2F_1(1+l+m,-n;2+l+m;-b)}{l (1+l+m)}\\
+\sum _{l=1}^{\infty } \sum
   _{j=1}^{n+1} \frac{(-1)^l b^{-1+j} \beta ^l \binom{n}{-1+j} \log (j+l+m)}{l (j+l+m)}
\end{multline}
where $n>0,Re(\beta)\leq 1$.
\end{example}
\begin{example}
Use equation (\ref{eq:5.22}) and form a second equation by replacing $m\to s$ and take their difference. Next take the first partial derivative with respect to $k$ and apply l'Hopital's rule as $k\to -1$ and simplify.
\begin{multline}
\int_0^1 \frac{\left(x^2-1\right) \log (1+x) \log \left(\log \left(\frac{1}{x}\right)\right)}{\sqrt{x} \log
   \left(\frac{1}{x}\right)} \, dx\\
=\sum _{l=1}^{\infty } \frac{(-1)^l \log \left(\frac{1+2 l}{3+2 l}\right) \left(2
   \gamma +\log \left(-\frac{1}{4}+(1+l)^2\right)\right)}{2 l}\\
+\sum _{l=1}^{\infty } \frac{(-1)^l \log \left(\frac{3+2
   l}{5+2 l}\right) \left(2 \gamma +\log \left(-\frac{1}{4}+(2+l)^2\right)\right)}{2 l}
\end{multline}
\end{example}
\section{Definite integrals involving the Bessel function}
In this section we derive definite integrals of the Bessel function and the product of Bessel functions of the first kind in terms of special functions. We also evaluate special cases along with providing derivations for the definite integrals listed in current literature. The contour integral representation involving the Bessel function of the first kind $J_{v}(x)$, used in the derivation of the main theorem in this section is given by;
\begin{multline}\label{eq:6.1}
\frac{1}{2\pi i}\int_{C}\int_{0}^{1}\frac{(-1)^l a^w w^{-k-1} 2^{-2 l-v} (b x+1)^n z^{2 l+v} x^{2 l+m+v+w}}{l! \Gamma (l+v+1)}dxdw\\
=\frac{1}{2\pi i}\int_{C}\sum_{j=1}^{n+1}\sum_{l=0}^{\infty}\frac{(-1)^l a^w
   b^{j-1} w^{-k-1} 2^{-2 l-v} \binom{n}{j-1} z^{2 l+v}}{l! \Gamma (l+v+1) (j+2 l+m+v+w)}dw
\end{multline}
where $|Re(b)|<1$.
\subsection{Left-hand side contour integral representation}
Using a generalization of Cauchy's integral formula \ref{intro:cauchy} we replace $y\to \log(ax)$ and multiply both sides by $\frac{(-1)^l 2^{-2 l-v} (b x+1)^n z^{2 l+v} x^{2 l+m+v}}{l! \Gamma (l+v+1)}$ then take the definite integral over $x\in[0,1]$ and simplify to get;
\begin{multline}\label{eq:6.2}
\int_{0}^{1}\frac{(-1)^l 2^{-2 l-v} (b x+1)^n z^{2 l+v} \log ^k(a x) x^{2 l+m+v}}{k! l! \Gamma (l+v+1)}dx\\
=\frac{1}{2\pi i}\int_{0}^{1}\int_{C}\frac{(-1)^l w^{-k-1}
   2^{-2 l-v} (a x)^w (b x+1)^n z^{2 l+v} x^{2 l+m+v}}{l! \Gamma (l+v+1)}dwdx\\
   =\frac{1}{2\pi i}\int_{C}\int_{0}^{1}\frac{(-1)^l w^{-k-1}
   2^{-2 l-v} (a x)^w (b x+1)^n z^{2 l+v} x^{2 l+m+v}}{l! \Gamma (l+v+1)}dxdw
\end{multline}
where $|Re(m)|<1$.
We are able to switch the order of integration over $x$ and $w$ using Fubini's theorem for multiple integrals see page 178 in \cite{gelca}, since the integrand is of bounded measure over the space $\mathbb{C} \times [0,\infty)$.
\subsection{Right-hand side contour integral representation}
Using a generalization of Cauchy's integral formula \ref{intro:cauchy} we replace $y\to x+\log(a)$ and multiply both sides by $e^{mx}$ then take the infinite integral over $x\in[0,\infty)$ and simplify using [DLMF,\href{https://dlmf.nist.gov/8.6.E5}{8.6.5}] to get;
\begin{equation}
\frac{a^{-m} (-m)^{-k-1} \Gamma (k+1,-m \log (a))}{k!}=-\frac{1}{2\pi i}\int_{C}\frac{a^w w^{-k-1}}{m+w}dw
\end{equation}
 Next we replace $m\to m + j + 2 l + v$ and multiply both sides by $-\frac{(-1)^l b^{j-1} 2^{-2 l-v} \binom{n}{j-1} z^{2 l+v}}{l! \Gamma (l+v+1)}$ and take the double sum of both sides over $j\in[1,n+1],l\in[0,\infty)$ to get;
\begin{multline}\label{eq:6.4}
-\sum_{j=1}^{n+1}\sum_{l=0}^{\infty}\frac{(-1)^l b^{j-1} 2^{-2 l-v} \binom{n}{j-1} z^{2 l+v} \Gamma (k+1,-((j+2
   l+m+v) \log (a)))}{k! l! \Gamma (l+v+1)a^{j+2 l+m+v} (-j-2 l-m-v)^{k+1} }\\
=\sum_{j=1}^{n+1}\sum_{l=0}^{\infty}\frac{1}{2\pi i}\int_{C}\frac{(-1)^l a^w b^{j-1} w^{-k-1} 2^{-2 l-v} \binom{n}{j-1} z^{2 l+v}}{l!
   \Gamma (l+v+1) (j+2 l+m+v+w)}dw\\
=\frac{1}{2\pi i}\int_{C}\sum_{j=1}^{n+1}\sum_{l=0}^{\infty}\frac{(-1)^l a^w b^{j-1} w^{-k-1} 2^{-2 l-v} \binom{n}{j-1} z^{2 l+v}}{l!
   \Gamma (l+v+1) (j+2 l+m+v+w)}dw
\end{multline}
We are able to switch the order of integration and summation over $w$ using Tonellii's theorem for  integrals and sums see page 177 in \cite{gelca}, since the summand is of bounded measure over the space $\mathbb{C} \times [0,\infty)$
\begin{theorem}
For all $|Re(m)|<1$ then,
\begin{multline}\label{eq:6.1}
\int_0^1 x^m (1+b x)^n J_v(x z) (-\log (a x))^k \, dx\\
=\frac{z^v }{b a^m (2 a)^v}\sum _{l=0}^{\infty } \sum _{j=1}^{n+1}\binom{n}{j-1} 
   \frac{(-1)^l \left(\frac{z}{2}\right)^{2 l} b^j \Gamma (1+k,-(j+2 l+m+v) \log (a))}{a^{j+2 l}
   \Gamma (1+l) \Gamma (1+l+v) (j+2 l+m+v)^{k+1}}
\end{multline}
where $n>0$.
\end{theorem}
\begin{proof}
Since the right-hand sides of equations (\ref{eq:6.2}) and (\ref{eq:6.4}) are equal relative to (\ref{eq:6.1}) we may equate the left-hand sides and taking the infinite sum of the left-hand side over $l\in[0,\infty)$ and simplify the gamma function to yield the stated result.
\end{proof}
The contour integral representation involving the product of Bessel functions of the first kind $J_{u}(x)J_{v}(x)$, used in the derivation of the main theorem in this section is given by;
\begin{multline}\label{eq:6.6}
\frac{1}{2\pi i}\int_{0}^{1}\int_{C}\frac{(-1)^l (b x+1)^n x^{2 l+m} 2^{-2 l-\mu -v} \log ^k(a x) (l+v+\mu +1)_l z^{2 l+\mu +v}}{k! l! \Gamma (l+\mu
   +1) \Gamma (l+v+1)}dwdx\\
=\frac{1}{2\pi i}\sum_{j=1}^{n+1}\sum_{l=0}^{\infty}\frac{(-1)^l w^{-k-1} (a x)^w (b x+1)^n x^{2 l+m} 2^{-2 l-\mu -v} (l+v+\mu +1)_l z^{2 l+\mu
   +v}}{l! \Gamma (l+\mu +1) \Gamma (l+v+1)}
\end{multline}
where $|Re(m)|<1$.
\subsection{Left-hand side contour integral representation}
Using a generalization of Cauchy's integral formula \ref{intro:cauchy} we replace $y\to \log(ax)$ and multiply both sides by $\frac{(-1)^l (b x+1)^n x^{2 l+m} 2^{-2 l-\mu -v} (l+v+\mu +1)_l z^{2 l+\mu +v}}{l! \Gamma (l+\mu +1) \Gamma (l+v+1)}$ then take the definite integral over $x\in[0,1]$ and simplify to get;
\begin{multline}\label{eq:6.7}
\int_{0}^{1}\frac{(-1)^l (b x+1)^n x^{2 l+m} 2^{-2 l-\mu -v} \log ^k(a x) (l+v+\mu +1)_l z^{2 l+\mu +v}}{k! l! \Gamma (l+\mu
   +1) \Gamma (l+v+1)}dx\\
=\frac{1}{2\pi i}\int_{0}^{1}\int_{C}\frac{(-1)^l w^{-k-1} (a x)^w (b x+1)^n x^{2 l+m} 2^{-2 l-\mu -v} (l+v+\mu +1)_l z^{2 l+\mu
   +v}}{l! \Gamma (l+\mu +1) \Gamma (l+v+1)}dwdx\\
   =\frac{1}{2\pi i}\int_{C}\int_{0}^{1}\frac{(-1)^l w^{-k-1} (a x)^w (b x+1)^n x^{2 l+m} 2^{-2 l-\mu -v} (l+v+\mu +1)_l z^{2 l+\mu
   +v}}{l! \Gamma (l+\mu +1) \Gamma (l+v+1)}dwdx
\end{multline}
We are able to switch the order of integration over $x$ and $w$ using Fubini's theorem for multiple integrals see page 178 in \cite{gelca}, since the integrand is of bounded measure over the space $\mathbb{C} \times [0,\infty)$.
\subsection{Right-hand side contour integral representation}
Using a generalization of Cauchy's integral formula \ref{intro:cauchy} we replace $y\to x+\log(a)$ and multiply both sides by $e^{mx}$ then take the infinite integral over $x\in[0,\infty)$ and simplify using [DLMF,\href{https://dlmf.nist.gov/8.6.E5}{8.6.5}] to get;
\begin{equation}
\frac{a^{-m} (-m)^{-k-1} \Gamma (k+1,-m \log (a))}{k!}=-\frac{1}{2\pi i}\int_{C}\frac{a^w w^{-k-1}}{m+w}dw
\end{equation}
 Next we replace $m\to m + j + 2 l + m$ and multiply both sides by\\ $-\frac{(-1)^l b^{j-1} \binom{n}{j-1} 2^{-2 l-\mu -v} (l+v+\mu +1)_l z^{2 l+\mu +v}}{l! \Gamma (l+\mu +1) \Gamma (l+v+1)}$ and take the double sum of both sides over $j\in[1,n+1],l\in[0,\infty)$ to get;
 \begin{multline}\label{eq:6.9}
-\sum_{j=1}^{n+1}\sum_{l=0}^{\infty}\frac{(-1)^l b^{j-1} \binom{n}{j-1} 2^{-2 l-\mu -v}  (l+v+\mu +1)_l z^{2 l+\mu +v}
   \Gamma (k+1,-((j+2 l+m) \log (a)))}{k! l! \Gamma (l+\mu +1) \Gamma (l+v+1)a^{j+2 l+m}(-j-2 l-m)^{k+1} }\\
=\frac{1}{2\pi i}\sum_{j=1}^{n+1}\sum_{l=0}^{\infty}\int_{C}\frac{(-1)^l a^w b^{j-1} w^{-k-1}
   \binom{n}{j-1} 2^{-2 l-\mu -v} (l+v+\mu +1)_l z^{2 l+\mu +v}}{l! \Gamma (l+\mu +1) \Gamma (l+v+1) (j+2
   l+m+w)}dw
\end{multline}
We are able to switch the order of integration and summation over $w$ using Tonellii's theorem for  integrals and sums see page 177 in \cite{gelca}, since the summand is of bounded measure over the space $\mathbb{C} \times [0,\infty)$
\begin{theorem}
For all $|Re(m)|<1$ then,
\begin{multline}\label{eq:6.10}
\int_0^1 x^m (1+b x)^n J_v(x z) J_{\mu }(x z) (-\log (a x))^k \, dx\\
=\frac{z^{v+\mu }}{2^{v+\mu } a^{m+v+\mu } b} \sum _{l=0}^{\infty } \sum
   _{j=1}^{n+1} \frac{(-1)^l \left(\frac{z}{2}\right)^{2 l} b^j }{a^{j+2 l} \left(\Gamma (l+1) \Gamma (1+l+v) \Gamma (1+l+\mu ) (j+2 l+m+v+\mu
   )^{k+1}\right)}\\
\binom{n}{j-1} \Gamma (1+k,-((j+2 l+m+v+\mu ) \log
   (a))) (1+l+v+\mu )_l
\end{multline}
\end{theorem}
\begin{proof}
Since the right-hand sides of equations (\ref{eq:6.7}) and (\ref{eq:6.9}) are equal relative to (\ref{eq:6.6}) we may equate the left-hand sides and taking the infinite sum of the left-hand side over $l\in[0,\infty)$ and simplify the gamma function to yield the stated result.
\end{proof}
\subsection{Evaluations involving the Bessel function}
\begin{example}
Derivation of entry (6.561.1) in \cite{grad}. Use equation (\ref{eq:6.1}) and set $k=0,n=0,j=1,m\to v,z\to a$ and simplify.
\begin{multline}
\int_0^1 x^v J_v(a x) \, dx=\frac{2^{-v} a^v \,
   _1F_2\left(\frac{1}{2}+v;1+v,\frac{3}{2}+v;-\frac{a^2}{4}\right)}{(1+2 v) \Gamma (1+v)}\\
=2^{v-1} a^{-v} \sqrt{\pi }
   \Gamma \left(v+\frac{1}{2}\right) (J_v(a) \pmb{H}_{v-1}(a)-\pmb{H}_v(a) J_{v-1}(a))
\end{multline}
where $Re(v)>-1/2$.
\end{example}
\begin{example}
Derivation of entry (6.561.5) in \cite{grad}. Errata. Use equation (\ref{eq:6.1}) and set $k=0,n=0,j=1,m\to v+1,z\to a,v\to v+1$ and simplify.
\end{example}
\begin{example}
Derivation of entry (6.561.9) in \cite{grad}. Use equation (\ref{eq:6.1}) and set $k=0,n=0,j=1,m\to 1-v,z\to a$ and simplify.
\begin{equation}
\int_0^1 x^{1+v} J_{1+v}(a x) \, dx=\frac{2^{-1-v} a^{1+v} \,
   _1F_2\left(\frac{3}{2}+v;2+v,\frac{5}{2}+v;-\frac{a^2}{4}\right)}{(1+2 (1+v)) \Gamma (2+v)}
\end{equation}
where $|Re(v)<-1$.
\end{example}
%
%
\begin{example}
Derivation of entry (6.561.13(7)) in \cite{grad}. Use equation (\ref{eq:6.1}) and set $k=0,n=0,j=1,m\to u,z\to a$ and simplify using [Wolfram,\href{https://mathworld.wolfram.com/LommelFunction.html}{2}] in terms of the Lommel $S^{(2)}$ function.
\begin{multline}\label{eq:6.14}
\int_0^1 x^u J_v(a x) \, dx=\frac{2^{-v} a^v \,
   _1F_2\left(\frac{1}{2}+\frac{u}{2}+\frac{v}{2};\frac{3}{2}+\frac{u}{2}+\frac{v}{2},1+v;-\frac{a^2}{4}\right)}{(1+u+v)
   \Gamma (1+v)}\\
=\frac{2^u \Gamma \left(\frac{1}{2} (v+u+1)\right)}{a^{u+1} \Gamma \left(\frac{1}{2}
   (v-u+1)\right)}+a^{-u} ((u+v-1) J_v(a) S^{(2)}_{u-1,v-1}(a)-J_{v-1}(a) S^{(2)}_{u,v}(a))
\end{multline}
where $Re(a)>0,Re(u+v)>1$.
\end{example}
\begin{example}
From equation (\ref{eq:6.14}).
\begin{multline}
\frac{2^{-v} a^v \,
   _1F_2\left(\frac{1}{2}+\frac{u}{2}+\frac{v}{2};\frac{3}{2}+\frac{u}{2}+\frac{v}{2},1+v;-\frac{a^2}{4}\right)}{(1+u+v)
   \Gamma (1+v)}\\
=\frac{2^u \Gamma \left(\frac{1}{2} (v+u+1)\right)}{a^{u+1} \Gamma \left(\frac{1}{2}
   (v-u+1)\right)}+a^{-u} ((u+v-1) J_v(a) S^{(2)}_{u-1,v-1}(a)-J_{v-1}(a) S^{(2)}_{u,v}(a))
\end{multline}
where $Re(a)>0$.
\end{example}
\begin{example}
Use equation (\ref{eq:6.10}) and set $k=0,n\to \beta$ and simplify.
\begin{multline}
\int_0^1 x^m (1+b x)^{\beta } J_v(x z) J_{\mu }(x z) \, dx\\
=\sum _{l=0}^{\infty } \frac{(-1)^l 2^{-2 l-v-\mu }
   z^{2 l+v+\mu } \Gamma (1+2 l+v+\mu ) \, _2F_1(-\beta ,1+2 l+m+v+\mu ;2+2 l+m+v+\mu ;-b)}{(1+2 l+m+v+\mu ) \Gamma
   (1+l) \Gamma (1+l+v) \Gamma (1+l+\mu ) \Gamma (1+l+v+\mu )}
\end{multline}
where $Re(m)>0$.
\end{example}
\begin{example}
Use equation (\ref{eq:6.10}) and set $a=1,n=0,j=1,m=0$ and take the first partial derivative with respect to $k$ and set $k=0$ and simplify.
\begin{multline}
\int_0^1 J_v(x z) J_{\mu }(x z) \log \left(\log \left(\frac{1}{x}\right)\right) \, dx\\
=-\frac{2^{-v-\mu } \gamma 
   z^{v+\mu } \, _3F_4\left(\frac{1}{2}+\frac{v}{2}+\frac{\mu }{2},\frac{1}{2}+\frac{v}{2}+\frac{\mu
   }{2},1+\frac{v}{2}+\frac{\mu }{2};1+v,\frac{3}{2}+\frac{v}{2}+\frac{\mu }{2},1+\mu ,1+v+\mu ;-z^2\right)}{(1+v+\mu )
   \Gamma (1+v) \Gamma (1+\mu )}\\
-\sum _{l=0}^{\infty } \frac{(-1)^l 2^{-2 l-v-\mu } z^{2 l+v+\mu } \log (1+2 l+v+\mu )
   (1+l+v+\mu )_l}{(1+2 l+v+\mu ) \Gamma (1+l) \Gamma (1+l+v) \Gamma (1+l+\mu )}
\end{multline}
where $|Re(v)>|<-1,|Re(\mu)>|<-1$.
\end{example}
\begin{example}
Use equation (\ref{eq:6.10}) and set $k=-1,a\to e^a$ and form a second equation by replacing $a\to -a$ and take their difference with $a\to ai,n=m=0,j=1$ and simplify.
\begin{multline}
\int_0^1 \frac{J_v(x z) J_{\mu }(x z)}{a^2+\log ^2(x)} \, dx
=-\sum _{j=1}^{n+1} \sum _{l=0}^{\infty } \frac{i
   (-1)^l 2^{-1-2 l-v-\mu } e^{-i a (1+2 l+v+\mu )} z^{2 l+v+\mu } }{a \Gamma (1+l) \Gamma (1+l+v) \Gamma (1+l+\mu
   )}\\\left(\Gamma (0,-i a (1+2 l+v+\mu ))-e^{2 i a (1+2
   l+v+\mu )} \Gamma (0,i a (1+2 l+v+\mu ))\right) (1+l+v+\mu )_l
\end{multline}
where $|Re(v)>|<-1,|Re(\mu)>|<-1$.
\end{example}
\section{Definite integrals involving the hypergeometric function}
The contour integral representation used in this section is given by;
\begin{multline}
\frac{1}{2\pi i}\int_{C}\int_{0}^{1}\sum_{h=0}^{\infty}\sum_{l=0}^{\infty}\frac{a^w s^h w^{-k-1} z^l (b x+1)^n (\alpha )_l (\beta )_l x^{h+l+m+w}}{h! l! (\gamma )_l}dxdw\\
=\frac{1}{2\pi i}\int_{C}\sum_{h=0}^{\infty}\sum_{j=1}^{n+1}\sum_{l=0}^{\infty}\frac{a^w b^{j-1}
   s^h w^{-k-1} z^l \binom{n}{j-1} (\alpha )_l (\beta )_l}{h! l! (\gamma )_l (h+j+l+m+w)}
\end{multline}
where $|Re(m)|<1$.
We repeat the procedure in section 6. The left-hand side contour integral is derived using the approach in section 6. We multiply both sides of Cauchy's integral formula \ref{intro:cauchy} by $\frac{s^h z^l (b x+1)^n (\alpha )_l (\beta )_l x^{h+l+m}}{h! l! (\gamma )_l}$ and simplify the sums over $l\in[0,\infty),h\in[0,\infty)$. For the right-hand side contour integral representation we replace $m\to h + j + l + m$ and multiply both sides by $\frac{b^{j-1} s^h z^l \binom{n}{j-1} (\alpha )_l (\beta )_l}{h! l! (\gamma )_l}$ and simplify.
\begin{theorem}
For all $Re(\gamma)>1$ then,
\begin{multline}\label{eq:7.1}
\int_0^1 e^{s x} x^m (1+b x)^n \, _2F_1(\alpha ,\beta ;\gamma ;x z) \log ^k(a x) \, dx\\
=\sum _{l=0}^{\infty }
   \sum _{h=0}^{\infty } \sum _{j=1}^{n+1} \frac{a^{-h-j-l-m} b^{-1+j}  s^h z^l \binom{n}{-1+j} \Gamma
   (1+k,-((h+j+l+m) \log (a))) (\alpha )_l (\beta )_l}{h! l! (\gamma )_l(-h-j-l-m)^{1+k}}
\end{multline}
\end{theorem}
\begin{example}
Derivation of equation (20.2.8) in \cite{erdt2}. Use equation (\ref{eq:7.1}) and set $k=0,z=1,s\to -z,m\to \gamma-1$ and simplify the sum over $j\in[1,n+1]$. Next replace $n\to \rho-1,b=-1$ and simplify the sum over $h\in[0,\infty)$. Note equation  (\ref{eq:7.1}) can be used to derive equations (7.512.1-6) in \cite{grad}.
\begin{multline}\label{eq:7.2}
\int_0^1 e^{-x z} (1-x)^{-1+\rho } x^{-1+\gamma } \, _2F_1(\alpha ,\beta ;\gamma ;x) \, dx\\
=\sum _{l=0}^{\infty }
   \frac{\Gamma (l+\alpha ) \Gamma (l+\beta ) \Gamma (\gamma ) \Gamma (\rho ) \, _1F_1(l+\gamma ;l+\gamma +\rho
   ;-z)}{\Gamma (1+l) \Gamma (\alpha ) \Gamma (\beta ) \Gamma (l+\gamma +\rho )}\\
=\frac{(\Gamma (\gamma ) \Gamma (\rho )
   \Gamma (\gamma +\rho -\alpha -\beta )) \exp (-z) \, _2F_2(\rho ,\gamma +\rho -\alpha -\beta ;\gamma +\rho -\alpha
   ,\gamma +\rho -\beta ;z)}{\Gamma (\gamma +\rho -\alpha ) \Gamma (\gamma +\rho -\beta )}
\end{multline}
\end{example}
\begin{example}
A Series functional equation from equation (\ref{eq:7.2}) and equation (20.2.8) in \cite{erdt2}.
\begin{multline}
\sum _{l=0}^{\infty } \frac{(\Gamma (l+\alpha ) \Gamma (l+\beta )) }{\Gamma (1+l) \Gamma (l+\gamma +\rho )}\, _1F_1(l+\gamma ;l+\gamma +\rho
   ;-z)\\
=\frac{\Gamma (\gamma +\rho -\alpha -\beta ) \Gamma (\alpha ) \Gamma
   (\beta ) \exp (-z) }{\Gamma (\gamma +\rho -\alpha ) \Gamma (\gamma +\rho -\beta )}\, _2F_2(\rho ,\gamma +\rho -\alpha -\beta ;\gamma +\rho -\alpha ,\gamma +\rho -\beta
   ;z)
\end{multline}
where $Re(\gamma)>1$.
\end{example}
\section{Definite integrals over a finite interval}
The integrals derived in the proceeding sections can be considered as generalized forms of the Riemann-Louville integral of order $\alpha$ see \cite{erdeyli} pages 151-170 and \cite{erdt2} pages 185-200 (Chapter VIII). The Riemann-Louville integral of order $\alpha > 0$, of a function $f\in [0,\infty)$ is defined as 
\begin{equation}
I^{\alpha}f(x)=\frac{1}{\Gamma(\alpha)}\int_{0}^{x}(x-y)^{\alpha -1}f(y)dy
\end{equation}
where $Re(x) > 0$. If $\alpha$ is an integer, this is simply the $\alpha$ times repeated integral of $f$ with fixed lower limit $0$. $I^{\alpha}f$ can be regarded as the convolution of $f$ (assumed to vanish for $Re(x) < 0$) with the function $P_{\alpha}$ defined by;
\begin{equation}
P_{\alpha}=\begin{cases}
			\frac{x^{\alpha-1}}{\Gamma(\alpha)}, & \text{if $Re(x)>0$}\\
            0, & \text{if $Re(x)\leq 0$}
		 \end{cases}
\end{equation}
and the formula $I^{\alpha}f = P_{\alpha}\times f$ is capable of considerable expansion.
\\\\
In this example, the definite integral contour representation over a finite interval is given by;
\begin{multline}\label{eq:8.1}
\frac{1}{2\pi i}\int_{C}\int_{0}^{b}a^w z^j w^{-k-1} \binom{-d}{j} x^{c j+m+w}dxdw
=\frac{1}{2\pi i}\int_{C}\frac{a^w z^j w^{-k-1} \binom{-d}{j} b^{c j+m+w+1}}{c
   j+m+w+1}dw
\end{multline}
where $|Re(m)|<1$.
\subsection{Left-hand side contour integral representation}
Use equation (\ref{intro:cauchy}) and replace $y\to \log(ax)$ and multiply both sides by $z^j \binom{-d}{j} x^{c j+m}$  and simplify. 
\begin{multline}\label{eq:8.2}
\frac{1}{2\pi i}\int_{C}\int_{0}^{b}\frac{z^j \binom{-d}{j} \log ^k(a x) x^{c j+m}}{k!}dxdw=\frac{1}{2\pi i}\int_{C}z^j w^{-k-1} (a x)^w \binom{-d}{j} x^{c j+m}dw
\end{multline}
\subsection{Right-hand side contour integral representation}
Use equation (\ref{eq:4.4}) and replace $m\to 1+c j+m,a\to a b$ and multiply both sides by $z^j \binom{-d}{j} b^{c j+m+1}$ and simplify. 
\begin{multline}\label{eq:8.3}
\frac{z^j \binom{-d}{j} b^{c j+m+1} (-c j-m-1)^{-k-1} (a b)^{-c j-m-1} \Gamma (k+1,-((c j+m+1) \log (a
   b)))}{k!}\\
=-\frac{1}{2\pi i}\int_{C}\frac{z^j w^{-k-1} (a b)^w \binom{-d}{j} b^{c j+m+1}}{c j+m+w+1}dw
\end{multline}
\begin{theorem}
For all $|Re(m)|<1$ then,
\begin{multline}\label{eq:8.4}
\int_0^b \frac{x^m (-\log (a x))^k}{\left(1+x^c z\right)^d} \,
   dx\\
=\frac{1}{a^{1+m}}\sum _{j=0}^{\infty } \frac{\left(\frac{z}{a^c}\right)^j
  }{(1+c
   j+m)^{k+1}} \binom{-d}{j} \Gamma (1+k,-((1+c j+m) \log (a b)))
\end{multline}
\end{theorem}
\begin{proof}
Since the right-hand sides of equations (\ref{eq:8.2}) and (\ref{eq:8.3}) are equal relative to (\ref{eq:8.1}) we may equate the left-hand sides and taking the infinite sum of the left-hand side over $j\in[0,\infty)$ and simplify the gamma function. We are able to switch the order of integration over $x$ and $w$ using Fubini's theorem for multiple integrals see page 178 in \cite{gelca}, since the integrand is of bounded measure over the space $\mathbb{C} \times [0,b]$. We are able to switch the order of integration and summation over $w$ using Tonellii's theorem for  integrals and sums see page 177 in \cite{gelca}, since the summand is of bounded measure over the space $\mathbb{C} \times [0,\infty)$. Simplifying the sums will yield the stated result. 
\end{proof}
\subsection{Derivations and evaluations}
The definite integral form in equation (\ref{eq:8.4}) can be used to derive entries 3.191.1, 3.191.3, 3.192.1-3, 3.193, 3.194.1, 3.194.5, 3.194.8, 3.216, 3.222.1, 3.226, 3.231, 3.234,  3.237, 4.243, 4.244, 4.245, 4.246, 4.247, 4.251, 4.253, 4.254, 4.256, 4.261, 4.262, 4.263, 4.264, 4.265, 4.266, 4.267, 4.268, 4.269, 4.271, 4.272, 4.273, 4.274, 4.275, 4.281, 4.282, 4.283, 4.285, 4.291, 4.293, 4.294, 4.295, 4.296, 4.297, 4.298 in \cite{grad}, and 4.15.1-3 in \cite{brychkov} using parameter substitutions and algebraic methods.
\begin{example}
Derivation of equation (4.241.11) in \cite{grad}. Use equation (\ref{eq:8.4}) and set $k\to 1,d\to \frac{1}{2},c\to 2,z\to -1,m\to -\frac{1}{2},a\to 1,b\to 1$ and simplify the series on the right-hand side.
\begin{equation}
\int_0^1 \frac{\log (x)}{\sqrt{x} \sqrt{1-x^2}} \, dx=-\frac{\pi ^{3/2} \Gamma \left(\frac{5}{4}\right)}{\Gamma
   \left(\frac{3}{4}\right)}=\frac{1}{8} \left(-\sqrt{2 \pi }\right) \Gamma \left(\frac{1}{4}\right)^2
\end{equation}
\end{example}
\begin{example}
Derivation of equation (4.247.2(6)) in \cite{grad}. Use equation (\ref{eq:8.4}) and set $k\to 1,d\to \frac{1}{n},c\to 2,z\to -1,m\to -\frac{n-1}{n},a\to 1,b\to 1$ and simplify the series.
\begin{multline}
\int_0^1 \frac{\log (x)}{\left(x^{n-1} \left(1-x^2\right)\right)^{1/n}} \, dx=-2^{-2-\frac{1}{n}} \sqrt{\pi }
   \cot \left(\frac{\pi }{2 n}\right) \Gamma \left(\frac{1}{2 n}\right) \Gamma \left(\frac{-1+n}{2
   n}\right)\\
   =-\frac{\pi  B\left(\frac{1}{2 n}\frac{1}{2 n}\right)}{8 \sin \left(\frac{\pi }{2 n}\right)}
\end{multline}
\end{example}
\begin{example}
Derivation of equation (4.274) in \cite{grad}. Use equation (\ref{eq:8.4}) and set $k\to -\frac{1}{2},a\to e,m\to \frac{1}{q}-1,b\to \frac{1}{e},d\to 0,j\to 0$ and simplify.
\begin{equation}
\int_0^{\frac{1}{e}} \frac{x^{-1+\frac{1}{q}}}{\sqrt{-\log (e x)}} \, dx=e^{-1/q} \sqrt{q \pi }
\end{equation}
\end{example}
\begin{example}
Derivation of equation (4.275.1) in \cite{grad}. Use equation (\ref{eq:8.4}) and form two equations. The first equation is given when $d\to 0,m\to 0,a\to 1,b\to 1,j\to 0,k\to q-1$ and the second when $k\to 0,c\to 1,z\to -1,d\to 1-q,m\to p-1,b\to 1$ and taking their difference to get;
\begin{multline}
\int_0^1 \left(-(1-x)^{-1+q} x^{-1+p}+(-\log (x))^{-1+q}\right) \, dx=\Gamma (q) \left(1-\frac{\Gamma
   (p)}{\Gamma (p+q)}\right)
\end{multline}
\end{example}
\begin{example}
Derivation of equation (4.261.21) in \cite{grad}. Use equation (\ref{eq:8.4}) and set $k\to 2,a\to 1,b\to 1$ and simplify the series. Then replace $m\to p-1,d\to 1-q,c\to 1,z\to -1$ to get;
\begin{multline}
\int_0^1 (1-x)^{-1+q} x^{-1+p} \log ^2(x) \, dx\\
=\frac{\Gamma (p) \Gamma (q) \left(\psi ^{(0)}(p)^2-2 \psi
   ^{(0)}(p) \psi ^{(0)}(p+q)+\psi ^{(0)}(p+q)^2+\psi ^{(1)}(p)-\psi ^{(1)}(p+q)\right)}{\Gamma (p+q)}
\end{multline}
\end{example}
\begin{example}
Derivation of equation (4.253.1(8)) in \cite{grad}. Use equation (\ref{eq:8.4}) and set $a\to 1,b\to 1,z\to -1,c\to r,m\to u-1,d\to 1-v,k\to 1$ and simplify;
\begin{multline}
\int_0^1 x^{-1+u} \left(1-x^r\right)^{-1+v} \log (x) \, dx=\frac{\Gamma \left(\frac{r+u}{r}\right) \Gamma (v)
   \left(\psi ^{(0)}\left(\frac{u}{r}\right)-\psi ^{(0)}\left(\frac{u}{r}+v\right)\right)}{r u \Gamma \left(\frac{u+r
   v}{r}\right)}
\end{multline}
\end{example}
\begin{example}
Derivation of equation (4.246) in \cite{grad}. Use equation (\ref{eq:8.4}) and set $k\to 1,a\to 1,b\to 1,z\to -1,c\to 2,m\to 0,d\to \frac{1}{2}-n$ and simplify;
\begin{multline}
\int_0^1 \left(1-x^2\right)^{-\frac{1}{2}+n} \log (x) \, dx=\frac{\sqrt{\pi } \Gamma \left(\frac{1}{2} (1+2
   n)\right) \left(\psi ^{(0)}\left(\frac{1}{2}\right)-\psi ^{(0)}(1+n)\right)}{4 \Gamma (1+n)}\\
=-\frac{\pi  (-1+2
   n)\text{!!} \left(H_n+2 \log (2)\right)}{4 (2 n)\text{!!}}
\end{multline}
\end{example}
\begin{example}
Derivation of equation (4.1.5.179) in \cite{brychkov}. Use equation (\ref{eq:8.4}) and set $a\to 1,b\to 1,z\to 1,d\to n+1,m\to a-1,c\to b,k\to m+n$ and then take the first partial derivative with respect to $m$ and simplify to get;
\begin{multline}
\int_0^1 \frac{x^{-1+a} \log ^{m+n}(x) \log \left(\log
   \left(\frac{1}{x}\right)\right)}{\left(1+x^b\right)^{n+1}} \, dx\\
=\sum _{j=0}^{\infty } (-1)^{-m-n} (a+b j)^{-1-m-n}
   \binom{-1-n}{j} \Gamma (1+m+n) (-\log (a+b j)+\psi ^{(0)}(1+m+n))\\
=\frac{(-1)^n}{n!}D_{a}^mD_{b}^n\left[\frac{C+\log(2b)}{2b}\left(\psi ^{(0)}\frac{a}{2b}-\psi ^{(0)}\frac{a+b}{2b}\right)\right. \\ \left.
+\frac{1}{2b}\left(\zeta'\left(1,\frac{a}{2b}\right)-\zeta'\left(1,\frac{a+b}{2b}\right) \right) \right]
\end{multline}
where $Re(a)>0$.
\end{example}
\begin{example}
Derivation of equation (4.256) in \cite{grad}. Use equation (\ref{eq:8.4}) and set $a\to 1,b\to 1,d\to 1-\frac{m}{n},m\to 0,z\to -1,c\to n,k\to 1$ and simplify.
\begin{equation}
\int_0^1 \left(1-x^n\right)^{-1+\frac{m}{n}} \log (x) \, dx=\frac{\Gamma \left(\frac{1}{n}\right) \Gamma
   \left(\frac{m}{n}\right) \left(\psi ^{(0)}\left(\frac{1}{n}\right)-\psi
   ^{(0)}\left(\frac{1+m}{n}\right)\right)}{n^2 \Gamma \left(\frac{1+m}{n}\right)}
\end{equation}
\end{example}
\begin{example}
The Beta function. Use equation (\ref{eq:8.4}) and set $k\to 0,d\to 1-q,m\to p-1,c\to 1,z\to -1,b\to 1$ and simplify.
\begin{equation}
\int_0^1 (1-x)^{-1+q} x^{-1+p} \, dx=\frac{\Gamma (p) \Gamma (q)}{\Gamma (p+q)}
\end{equation}
\end{example}
\begin{example}
Derivation of equation (4.267.12) in \cite{grad}. Use equation (\ref{eq:8.4}) and set $a\to 1,b\to 1,c\to 1,z\to -a,d\to n$. Then form a second equation by replacing $m\to s$ and take their difference with $m\to p-1,s\to q-1$ and simplify.
\begin{equation}
\int_0^1 \frac{x^{-1+p}-x^{-1+q}}{(1-a x)^n \log (x)} \, dx=\sum _{j=0}^{\infty } (-a)^j \binom{-n}{j} \log
   \left(\frac{j+p}{j+q}\right)
\end{equation}
where $|Re(a)|<1,Re(q)>0,Re(p)>0$.
\end{example}
\begin{example}
Derivation of equation (4.267.14) in \cite{grad}. Use equation (\ref{eq:8.4}) and set $a\to 1,b\to 1,z\to 1,c\to 1,d\to 1$ and simplify in terms of the Hurwitz zeta function. Next form a second equation by replacing $m\to s$ and take their difference. Next use l'Hopital's rule as $k\to -1$ and $m\to p-1,s\to q-1$. Next form another equation by replacing $p\to 2 n+p+1,q\to 2 n+q+1$ and take their difference and simplify the log-gamma function.
\begin{multline}
\int_0^1 \frac{\left(1+x^{1+2 n}\right) \left(x^p-x^q\right)}{x (1+x) \log (x)} \, dx=\log \left(\frac{\Gamma
   \left(1+n+\frac{p}{2}\right) \Gamma \left(\frac{1+p}{2}\right) \Gamma \left(\frac{1}{2}+n+\frac{q}{2}\right) \Gamma
   \left(\frac{q}{2}\right)}{\Gamma \left(\frac{p}{2}\right) \Gamma \left(\frac{1}{2} (1+2 n+p)\right) \Gamma
   \left(1+n+\frac{q}{2}\right) \Gamma \left(\frac{1+q}{2}\right)}\right)
\end{multline}
where $Re(q)>0,Re(p)>0$.
\end{example}
\begin{example}
Derivation of equation (4.267.30) in \cite{grad}. Errata. Use equation (\ref{eq:8.4}) and set $a\to 1,b\to 1,d\to 1,z\to -1,c\to p+q+2 s$ and simplify in terms of the Hurwitz zeta function. Next form four equations by replacing $m\to s-1,m\to p+s-1, m\to m\to q+s-1, m\to p+q+s-1$ and add these equations. Next apply l'Hopital's rule as $k\to -1$ and simplify.
\begin{multline}
\int_0^1 \frac{x^{-1+s} \left(1-x^p\right) \left(1-x^q\right)}{\left(1-x^{p+q+2 s}\right) \log (x)} \, dx=\log
   \left(-\frac{\Gamma \left(-\frac{q+s}{p+q+2 s}\right) \Gamma \left(\frac{p+2 q+3 s}{p+q+2 s}\right) \sin
   \left(\frac{\pi  s}{p+q+2 s}\right)}{\pi }\right)\\
\neq2 \log \left(\sin \left(\frac{\pi  s}{p+q+2 s}\right) \csc
   \left(\frac{\pi  (p+s)}{p+q+2 s}\right)\right)
\end{multline}
where $Re(q)>0,Re(p)>0$.
\end{example}
\begin{example}
Derivation of equation (4.267.38) in \cite{grad}. Errata. Use equation (\ref{eq:8.4}) and set $a\to 1,b\to 1,c\to 1,z\to -1,d\to 1$ and simplify in terms of the Hurwitz zeta function. Next form five equations with $m\to 1, m\to 2, m\to p, m\to q, m\to p+q$ and add. Next apply l'Hopital's rule as $k\to -1$ and simplify.
\begin{multline}
\int_0^1 \frac{2 x-x^2-x^p-x^q+x^{p+q}}{(1-x) \log (x)} \, dx=\log \left(\frac{2 p q \Gamma (p) \Gamma
   (q)}{\Gamma (1+p+q)}\right)\neq \log (B(p,q))
\end{multline}
where $Re(q)>0,Re(p)>0$.
\end{example}
\begin{example}
Derivation of equation (4.267.39) in \cite{grad}. Use equation (\ref{eq:8.4}) and set $a\to 1,b\to 1,c\to p,z\to -1,d\to -n,m\to 0$ and simplify. Next apply l'Hopital's rule as $k\to -1$ and simplify.
\begin{multline}
\int_0^1 \frac{\left(-1+x^p\right)^n}{\log (x)} \, dx=\sum _{j=0}^n (-1)^{j-n} \binom{n}{j} \log (1+j p)=\sum
   _{j=0}^n (-1)^{n-j} \binom{n}{n-j} \log (1+j p)
\end{multline}
where $n>0,Re(p)>0$.
\end{example}
\begin{example}
Derivation of equation (3.237) in \cite{grad}. Errata. Use equation (\ref{eq:8.4}) and set $k\to 0,c\to 1,z\to \frac{1}{u},d\to 1,a\to 1,m\to 0$ and simplify. Form to equation with $b\to n,b\to n+1$ and take their difference.
\begin{multline}
\sum _{n=0}^{\infty } (-1)^{1+n} \int_n^{n+1} \frac{1}{x+u} \, dx=\sum _{n=0}^{\infty } (-1)^{1+n} \log
   \left(1+\frac{1}{n+u}\right)=-\log \left(\frac{u \Gamma \left(\frac{u}{2}\right)^2}{2 \Gamma
   \left(\frac{u+1}{2}\right)^2}\right)
\end{multline}
where $Re(u)>0$.
\end{example}
\section{Derivation of entries in Prudnikov Volume I.}
In this section we will derive a few entries in \cite{prud1} and list which entries and Tables can be derived using equation (\ref{eq:8.4}) with parameter substitutions and algebraic methods. The following entries can be derived; 2.6.3.1-2, 2.6.4, 2.6.6, 2.6.7, 2.6.9.1-4,2.6.9.7-12,13,17,19,20,21,22, 2.6.13.1,5-9,15-17,21,22-24, 2.6.14.1,18, 2.6.15.1-3, 10,13-15, 2.6.17.1-7,15-22,24,26-30,39,41-45, 2.6.18, 2.6.19.1-2,2.6.19.4-7,9,10-14, 2.6.20
\begin{example}
Derivation of equation (2.6.3.1) in \cite{prud1}. Use equation (\ref{eq:8.4}) and set $d\to 0,j\to 0,m\to \alpha -1,k\to \sigma ,a\to \frac{1}{a},b\to a$ and simplify.
\begin{equation}
\int_0^{\alpha } x^{-1+\alpha } \log ^{\sigma }\left(\frac{a}{x}\right) \, dx=a^{\alpha } \alpha ^{-1-\sigma }
   \Gamma (1+\sigma )
\end{equation}
where $Re(a)>0,Re(\alpha)>0,Re(\sigma)>0$.
\end{example}
\begin{example}
Derivation of equation (2.6.4.1 ) in \cite{prud1}. Use equation (\ref{eq:8.4}) and set $a\to \frac{1}{a},c\to \mu ,z\to \alpha ^{-\mu },d\to \rho ,k\to \sigma ,m\to \alpha -1,b\to a$ and simplify.
\begin{equation}
\int_0^a \frac{x^{-1+\alpha } \log ^{\sigma }\left(\frac{a}{x}\right)}{\left(\alpha ^{\mu }+x^{\mu
   }\right)^{\rho }} \, dx=a^{\alpha } \alpha ^{-\mu  \rho } \Gamma (1+\sigma ) \sum _{j=0}^{\infty }
   \frac{\left(\frac{a}{\alpha }\right)^{j \mu } \binom{-\rho }{j}}{(\alpha +j \mu )^{\sigma +1}}
\end{equation}
where $Re(a)>0,Re(\alpha)>0,Re(\sigma)>0$.
\end{example}
\begin{example}
Derivation of equation (2.6.5.2). Errata. in \cite{prud1}. Use equation (\ref{eq:8.4}) and set $m\to \alpha -1,d\to -m,c\to \mu ,z\to -\alpha ^{-\mu },a\to \frac{1}{a},k\to n,b\to a$ and simplify.
\begin{multline}
\int_0^a x^{-1+\alpha } \left(\alpha ^{\mu }-x^{\mu }\right)^m \log ^n\left(\frac{x}{a}\right) \, dx\\
=\alpha
   ^{\mu  m} \sum _{j=0}^{\infty } (-1)^{-n} a^{\alpha } \left(-a^{\mu } \alpha ^{-\mu }\right)^j (\alpha +j \mu
   )^{-1-n} \binom{m}{j} \Gamma (1+n)\\
\neq(-1)^n n! a^{\alpha +\mu  m} \sum _{k=0}^m \frac{\binom{m}{k} (-1)^k}{(\alpha
   +\mu  k)^{n+1}}
\end{multline}
where $Re(a)>0,Re(u)>0$.
\end{example}
\begin{example}
Derivation of equation (2.6.19.5) in \cite{prud1}, generalized form over complex range for the parameters.. Use equation (\ref{eq:8.4}) and set $a=b=c=d=1,k\to 2n$. Then take the indefinite integral of both sides with respect to $z$ and set $m=0$ simplify the series in terms of the polylogarithm function. The Bernoulli form can be derived by using equations [Wolfram,\href{https://mathworld.wolfram.com/HurwitzZetaFunction.html}{9}].
\begin{equation}
\int_0^1 \frac{\log ^{2 n}(x) \log (1+x z)}{x} \, dx=-(-1)^{2 n} \Gamma (1+2 n) \text{Li}_{2+2 n}(-z)
\end{equation}
where $Re(n)>-1$.
\end{example}
\begin{example}
Extended Nielsen definite integral \cite{nielsen}. Use equation (\ref{eq:8.4}) and take the $n$-th partial derivative with respect to $d$ and simplify the right-hand side derivative using equation [Wolfram,\href{http://functions.wolfram.com/06.03.20.0007.02}{02}] and [Wolfram,\href{http://functions.wolfram.com/06.10.20.0007.01}{01}].
\begin{multline}
\int_0^b \frac{x^m \left(\log ^k\left(\frac{1}{a x}\right) \log ^n\left(1+x^c z\right)\right)}{\left(1+x^c z\right)^d} \, dx\\
=\frac{1}{a^{1+m}}\sum _{j=0}^{\infty } \sum _{l=1}^j \frac{(-1)^{j+l} z^j
   \Gamma \left(1+k,(1+c j+m) \log \left(\frac{1}{a b}\right)\right) (1+l-n)_n S_j^{(l)}}{a^{c j} (1-d-j)^{n-l} (1+c j+m)^{k+1} j!}
\end{multline}
where $n>0,Re(m)>0,Re(c)\geq 0$.
\end{example}
\begin{example}
Prime counting function in terms of a definite integral and the incomplete gamma function. See  [Wolfram,\href{https://mathworld.wolfram.com/PrimeCountingFunction.html}{1}] and 
[Wolfram,\href{https://mathworld.wolfram.com/PrimeCountingFunction.html}{13}]. Use equation (\ref{eq:8.4}) and set $c=0,a=1,k=-1,m=0$ and form a second equation by replacing $b\to 2$ and take their difference. We also look at a plot of the closed form (blue) versus the built-in Wolfram Mathematica PrimePi function.
\begin{equation}
\int_{2}^{b} \frac{1}{\log (x)} \, dx=\Gamma (0,-\log (2))-\Gamma (0,-\log(b))
\end{equation}
where $Re(b)>2$.
\end{example}
\begin{figure}[H]
\caption{Plot of $\Gamma (0,-\log (2))-\Gamma (0,-\log(b))$}
\includegraphics[width=8cm]{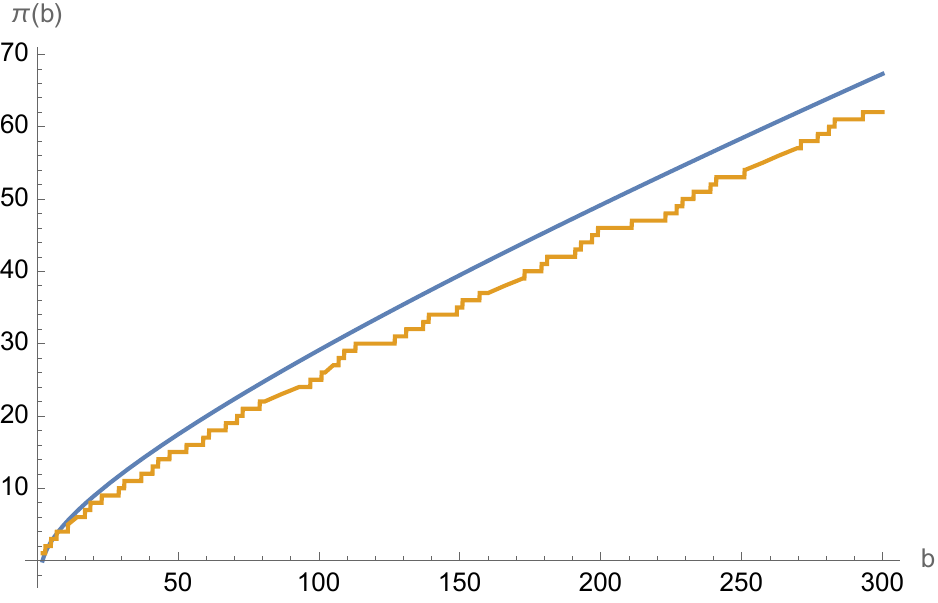}
\end{figure}
%
%
%
%
\section{Definite integrals of the product of generalized rational and logarithmic functions}
The main theorem used in this section has sub-functions after parameter substitutions which yield closed forms equivalent to elliptic functions. A few example will be evaluated using this theorem to illustrate the procedure for derivation of the examples listed. The entries in Gradshteyn and Ryzhik which can be derived using this theorem are; 3.121, 3.143, 3.146, 3.152, 3.153, 3.154, 3.155, 3.156, 3.158, 3.159, 3.161, 3.162, 3.163, 3.164, 3.166, 3.169, 3.171, 3.172, 3.173, 3.181, 3.182, 3.183, 3.184, 3.185, 3.191, 3.196, 3.197, 3.198, 3.212, 3.213, 3.214, 3.215, 3.223, 3.224, 3.227, 3.254, 3.259, 4.234, 4.267.43(10),  4.268. Equation (\ref{eq:10.4}) can also be used to derive power series representations for Inverse Jacobi Functions [Wolfram,\href{https://functions.wolfram.com/EllipticFunctions/}{1}], where previous work on power series was conducted by Carlson (2008)  \cite{carlson}.
The contour integral representation used to derive the main theorem in this section is given by;
\begin{multline}\label{eq:10.1}
\frac{1}{2\pi i}\int_{0}^{b}\int_{C}\sum_{j=0}^{\infty}\sum_{l=0}^{\infty}a^w z^j w^{-k-1} \alpha ^l \binom{-d}{j} \binom{-f}{l} x^{c j+l s+m+w}dwdx\\
=\frac{1}{2\pi i}\int_{C}\sum_{j=0}^{\infty}\sum_{l=0}^{\infty}\frac{a^w z^j w^{-k-1} \alpha ^l
   \binom{-d}{j} \binom{-f}{l} b^{c j+l s+m+w+1}}{c j+l s+m+w+1}dw
\end{multline}
where $0< Re(b)$.
\subsection{Left-hand side contour integral representation}
Use equation (\ref{intro:cauchy}) and replace $y\to \log(ax)$ and multiply both sides by $z^j \alpha ^l \binom{-d}{j} \binom{-f}{l} x^{c j+l s+m}$ to get;
\begin{multline}\label{eq:10.2}
\frac{1}{2\pi i}\int_{0}^{1}\int_{C}\frac{z^j \alpha ^l \binom{-d}{j} \binom{-f}{l} \log ^k(a x) x^{c j+l s+m}}{k!}dwdx\\
=\frac{1}{2\pi i}\int_{0}^{1}\int_{C}z^j w^{-k-1} \alpha ^l (a x)^w
   \binom{-d}{j} \binom{-f}{l} x^{c j+l s+m}dwdx\\
   =\frac{1}{2\pi i}\int_{C}\int_{0}^{1}z^j w^{-k-1} \alpha ^l (a x)^w
   \binom{-d}{j} \binom{-f}{l} x^{c j+l s+m}dwdx
\end{multline}
where $Re(b)>0$.
\subsection{Right-hand side contour integral representation}
Use equation (\ref{eq:4.5}) and set $\beta=\gamma=0$ and replace $m\to c j+l s+m+1,a\to a b$ and multiply both sides by $z^j \alpha ^l \binom{-d}{j} \binom{-f}{l} b^{c j+l s+m+1}$ to get;
\begin{multline}\label{eq:10.3}
\frac{z^j \alpha ^l \binom{-d}{j} \binom{-f}{l} b^{c j+l s+m+1}  (a b)^{-c j-l s-m-1}
   \Gamma (k+1,-((c j+m+l s+1) \log (a b)))}{k!(-c j-l s-m-1)^{k+1}}\\
=-\frac{1}{2\pi i}\int_{C}\frac{z^j w^{-k-1} \alpha ^l (a b)^w \binom{-d}{j} \binom{-f}{l}
   b^{c j+l s+m+1}}{c j+l s+m+w+1}dw
\end{multline}
where $Re(b)>0$.
\begin{theorem}
For all $Re(a)>0,Re(b)>0$ then,
\begin{multline}\label{eq:10.4}
\int_0^b x^m \left(1+x^c z\right)^{-d} \left(1+x^s \alpha \right)^{-f} (-\log (a x))^k \, dx\\
=\sum
   _{j=0}^{\infty } \sum _{l=0}^{\infty } \frac{\left(a^{-1-c j-m-l s} z^j \alpha ^l\right) \binom{-d}{j}
   \binom{-f}{l} \Gamma (1+k,-((1+c j+m+l s) \log (a b)))}{(1+c j+m+l s)^{k+1}}
\end{multline}
\end{theorem}
\begin{proof}
We use equations (\ref{eq:10.2}) and (\ref{eq:10.3}) and take the infinite sums over $j\in[0,\infty),l\in[0,\infty)$ and note the right-hand sides are equal relative to equation (\ref{eq:10.1}). We are able to switch the order of integration over $x$ and $w$ using Fubini's theorem for multiple integrals see page 178 in \cite{gelca}, since the integrand is of bounded measure over the space $\mathbb{C} \times [0,b]$. We are able to switch the order of integration and summations over $w$ using Tonellii's theorem for  integrals and sums see page 177 in \cite{gelca}, since the summand is of bounded measure over the space $\mathbb{C} \times [0,\infty) \times [0,\infty)$. Simplifying the sums and gamma function will yield the stated result.
\end{proof}
\subsection{Definite integrals in terms of elliptic functions}
In this section we derive definite integrals over a finite domain in terms of elliptic functions and infinite series involving the hypergeometric function. We derive a few examples to show the parameter substitution and algebraic method involved with many other examples listed for similar evaluation. The main theorem in this section can be used derived entries 3.138 in \cite{grad}. Complete elliptic integrals find applications in various fields, including geometry, physics, mechanics, electrodynamics, statistical mechanics, astronomy, geodesy, geodesics on conics, and magnetic field calculations [Wolfram,\href{https://functions.wolfram.com/EllipticIntegrals/EllipticPi/introductions/CompleteEllipticIntegrals/ShowAll.html}{Applications}].
\begin{example}
Derivation of equation (2.2.5.18) in \cite{prud1} in terms of the complete elliptic integral $E(m)$. Use equation (\ref{eq:10.4}) and set $k\to 0,m\to 0,a\to 1,b\to 1,d\to -\frac{1}{2},f\to \frac{1}{2},c\to 2,\alpha \to 1,s\to 2,z\to -z^2$ and simplify.
\begin{equation}
-i E\left(i \sinh ^{-1}(1)|-z^2\right)=\sum _{l=0}^{\infty } \frac{\binom{-\frac{1}{2}}{l} \,
   _2F_1\left(-\frac{1}{2},\frac{1}{2}+l;\frac{3}{2}+l;z^2\right)}{1+2 l}
\end{equation}
where $0< Re(z)<1$.
\end{example}
\begin{example}
Derivation of equation (2.2.5.19) in \cite{prud1}. Use equation (\ref{eq:10.4}) and set $k\to 0,a\to 1,d\to \frac{1}{2},f\to 1,z\to \frac{1}{a},c\to 1,s\to 1,\alpha \to k^2,b\to a,m\to
   \frac{1}{2}$ and simplify.
   \begin{multline}
\int_0^a \frac{\sqrt{\frac{x}{a+x}}}{1+k^2 x} \, dx=\sum _{l=0}^{\infty } \frac{2 (-1)^l a^{1+l} k^{2 l} \,
   _2F_1\left(\frac{1}{2},\frac{3}{2}+l;\frac{5}{2}+l;-1\right)}{3+2 l}\\
=\frac{\sqrt{2 \left(1+k^2 a\right)}-2
   E\left(\frac{\pi }{4}|\sqrt{1-k^2 a}\right)}{k^2}
\end{multline}
where $0< Re(k)<1,Re(a)>0$.
\end{example}
\begin{example}
Derivation of  Elliptic Integral of the First Kind. Use equation (\ref{eq:10.4}) and set $k\to 0,m\to 0,d\to \frac{1}{2},f\to \frac{1}{2},c\to 2,s\to 2,\alpha \to -1,z\to -k^2,a\to 1,b\to \sin (\phi
   )$ and simplify.
   \begin{multline}
\int_0^{\sin (\phi )} \frac{1}{\sqrt{1-x^2} \sqrt{1-k^2 x^2}} \, dx\\
=\sum
   _{l=0}^{\infty } \frac{(-1)^l \binom{-\frac{1}{2}}{l} \,
   _2F_1\left(\frac{1}{2},\frac{1}{2}+l;\frac{3}{2}+l;k^2 \sin ^2(\phi )\right)
   \sin ^{1+2 l}(\phi )}{1+2 l}\\
=F\left(\phi \left|k^2\right.\right)
\end{multline}
where $0< Re(k)<1,|Re(\phi)|<\pi/2$.
\end{example}
\begin{example}
Derivation of equation (4.242.1) in \cite{grad}. Use equation (\ref{eq:10.4}) and set $m\to 0,k\to 1,a\to 1,d\to \frac{1}{2},c\to 2,f\to \frac{1}{2},s\to 2,\alpha \to -b^2,z\to a^2,b\to
   \frac{1}{b}$ and simplify.
   \begin{multline}
\int_0^b \frac{\log (x)}{\sqrt{1+a^2 x^2} \sqrt{1-b^2 x^2}} \, dx\\
=-\sum _{j=0}^{\infty } \sum _{l=0}^{\infty }
   a^{2 j} (-1)^l b^{2 l} \binom{-\frac{1}{2}}{j} \binom{-\frac{1}{2}}{l} E_{-1}(-((1+2 j+2 l) \log (b))) \log
   ^2(b)
\end{multline}
where $0< Re(b)<1,Re(a)>0$.
\end{example}
\begin{example}
Derivation of Example 1 in in \cite{paul}. Use equation (\ref{eq:10.4}) and set $d\to -\frac{1}{2},f\to \frac{1}{2},k\to 0,a\to 1,m\to 0,c\to 2,s\to 2,z\to \alpha ^2,\alpha \to \beta
   ^2$ and simplify.
   \begin{multline}
E\left(i \sinh ^{-1}(b \beta )|\frac{\alpha ^2}{\beta ^2}\right)=\sum _{l=0}^{\infty } \frac{i b^{1+2 l} \beta
   ^{1+2 l} \binom{-\frac{1}{2}}{l} \, _2F_1\left(-\frac{1}{2},\frac{1}{2}+l;\frac{3}{2}+l;-b^2 \alpha ^2\right)}{1+2
   l}
\end{multline}
where $0< Re(\beta)<1,Re(\alpha)>0$.
\end{example}
\begin{example}
Derivation of equation (3.138.1 ) in \cite{grad}. Use equation (\ref{eq:10.4}) and set $d\to \frac{1}{2},f\to \frac{1}{2},m\to -\frac{1}{2},k\to 0,a\to 1,c\to 1,s\to 1,z\to -1,\alpha \to -k^2,b\to
   u$ and simplify.
   \begin{multline}
\int_0^u \frac{1}{\sqrt{x (1-x) \left(1-k^2 x\right)}} \, dx=\sum _{l=0}^{\infty } \frac{2 (-1)^l k^{2 l}
   u^{\frac{1}{2}+l} \binom{-\frac{1}{2}}{l} \, _2F_1\left(\frac{1}{2},\frac{1}{2}+l;\frac{3}{2}+l;u\right)}{1+2 l}\\
=2 F\left(\sin ^{-1}\left(\sqrt{u}\right)|k^2\right)
\end{multline}
where $0< Re(\beta)<1,Re(\alpha)>0$.
\end{example}
\begin{example}
Derivation of equation (4.267.43(10) ) in \cite{grad}. Use equation (\ref{eq:10.4}) and set $b\to 1,a\to 1,z\to -1,\alpha \to -1,c\to p,s\to q,d\to -n,f\to -m,m\to r-1$ and apply l'Hopital's rule as $k\to -1$ and simplify.
\begin{multline}
\int_0^1 \frac{x^{-1+r} \left(1-x^p\right)^n \left(1-x^q\right)^m}{\log (x)} \, dx=\sum _{j=0}^n \sum _{l=0}^m
   (-1)^{j+l} \binom{m}{l} \binom{n}{j} \log (j p+l q+r)
\end{multline}
where $Re(r)>0$.
\end{example}
\begin{example}
Derivation of equation (3.121.1) in \cite{grad}. Use equation (\ref{eq:10.4}) and set $k\to 0,a\to 1,d\to 1,f\to 1,c\to 1,s\to 1,z\to -e^{-i t},\alpha \to -e^{i t},m\to -\frac{1}{2},b\to 1$ and simplify.
\begin{multline}
\int_0^1 \frac{1}{\sqrt{x} \left(1+x^2-2 x \cos (t)\right)} \, dx\\
=-i e^{\frac{i t}{2}} \left(-\coth
   ^{-1}\left(e^{\frac{i t}{2}}\right)+e^{i t} \tanh ^{-1}\left(e^{\frac{i t}{2}}\right)\right) (-i+\cot (t))\\
=2 \csc
   (t) \sum _{k=1}^{\infty } \frac{\sin (k t)}{2 k-1}
\end{multline}
where $Re(t)<\pi/2$.
\end{example}
\begin{example}
Derivation of equation (3.138.4) in \cite{grad}. Use equation (\ref{eq:10.4}) and set $k\to 0,a\to 1,m\to -\frac{1}{2},c\to 1,z\to 1,s\to 1,\alpha \to k^2,d\to \frac{1}{2},f\to \frac{1}{2},b\to
   u$ and simplify.
   \begin{multline}
\int_0^u \frac{1}{\sqrt{x} \sqrt{1+x} \sqrt{1+k^2 x}} \, dx=\sum _{l=0}^{\infty } \frac{2 k^{2 l}
   u^{\frac{1}{2}+l} \binom{-\frac{1}{2}}{l} \, _2F_1\left(\frac{1}{2},\frac{1}{2}+l;\frac{3}{2}+l;-u\right)}{1+2
   l}\\
=2 F\left(\tan ^{-1}\left(\sqrt{u}\right)|1-k^2\right)
\end{multline}
where $0< Re(k)<1,0< Re(u)<1$.
\end{example}
\begin{example}
Derivation of equation (2.6.10.4) in \cite{prud1}. Use equation (\ref{eq:10.4}) and set $k\to 0,a\to 1,f\to 1,d\to \rho ,s\to 1,c\to 1,z\to \frac{1}{a}$ then take the indefinite integral with respect to $\alpha$ and simplify.
\begin{multline}
\int_0^b \frac{x^{-1+m} \log (a+x)}{(a+x)^{\rho }} \, dx=\frac{a^{-\rho
   } b^m \, _2F_1\left(m,\rho ;1+m;-\frac{b}{a}\right) \log (a+b)}{m}\\
-\sum
   _{j=0}^{\infty } \frac{a^{-1-j-\rho } b^{1+j+m} \Gamma (1-\rho ) \,
   _2F_1\left(1,1+j+m;2+j+m;-\frac{b}{a}\right)}{(j+m) (1+j+m) \Gamma (1+j)
   \Gamma (1-j-\rho )}
\end{multline}
where $0< Re(m)<1,0< Re(\rho)>0$.
\end{example}
\begin{example}
Derivation of equation (2.6.13.14) in \cite{prud1}. Use equation (\ref{eq:10.4}) and set $k\to 0,a\to 1,f\to \frac{1}{2},d\to 1,s\to 2,c\to 1,\alpha \to -\frac{1}{a^2}$ and take the definite integral with respect to $z\in[-z,z]$ and simplify.
\begin{multline}\label{eq:10.15}
\int_0^b \frac{x^{-1+m} \tanh ^{-1}(x z)}{\sqrt{a^2-x^2}} \, dx=\sum _{j=0}^{\infty } \sum _{l=0}^{\infty }
   \frac{(-1)^l a^{-2 l} b^{1+j+2 l+m} z \left((-z)^j+z^j\right) \binom{-1}{j} \binom{-\frac{1}{2}}{l}}{2 a (1+j)
   (1+j+2 l+m)}
\end{multline}
where $-1< Im(a)<1,Re(a)>1$.
\end{example}
\begin{example}
Use equation (\ref{eq:10.15}) and set $m=1,z\to b$ and simplify.
\begin{multline}
\int_0^b \frac{\tanh ^{-1}(b x)}{\sqrt{a^2-x^2}} \, dx=2 \left(\sin ^{-1}\left(\frac{b}{a}\right)-\sin
   ^{-1}\left(\frac{\sqrt{1-\sqrt{1-\frac{b^2}{a^2}}}}{\sqrt{2}}\right)\right) \tanh ^{-1}\left(b^2\right)\\
+\sum
   _{l=0}^{\infty } \frac{(-1)^l a^{-1-2 l} b^{3+2 l} \sqrt{\pi } \, _2F_1\left(1,1+l;2+l;b^4\right)}{(1+2 l)^2
   \Gamma \left(-\frac{1}{2}-l\right) \Gamma (2+l)}
\end{multline}
where $-1< Im(a)<1,Re(a)>1$.
\end{example}
\begin{example}
Derivation of equation (2.6.13.14) in \cite{prud1}. Use equation (\ref{eq:10.15}) and set $b\to a$ and simplify. 
\begin{multline}
\int_0^a \frac{\tanh ^{-1}(a x)}{\sqrt{(a-x) (a+x)}} \, dx=\frac{1}{2} \pi  \tanh ^{-1}\left(a^2\right)+\sum
   _{l=0}^{\infty } \frac{(-1)^l a^2 \sqrt{\pi } \, _2F_1\left(1,1+l;2+l;a^4\right)}{(1+2 l)^2 \Gamma
   \left(-\frac{1}{2}-l\right) \Gamma (2+l)}
\end{multline}
where $-1< Im(a)<1,Re(a)>1$.
\end{example}
\begin{example}
Power series derivation for equation 145(35) in \cite{bdh}. Use equation (\ref{eq:10.4}) and set $d\to \frac{1}{2},f\to \frac{1}{2},c\to 2,s\to 2,k\to 1,a\to 1,m\to 0,z\to -z^2,\alpha \to -\alpha ^2$. Next form a second equation by replacing $b\to q$ and take their difference. Then replace $z\to p,\alpha \to p$ and simplify.
\begin{multline}
\int_p^q \frac{\log (x)}{\sqrt{1-p^2 x^2} \sqrt{1-q^2 x^2}} \, dx\\
=\sum _{j=0}^{\infty } \sum _{l=0}^{\infty }
   \left(-p^2\right)^j \left(-q^2\right)^l \binom{-\frac{1}{2}}{j} \binom{-\frac{1}{2}}{l} \left(E_{-1}(-((1+2 j+2 l)
   \log (p))) \log ^2(p)\right. \\ \left.
-E_{-1}(-((1+2 j+2 l) \log (q))) \log ^2(q)\right)
\end{multline}
where $Re(p)>0,Re(q)>0$.
\end{example}
\begin{example}
Derivation of equation 145(5) in \cite{bdh}. Use equation (\ref{eq:10.4}) and set $d\to 0,f\to 0,z\to 1,\alpha \to 1,c\to 0,s\to 0$ and replace $a\to e^a$. Next take the first partial derivative with respect to $k$ and set $k=0$. Next set $a=1/2,b=e^{-1}$. Take the indefinite integral with respect to $m$ and set $m\to 2q-1$ and simplify.
\begin{equation}
\int_0^{\frac{1}{e}} \frac{x^{2 q-1} \log \left(2 \log \left(\frac{1}{x}\right)-1\right)}{\log (x)} \,
   dx=-\frac{1}{2}E_1(q){}^2
\end{equation}
where $Re(q)>0$.
\end{example}
\begin{example}
Extended form of equation 123(2) in \cite{bdh}. Use equation (\ref{eq:10.4}) and set $a\to 1,b\to 1,d\to -p,f\to -q,c\to 1,s\to 1,k=-1$ then form a second equation by replacing $m\to s$ and their difference.
\begin{multline}
\int_0^1 \frac{\left(-x^m+x^s\right) (1+x z)^p (1+x \alpha )^q}{\log (x)} \, dx=\sum _{j=0}^{\infty } \sum
   _{l=0}^{\infty } z^j \alpha ^l \binom{p}{j} \binom{q}{l} \log \left(\frac{1+j+l+m}{1+j+l+s}\right)
\end{multline}
where $Re(p)>0,Re(q)>0$.
\end{example}
\begin{example}
Derivation equation 33(10)  in \cite{bdh}. Use equation (\ref{eq:10.4}) and set $d\to 0,f\to 0,a\to e,c\to 0,s\to 0,z\to 1,\alpha \to 1$ then take the first partial derivative with respect to $m$ and set $m=0$. Next form two equations by replacing $b=e,b=1$ and take their difference and set $k=-2$ and simplify.
\begin{equation}
\int_1^e \frac{\log (x)}{(1+\log (x))^2} \, dx=\frac{1}{2} (-2+e)
\end{equation}
\end{example}
\begin{example}
Derivation of (4.1.5.7) in \cite{brychkov}. Use equation (\ref{eq:10.4}) and set $d\to 1,c\to 1,z\to 1,f\to 1,k\to 0,a\to 1,b\to 1$. Next take the indefinite integral with respect to $\alpha$ and set $s=2+\sqrt{3}$. Then set $m=2+\sqrt{3},\alpha=1$ and simplify. Note entries (4.1.5.8-15) can be derived in a similar manner with parameter substitutions and algebraic methods being applied to equation (\ref{eq:10.4}).
\begin{multline}\label{eq:10.22}
\int_0^1 \frac{\log \left(1+x^{2+\sqrt{3}}\right)}{1+x} \, dx\\
=\sum _{j=0}^{\infty } \frac{(-1)^j \left(-\psi
   ^{(0)}\left(\frac{1}{2} \left(3+\sqrt{3}+\left(2+\sqrt{3}\right) j\right)\right)+\psi ^{(0)}\left(\frac{1}{2}
   \left(4+\sqrt{3}+\left(2+\sqrt{3}\right) j\right)\right)\right)}{2 (1+j)}\\
=\frac{1}{12} \pi ^2
   \left(1-\sqrt{3}\right)+\log (2) \log \left(1+\sqrt{3}\right)
\end{multline}
\end{example}
\begin{example}
Derivation of  (4.1.5.7) in \cite{brychkov}.  Repeat the steps in equation (\ref{eq:10.22}) with $s=3+\sqrt{8},\alpha=1$ and simplify.
\begin{multline}
\int_0^1 \frac{\log \left(1+x^{3+2 \sqrt{2}}\right)}{1+x} \, dx\\
=\sum _{l=0}^{\infty } \frac{(-1)^l \left(-\psi
   ^{(0)}\left(\frac{1}{2} \left(4+2 \sqrt{2}+\left(3+2 \sqrt{2}\right) l\right)\right)+\psi ^{(0)}\left(\frac{1}{2}
   \left(5+2 \sqrt{2}+\left(3+2 \sqrt{2}\right) l\right)\right)\right)}{2 (1+l)}\\
=\frac{1}{24} \pi ^2
   \left(3-\sqrt{32}\right)+\frac{1}{2} \log (2) \left(\log (2)+\frac{3}{2} \log
   \left(3+\sqrt{8}\right)\right)
\end{multline}
\end{example}
\begin{example}
Derivation of Euler's first integral, see page 146 in \cite{hymers} page 146. Use equation (\ref{eq:10.4}) and set $k\to 0,f\to 0,s\to 0,\alpha \to 1,z\to -1,c\to n,d\to 1-\frac{q}{n},m\to p-1,a\to 1$ and simplify the series with $b=1$ and simplify.
\begin{equation}
\int_0^1 x^{-1+p} \left(1-x^n\right)^{-1+\frac{q}{n}} \, dx=\frac{\Gamma \left(1+\frac{p}{n}\right) \Gamma
   \left(\frac{q}{n}\right)}{p \Gamma \left(\frac{p}{n}+\frac{q}{n}\right)}
\end{equation}
where $Re(p)>0,Re(q)>0$.
\end{example}
\begin{example}
Derivation of equation (4.2.6) in \cite{nahin}. Use equation (\ref{eq:10.4}) and set $k\to 0,f\to 0,s\to 0,\alpha \to 1,z\to -1,c\to 1,d\to -n,a\to 1,m\to n,b\to 1$ and simplify.
\begin{equation}
\int_0^1 (1-x)^n x^n \, dx=\frac{2^{-1-2 n} \sqrt{\pi } \Gamma (2+n)}{(1+n) \Gamma \left(\frac{1}{2} (3+2
   n)\right)}=\frac{(n!)^2}{(2 n+1)!}
\end{equation}
where $Re(n)>0$.
\end{example}
\begin{example}
Derivation of equation  (4.5) in \cite{nahin}, challenge problem (C4.1). Use equation (\ref{eq:10.4}) and set $k\to 0,f\to 0,s\to 0,\alpha \to 1,z\to -1,c\to \frac{1}{2},d\to -n,a\to 1,m\to 0,b\to 1$ and simplify.
\begin{equation}
\int_0^1 \left(1-\sqrt{x}\right)^n \, dx=\frac{2}{(1+n) (2+n)}
\end{equation}
where $Re(n)>0$.
\end{example}
\begin{example}
Derivation of equation (5.1.2) in \cite{nahin}. Use equation (\ref{eq:10.4}) and set $k\to 1,a\to 1,f\to 0,\alpha \to 1,s\to 0,m\to 0,d\to 1,b\to 1,z\to 1,c\to 2$ and simplify.
\begin{equation}
\int_0^1 \frac{\log (x)}{1+x^2} \, dx=-C
\end{equation}

\end{example}
\begin{example}
Derivation of equation (5.5) in \cite{nahin}, extended form. Use equation (\ref{eq:10.4}) and set $m\to 0,d\to -1,f\to 1,k\to 0,a\to 1,c\to m,s\to n,b\to 1$ and simplify.
\begin{equation}
\int_0^1 \frac{1+x^m z}{1+x^n \alpha } \, dx=\frac{\Phi \left(-\alpha ,1,\frac{1}{n}\right)+z \Phi
   \left(-\alpha ,1,\frac{1+m}{n}\right)}{n}
\end{equation}
where $Re(m)>0,Re(n)>0,Re(z)>0,Re(\alpha)>0$.
\end{example}
\begin{example}
Derivation of Extended form of a powerful elementary integral equation (3.1) in \cite{valean}. Use equation (\ref{eq:10.4}) and set $k\to 0,a\to 1,d\to 1,f\to \frac{1}{2},s\to 2,\alpha \to -1,c\to 1,b\to 1$ and simplify.
\begin{multline}
\int_0^1 \frac{x^m}{\sqrt{1-x^2} (1+x z)} \, dx=\frac{\sqrt{\pi } \Gamma \left(\frac{1}{2}+\frac{m}{2}\right) \,
   _2F_1\left(1,\frac{1+m}{2};\frac{2+m}{2};z^2\right)}{m \Gamma \left(\frac{m}{2}\right)}\\
-\frac{\sqrt{\pi } z \Gamma
   \left(1+\frac{m}{2}\right) \, _2F_1\left(1,\frac{2+m}{2};\frac{3+m}{2};z^2\right)}{2 \Gamma
   \left(\frac{3}{2}+\frac{m}{2}\right)}
\end{multline}
where $Re(m)>0,Re(z)>0$.
\end{example}
\begin{example}
Derivation of equation (1.32) in \cite{nahin} over the complex plane. An infinite series is also derived from a double series two ways. Use equation (\ref{eq:10.4}) and set $k\to 0,a\to 1,d\to 1,f\to m,c\to 1,z\to 1,m\to 2 m,s\to 2,\alpha \to 1$ and compare when the infinite sums are reversed.
\begin{multline}
\sum _{l=0}^{\infty } b^{2 l} \binom{-m}{l} \Phi (-b,1,1+2 l+2 m)=\sum
   _{j=0}^{\infty } \frac{(-1)^j b^j \,
   _2F_1\left(m,\frac{1}{2}+\frac{j}{2}+m;\frac{3}{2}+\frac{j}{2}+m;-b^2\right)
   }{1+j+2 m}
\end{multline}
where $Re(m)>0,|Im(b)|\leq 1$.
\end{example}
\begin{example}
Derivation of generalized form of equation (1.10) in \cite{nahin}.. Use equation (\ref{eq:10.4}) and set $d\to 1,f\to 1,c\to 1,s\to 1,a\to 1,b\to 1$ and take the indefinite integral over $z$ and $\alpha$ then set $m=2$ and simplify.
\begin{multline}
\int_0^1 \log ^k\left(\frac{1}{x}\right) \log (1+x z) \log (1+x \alpha ) \, dx\\
=z \alpha  \Gamma (k+1) \sum
   _{j=0}^{\infty } \sum _{l=0}^{\infty } \frac{(-1)^{j+l} z^j \alpha ^l}{(1+j) (1+l) (3+j+l)^{k+1}}
\end{multline}
where $k,z,\alpha\in\mathbb{C}$.
\end{example}
\begin{example}
Derivation of Example (17) page 486 in \cite{duncan}. Use equation (\ref{eq:10.4}) and set $d\to n,f\to n-\frac{1}{2},z\to \frac{1}{z},c\to 1,s\to 2,\alpha \to -1,k\to 0,a\to 1,m\to 0$ and simplify. Next replace $z\to -\frac{a^2+1}{2 a}$. Next form a second equation by replacing $b\to -1$ and take their difference and simplify.
\begin{multline}
\int_{-1}^1 \frac{\left(1-x^2\right)^{\frac{1}{2}-n}}{\left(1+a^2-2 a x\right)^n} \, dx=\begin{cases}
			-\frac{(-1)^{-2 n}
   \sqrt{\pi } \Gamma \left(\frac{1}{2} (3-2 n)\right) \, _2F_1\left(\frac{1}{2}+\frac{n}{2},\frac{n}{2};2-n;\frac{4
   a^2}{\left(1+a^2\right)^2}\right)}{\Gamma (2-n) \left(1+a^2\right)^n}, & \text{if $n \leq -m/2$}\\
            \frac{\sqrt{\pi } (-1)^{-2 n} \left(a^2+1\right)^{-n} \Gamma \left(\frac{1}{2} (3-2 n)\right) \,
   _2F_1\left(\frac{n}{2}+\frac{1}{2},\frac{n}{2};2-n;\frac{4 a^2}{\left(a^2+1\right)^2}\right)}{\Gamma (2-n)}, & \text{if $n <0$}
		 \end{cases}
\end{multline}
where $Re(a)<1,m\in\mathbb{Z}$.
\end{example}
\begin{example}
Derivation of Extended Nielsen-Bowman integral \cite{bowman}. Use equation (\ref{eq:10.4}) and set $a\to 1,b\to 1,d\to 0,z\to 1,c\to 0,f\to \frac{1}{2}$ and simplify.
\begin{equation}
\int_0^1 \frac{x^m \log ^k\left(\frac{1}{x}\right)}{\sqrt{1+x^s \alpha }} \, dx=\Gamma (1+k) \sum
   _{l=0}^{\infty } \frac{\alpha ^l \binom{-\frac{1}{2}}{l}}{(1+m+l s)^{k+1}}
\end{equation}
where $|Re(m)|<1$.
\end{example}
\begin{example}
Extended form for Nielsen Generalized Polylogarithm see [Wolfram,\href{https://mathworld.wolfram.com/NielsenGeneralizedPolylogarithm.html}{1}]. Use equation (\ref{eq:10.4}) and set $a\to 1,b\to 1,d\to 0,z\to 1,c\to 0,s\to 1$. Then take the $h$-th derivative of both sides with respect to $f$ and simplify using [Wolfram,\href{http://functions.wolfram.com/06.03.20.0007.02}{02}]. Then replace $n\to -f,k\to l,m\to h$ and simplify,
\begin{multline}
\int_0^1 x^m (1+x \alpha )^{-f} (-\log (x))^k \log ^h(1+x \alpha ) \, dx\\
=\sum _{l=0}^{\infty }
   \frac{(1+l+m)^{-1-k} \alpha ^l \Gamma (1+k)\left((-f-l+1)_l\right){}^h}{l!}
\end{multline}
where $Re(\alpha)>1,h\geq 0$.
\end{example}
\begin{example}
Derivation of equation () in \cite{}. Use equation (\ref{eq:10.4}) and set 
\end{example}
\begin{example}
Derivation of an extended Bateman integral \cite{bateman}. Use equation (\ref{eq:10.4}) and set $k\to 0,a\to 1,c\to 1,s\to 1,z\to \frac{1}{z},\alpha \to \frac{1}{\alpha }$ next replace $z\to -z,\alpha\to -\alpha, z\to t-\sqrt{t^2-1},\alpha \to \sqrt{t^2-1}+t,f\to \frac{n+2}{2},d\to \frac{n+2}{2}$ and simplify.
\begin{multline}
\int_0^b \frac{x^{m-1}}{\left(1-2 t x+x^2\right)^{\frac{n+2}{2}}} \, dx\\
=\sum _{l=0}^{\infty } \frac{b^{l+m}
   \left(-t+\sqrt{-1+t^2}\right)^l \binom{-1-\frac{n}{2}}{l} \,
   _2F_1\left(l+m,1+\frac{n}{2};1+l+m;-\frac{b}{-t+\sqrt{-1+t^2}}\right)}{l+m}
\end{multline}
where $Re(m)>0$.
\end{example}
\section{Definite integrals with singularities}
An integral is considered singular if, at one or more points within the domain of integration, its integrand has an infinite value. Nevertheless, such integrals are said to exist if they converge. (They are considered to be nonexistent if they do not converge.) The Hilbert transform is the most widely used illustration of a singular integral. (Note that the logarithmic integral converges in the classical Riemann sense, therefore it is not singular.) Singular integrals are generally defined by taking the limit when the singularity vanishes and then removing the domain-space that contains the singularity.  [Wolfram,\href{https://mathworld.wolfram.com/SingularIntegral.html}{SingularIntegral}].
\begin{example}
Use equation (\ref{eq:10.4}) and set $a\to 2,b\to 1,m\to 1,z\to \frac{1}{4},c\to 4,\alpha \to \frac{1}{2},s\to 1,f\to 1,d\to 1$ and take the first partial derivative with respect to $k$ and set $k=0$ and simplify. Singularity at $x=1/2$.
\begin{multline}
\int_0^1 \frac{x \log \left(\log \left(\frac{1}{2 x}\right)\right)}{(2+x) \left(4+x^4\right)} \, dx
=\sum _{j=0}^{\infty } \sum _{l=0}^{\infty } \frac{(-1)^{j+l} 2^{-5-6 j-2 l} E_1(-((2+4
   j+l) \log (2)))}{2+4 j+l}\\
+\sum _{l=0}^{\infty } (-1)^l 2^{-5-l} \Phi \left(-\frac{1}{4},1,\frac{2+l}{4}\right) (i \pi +\log (\log (2)))
\end{multline}
\end{example}
\begin{figure}[H]
\caption{Plot of $Re(\frac{x \log (-\log (2 x))}{(x+2) \left(x^4+4\right)})$}
\includegraphics[width=8cm]{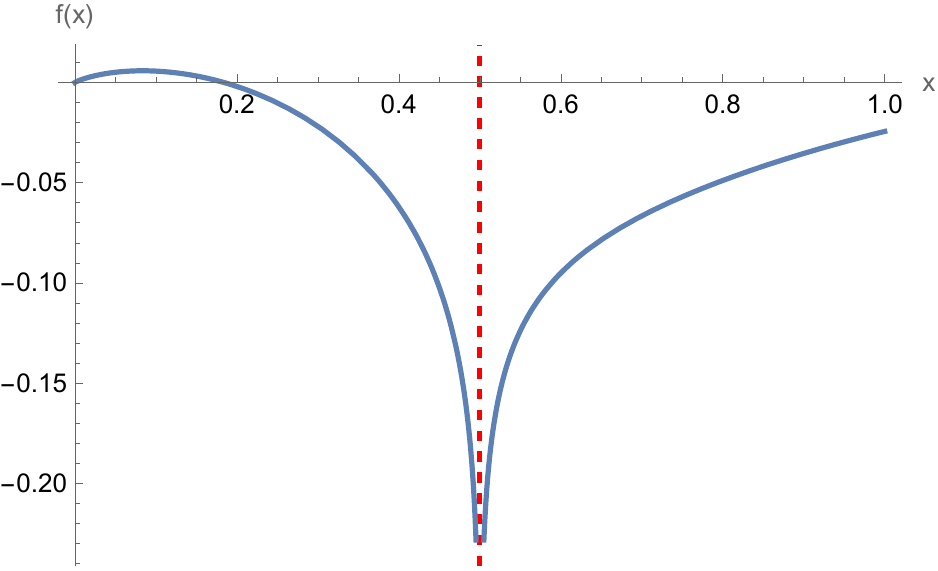}
\end{figure}
\begin{figure}[H]
\caption{Plot of $Im(\frac{x \log (-\log (2 x))}{(x+2) \left(x^4+4\right)})$}
\includegraphics[width=8cm]{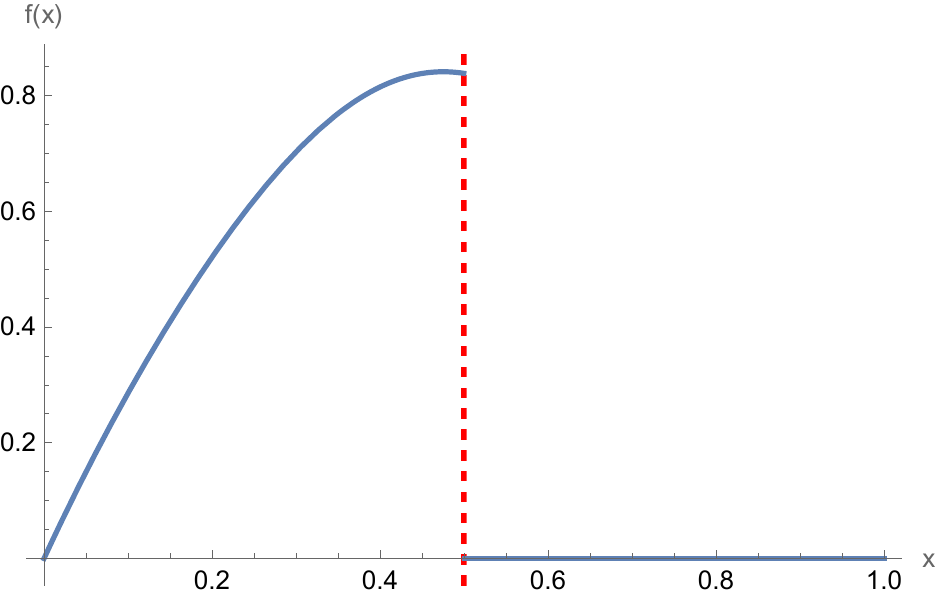}
\end{figure}
\begin{example}
Use equation (\ref{eq:10.4}) and set $a\to 2,b\to 1,m\to 1,z\to \frac{1}{4},c\to 4,\alpha \to \frac{1}{2},s\to 2,f\to \frac{1}{2},d\to \frac{1}{2}$ and take the first partial derivative with respect to $k$ and set $k=0$ and simplify. Singularity at $x=1/2$.
\begin{multline}
\int_0^1 \frac{x \log (\log (\frac{1}{2x}))}{\sqrt{2+x^2} \sqrt{4+x^4}} \, dx\\
=\sum _{j=0}^{\infty } \sum
   _{l=0}^{\infty } \frac{2^{\frac{1}{2} (-3) (3+4 j+2 l)} \binom{-\frac{1}{2}}{j} \binom{-\frac{1}{2}}{l}
   \left(E_1(-((1+2 j+l) \log (4)))+4^{1+2 j+l} (i \pi +\log (\log (2)))\right)}{1+2 j+l}
\end{multline}
\end{example}
\begin{example}
Use equation (\ref{eq:10.4}) and set $a\to \frac{3}{2},b\to 1,m\to 1,z\to -\frac{1}{5},c\to 4,\alpha \to -\frac{1}{2},s\to 2,f\to 3,d\to-\frac{1}{2}$ and take the first partial derivative with respect to $k$ and set $k=0$ and simplify. Singularity at $x=2/3$.
\begin{multline}
\int_0^1 \frac{x \sqrt{1-\frac{x^4}{5}} }{\left(-2+x^2\right)^3}\log \left(\log \left(\frac{2}{3
   x}\right)\right) \, dx\\
=-\sum _{j=0}^{\infty } \sum _{l=0}^{\infty } \frac{(-1)^{j+2 l}
   2^{-5-l} 5^{-j} 9^{-1-2 j-l} (1+l) (2+l) \binom{\frac{1}{2}}{j} }{1+2 j+l}\\
\left(4^{1+2 j+l} E_1\left(-2 (1+2 j+l) \log
   \left(\frac{3}{2}\right)\right)+9^{1+2 j+l} \left(i \pi +\log \left(\frac{1}{2} \log
   \left(\frac{9}{4}\right)\right)\right)\right)
\end{multline}
\end{example}
\begin{figure}[H]
\caption{Plot of $Re(\frac{x \sqrt{1-\frac{x^4}{5}} \log \left(\log \left(\frac{2}{3 x}\right)\right)}{\left(x^2-2\right)^3})$}
\includegraphics[width=8cm]{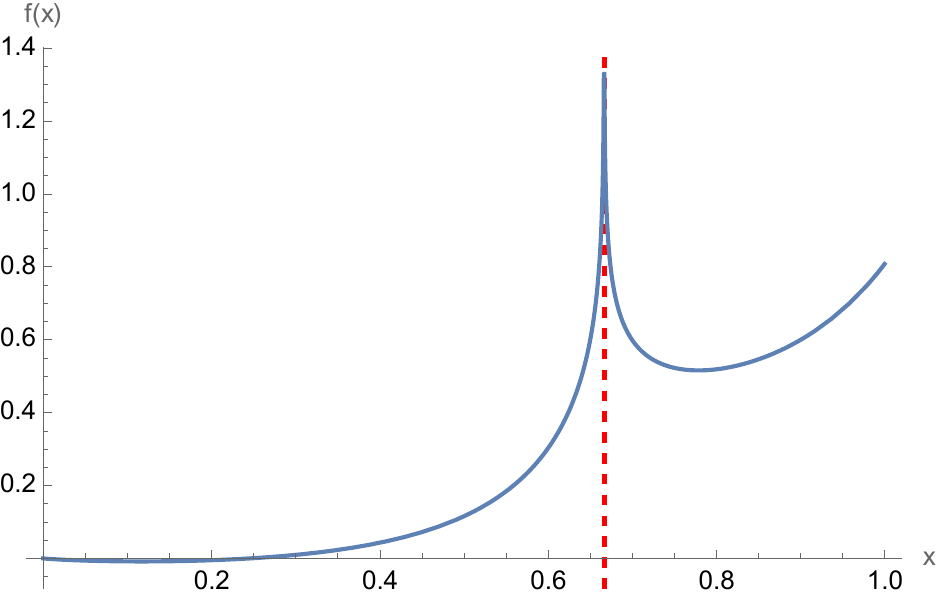}
\end{figure}
\begin{figure}[H]
\caption{Plot of $Im(\frac{x \sqrt{1-\frac{x^4}{5}} \log \left(\log \left(\frac{2}{3 x}\right)\right)}{\left(x^2-2\right)^3})$}
\includegraphics[width=8cm]{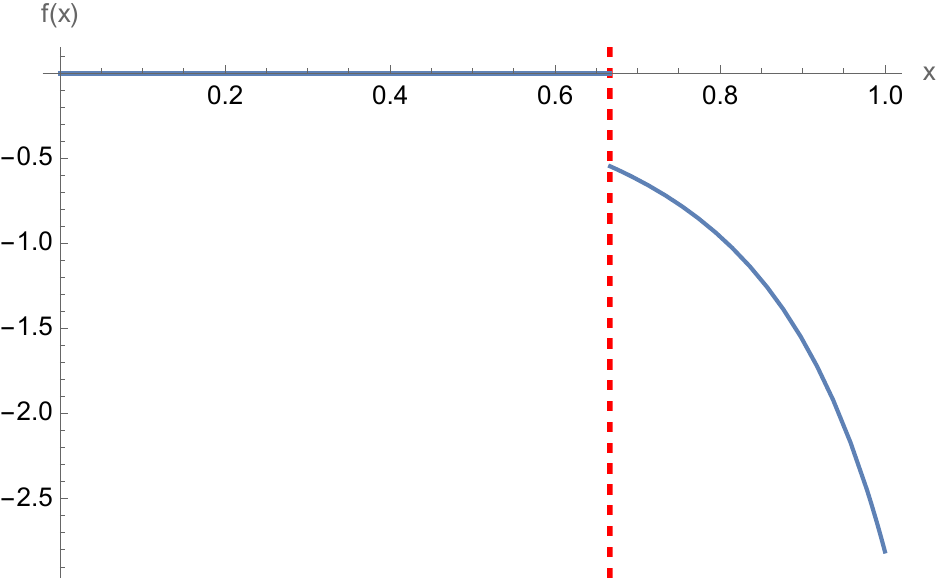}
\end{figure}
\begin{example}
Use equation (\ref{eq:10.4}) and set $a\to 3,b\to 1,m\to 1,z\to -\frac{1}{5},c\to 4,\alpha \to -\frac{1}{2},s\to 2,f\to 2,d\to -\frac{1}{2}$ and take the first partial derivative with respect to $k$ and set $k=0$ and simplify. Singularity at $x=1/3$.
\begin{multline}
\int_0^1 \frac{x \sqrt{1-\frac{x^4}{5}} \log \left(\log \left(\frac{1}{3
   x}\right)\right)}{\left(-2+x^2\right)^2} \, dx
=\sum _{j=0}^{\infty } \sum _{l=0}^{\infty } \frac{(-1)^{j+2 l}
   2^{-3-l} 5^{-j} 9^{-1-2 j-l} (1+l) \binom{\frac{1}{2}}{j} }{1+2 j+l}\\
\left(E_1(-2 (1+2 j+l) \log (3))+9^{1+2 j+l} (i \pi
   +\log (\log (3)))\right)
\end{multline}
\end{example}
\begin{figure}[H]
\caption{Plot of $Re(\frac{x \sqrt{1-\frac{x^4}{5}} \log (\log (3 x))}{\left(x^2-2\right)^2})$}
\includegraphics[width=8cm]{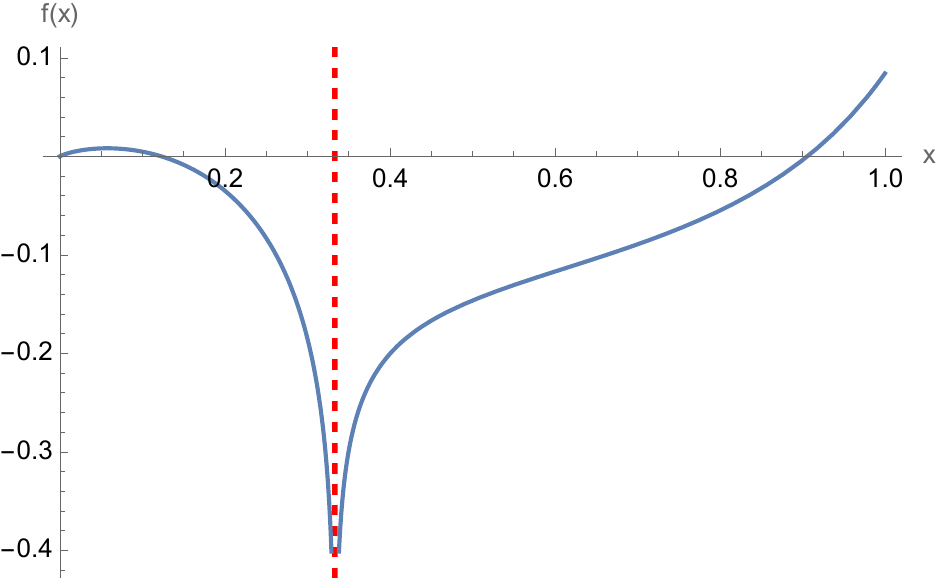}
\end{figure}
\begin{figure}[H]
\caption{Plot of $Im(\frac{x \sqrt{1-\frac{x^4}{5}} \log (\log (3 x))}{\left(x^2-2\right)^2})$}
\includegraphics[width=8cm]{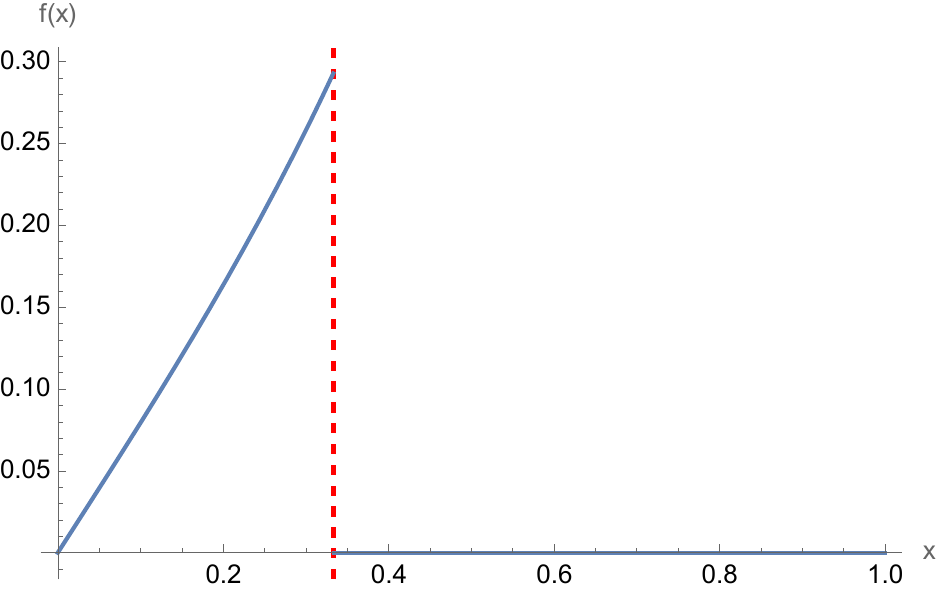}
\end{figure}
\begin{example}
Use equation (\ref{eq:10.4}) and set $a\to 3,b\to 1,m\to 1,z\to \frac{1}{7},c\to 7,\alpha \to \frac{1}{3},s\to 2,f\to \frac{3}{2},d\to-\frac{5}{2}$ and take the first partial derivative with respect to $k$ and set $k=0$ and simplify. Singularity at $x=1/3$.
\begin{multline}
\int_0^1 \frac{x \left(1+\frac{x^7}{7}\right)^{5/2} \log \left(\log \left(\frac{1}{3
   x}\right)\right)}{\left(1+\frac{x^2}{3}\right)^{3/2}} \, dx\\
=\sum _{j=0}^{\infty } \sum _{l=0}^{\infty }
   \frac{3^{-2-7 j-3 l} 7^{-j} \binom{-\frac{3}{2}}{l} \binom{\frac{5}{2}}{j} \left(E_1(-((2+7 j+2 l) \log
   (3)))+3^{2+7 j+2 l} (i \pi +\log (\log (3)))\right)}{2+7 j+2 l}
\end{multline}
\end{example}
\begin{figure}[H]
\caption{Plot of $Re(\frac{x \left(\frac{x^7}{7}+1\right)^{5/2} \log \left(\log \left(\frac{1}{3
   x}\right)\right)}{\left(\frac{x^2}{3}+1\right)^{3/2}})$}
\includegraphics[width=8cm]{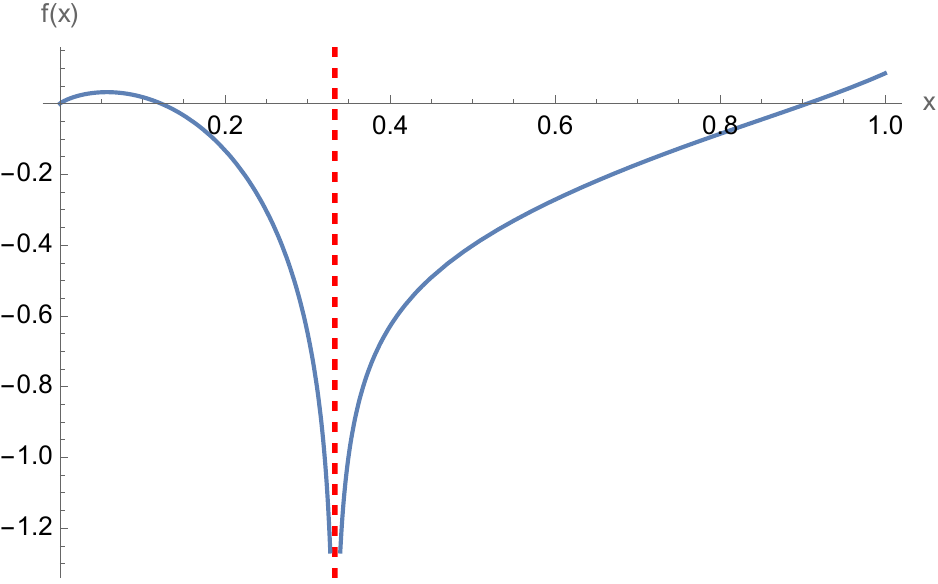}
\end{figure}
\begin{figure}[H]
\caption{Plot of $Im(\frac{x \left(\frac{x^7}{7}+1\right)^{5/2} \log \left(\log \left(\frac{1}{3
   x}\right)\right)}{\left(\frac{x^2}{3}+1\right)^{3/2}})$}
\includegraphics[width=8cm]{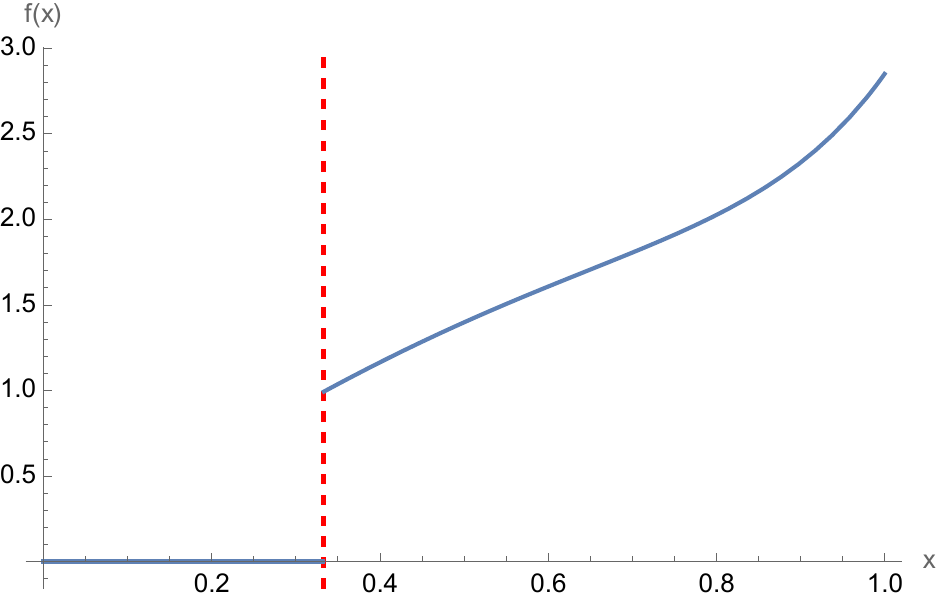}
\end{figure}
\begin{example}
Use equation (\ref{eq:10.4}) and set $a\to 4,b\to 1,d\to -\frac{1}{2},f\to \frac{1}{2},c\to 1,z\to \frac{1}{2},s\to 1,\alpha \to -\frac{1}{3},m\to0$ and take the first partial derivative with respect to $k$ and set $k=0$ and simplify. Singularity at $x=1/4$.
\begin{multline}
\int_0^1 \frac{\sqrt{2+x} \log \left(\log \left(\frac{1}{4 x}\right)\right)}{\sqrt{2-\frac{2 x}{3}}} \,
   dx\\
=\sum _{j=0}^{\infty } \sum _{l=0}^{\infty } \frac{\left(-\frac{1}{3}\right)^l 2^{-2-3 j-2 l}
   \binom{-\frac{1}{2}}{l} \binom{\frac{1}{2}}{j} \left(E_1(-((1+j+l) \log (4)))+4^{1+j+l} (i \pi +\log (\log
   (4)))\right)}{1+j+l}
\end{multline}
\end{example}
\begin{figure}[H]
\caption{Plot of $Re(\frac{\sqrt{x+2} \log \left(\log \left(\frac{1}{4 x}\right)\right)}{\sqrt{2-\frac{2 x}{3}}})$}
\includegraphics[width=8cm]{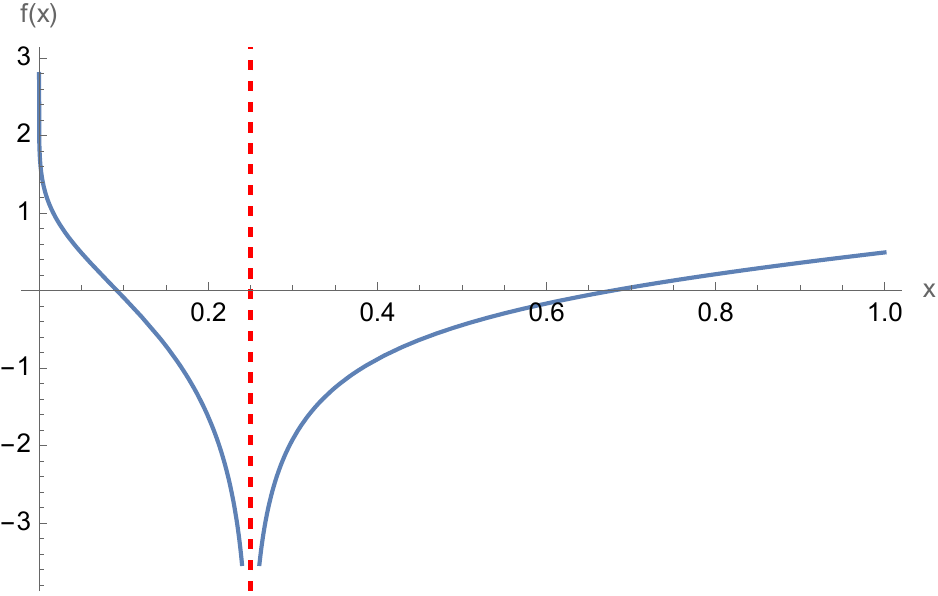}
\end{figure}
\begin{figure}[H]
\caption{Plot of $Im(\frac{\sqrt{x+2} \log \left(\log \left(\frac{1}{4 x}\right)\right)}{\sqrt{2-\frac{2 x}{3}}})$}
\includegraphics[width=8cm]{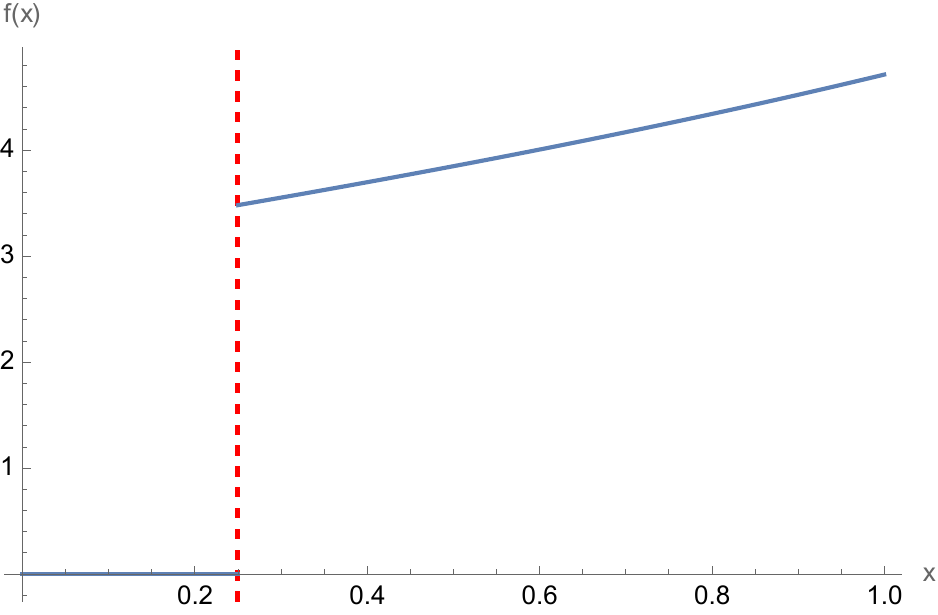}
\end{figure}
\subsection{Extended Sylvester integral}
In 1860, JJ. Sylvester \cite{sylvester} investigated a few definite integrals involving the product of the logarithmic and rational functions. These forms represent the logarithm function of an elliptic function form over the interval [0,1]. Alternate less complicated forms were previously studied by Euler in his Institutiones calculi integralis. In this section we derive an alternate closed form solution for one of Slyvester's integrals and extend the logarithmic function form and produced some plots with singularities. 
\begin{example}
 Alternate form for a Sylvester definite integral on pp. 298 in \cite{sylvester}. Use equation (\ref{eq:10.4}) and set $k\to 1,a\to 1,b\to 1,d\to \frac{1}{2},f\to 1,c\to 2,s\to 2,z\to -1,m\to 0,\alpha \to -c^2$ and simplify.
 \begin{multline}
\int_0^1 \frac{\log (x)}{\sqrt{1-x^2} \left(1-c^2 x^2\right)} \, dx=\frac{1}{2} \sqrt{\pi } \sum
   _{l=0}^{\infty } \frac{c^{2 l} \Gamma \left(\frac{3}{2}+l\right) \left(\psi ^{(0)}\left(\frac{1}{2}+l\right)-\psi
   ^{(0)}(1+l)\right)}{(1+2 l) \Gamma (1+l)}
\end{multline}
where $|Re(c)|<1$.
\end{example}
\begin{example}
A generalized logarithm definite integral. Use equation (\ref{eq:10.4}) and set $b\to 1,d\to \frac{1}{2},f\to 1,c\to 2,s\to 2,z\to -1,\alpha \to -c^2$ and simplify.
\begin{multline}
\int_0^1 \frac{x^m (-\log (a x))^k}{\sqrt{1-x^2} \left(1-c^2 x^2\right)} \, dx\\
=\sum _{j=0}^{\infty } \sum
   _{l=0}^{\infty } (-1)^j a^{-1-2 j-2 l-m} \left(-c^2\right)^l (1+2 j+2 l+m)^{-1-k} (-1)^l\\
 \binom{-\frac{1}{2}}{j}
   \Gamma (1+k,-((1+2 j+2 l+m) \log (a)))
\end{multline}
where $Re(k)>1, |Im(k)|>2\pi,  Re(a)\geq\pi, |Re(c)|<1$.
\end{example}
\begin{example}
Nested logarithm form. Use equation (\ref{eq:10.4}) and set $b\to 1,d\to \frac{1}{2},f\to 1,c\to 2,s\to 2,z\to -z^2,\alpha \to -c^2$ next take the first partial derivative with respect to $k$ and set $k=0$ and simplify.
\begin{multline}
\int_0^1 \frac{x^m \log \left(\log \left(\frac{1}{a x}\right)\right)}{\left(1-c^2 x^2\right) \sqrt{1-x^2 z^2}}
   \, dx\\
=\sum _{j=0}^{\infty } \sum _{l=0}^{\infty } \frac{(-1)^l a^{-1-2 j-2 l-m} \left(-c^2\right)^l
   \left(-z^2\right)^j \binom{-\frac{1}{2}}{j} }{1+2 j+2 l+m}\\
\left(E_1(-((1+2 j+2 l+m) \log (a)))+a^{1+2 j+2 l+m} \log (-\log
   (a))\right)
\end{multline}
where $Re(a)>\pi, |Re(c)|<1$.
\end{example}
\begin{figure}[H]
\caption{Plot of $Re\left(\frac{\sqrt{x} \log \left(\log \left(\frac{1}{2 x}\right)\right)}{\sqrt{1-\frac{x^2}{2}} \left(1-c^2 x^2\right)}\right)$}
\includegraphics[width=8cm]{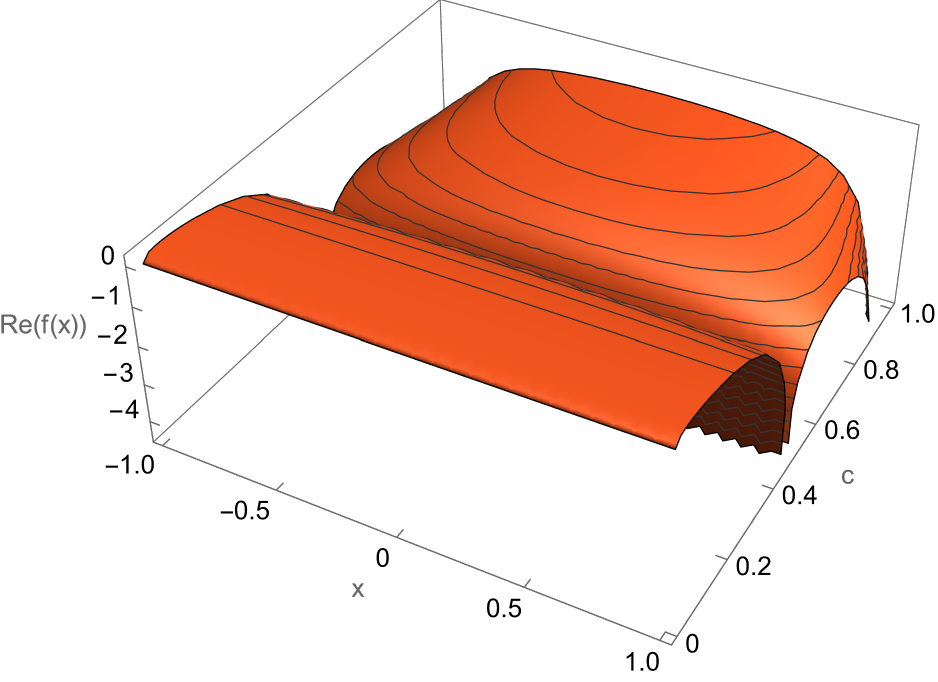}
\end{figure}
\begin{figure}[H]
\caption{Plot of $Im\left(\frac{\sqrt{x} \log \left(\log \left(\frac{1}{2 x}\right)\right)}{\sqrt{1-\frac{x^2}{2}} \left(1-c^2 x^2\right)}\right)$}
\includegraphics[width=8cm]{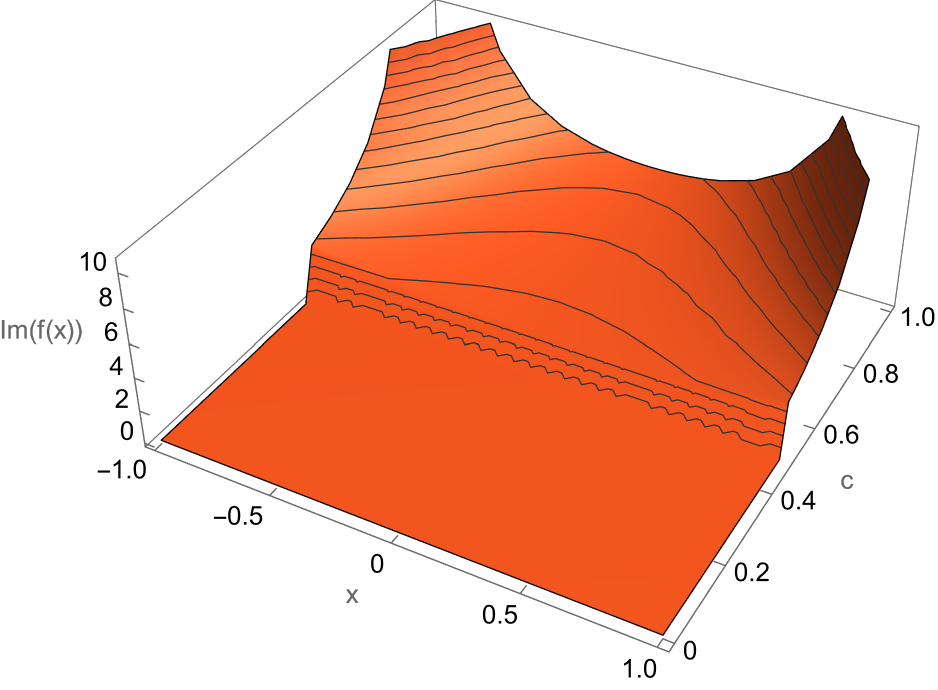}
\end{figure}
\begin{example}
Singularity at $x=1/2$. Use equation (\ref{eq:10.4}) and set $c\to 3,s\to 3,d\to \frac{1}{3},f\to -\frac{1}{3},m\to 0,a\to 2,b\to 1$ next take the first partial derivative with respect to $k$ and set $k=0$ and simplify.
\begin{multline}
\int_0^1 \frac{\sqrt[3]{1+x^3 \alpha } \log \left(\log \left(\frac{1}{2 x}\right)\right)}{\sqrt[3]{1+x^3 z}}
   \, dx\\
=\sum _{j=0}^{\infty } \sum _{l=0}^{\infty } \frac{2^{-1-3 j-3 l} z^j \alpha ^l \binom{-\frac{1}{3}}{j}
   \binom{\frac{1}{3}}{l} \left(E_1(-((1+3 j+3 l) \log (2)))+2^{1+3 j+3 l} (i \pi +\log (\log (2)))\right)}{1+3 j+3
   l}
\end{multline}
where $|Re(z)|<1,Re(\alpha)|<1$.
\end{example}
\begin{figure}[H]
\caption{Plot of $Re\left(\frac{\sqrt{1-\frac{x^3}{2}} \log \left(\log \left(\frac{1}{2 x}\right)\right)}{\sqrt{\frac{x^3}{2}+1}}\right)$}
\includegraphics[width=8cm]{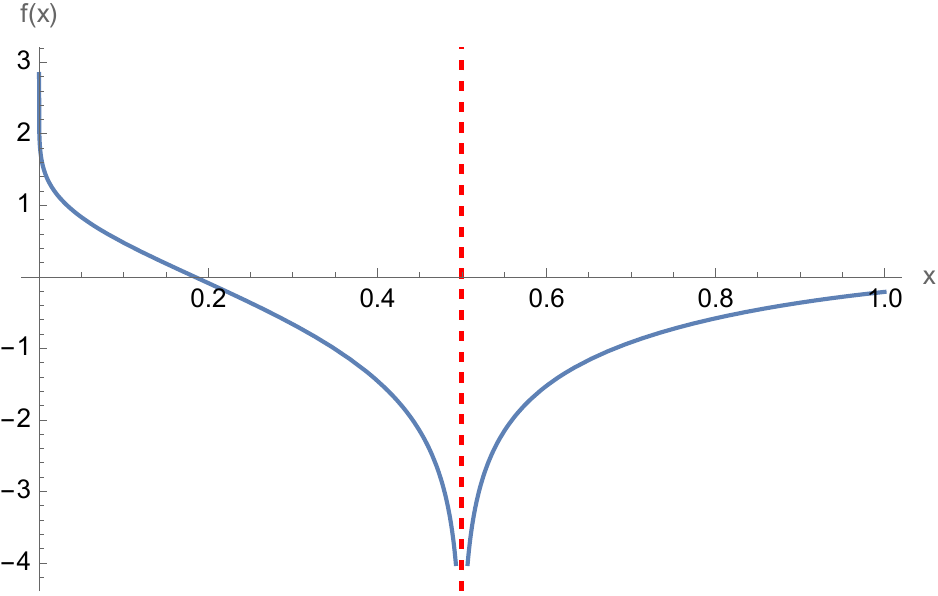}
\end{figure}
\begin{figure}[H]
\caption{Plot of $Im\left(\frac{\sqrt{1-\frac{x^3}{2}} \log \left(\log \left(\frac{1}{2 x}\right)\right)}{\sqrt{\frac{x^3}{2}+1}}\right)$}
\includegraphics[width=8cm]{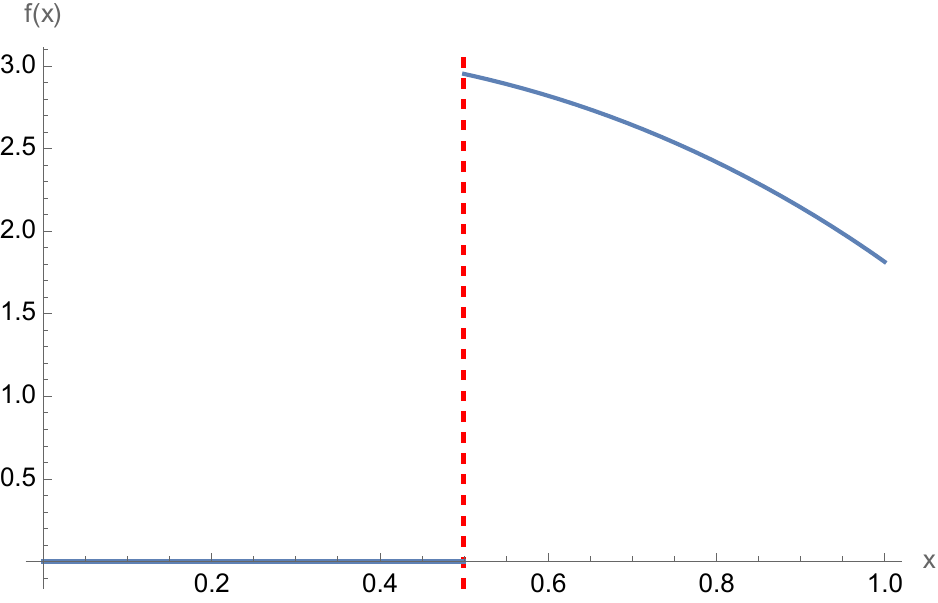}
\end{figure}
\begin{example}
Use equation (\ref{eq:10.4}) and set $d\to \frac{1}{2},f\to \frac{1}{2},c\to 2,s\to 2,b\to 1,a\to 2,z\to \frac{1}{z^2},\alpha \to \frac{1}{\alpha^2},k\to 0,a\to 1$ and simplify.
\begin{equation}\label{eq:11.11}
\int_0^1 \frac{x^{-1+s}}{\sqrt{a^2+x^2} \sqrt{b^2+x^2}} \, dx=\sum _{l=0}^{\infty } \frac{b^{-1-2 l}
   \binom{-\frac{1}{2}}{l} \, _2F_1\left(\frac{1}{2},l+\frac{s}{2};1+l+\frac{s}{2};-\frac{1}{a^2}\right)}{a (2
   l+s)}
\end{equation}
where $0< Re(s)<1,Re(b)>1$.
\end{example}
\begin{example}
The elliptic integral of the first kind $F(z|m)$ see [Wolfram, \href{https://mathworld.wolfram.com/EllipticIntegraloftheFirstKind.html}{1}]. Use equation (\ref{eq:11.11}) and set $m=1$ and simplify.
\begin{equation}
\sum _{l=0}^{\infty } \frac{\binom{-\frac{1}{2}}{l} \,
   _2F_1\left(\frac{1}{2},\frac{1}{2}+l;\frac{3}{2}+l;-\frac{1}{a^2}\right)}{(2 l+1) b^{2 l}}=-i a F\left(i
   \text{csch}^{-1}(a)|\frac{a^2}{b^2}\right)
\end{equation}
where $a,b\in\mathbb{C}$.
\end{example}
\begin{example}
 Extended Jellett integral exhibited by a geometrical construction, see pp.402 example 18 in \cite{williamson}. Use equation (\ref{eq:10.4}) and set $k\to 0,a\to 1,d\to \frac{1}{2},z\to \frac{1}{z},c\to m,m\to 0,f\to 0,\alpha \to 1,s\to 0$ and form a second equation by replacing $b\to a$ and take their difference and simplify.
\begin{equation}
\int_a^b \frac{1}{\sqrt{x^m+z}} \, dx=\frac{-a \,
   _2F_1\left(\frac{1}{2},\frac{1}{m};1+\frac{1}{m};-\frac{a^m}{z}\right)+b \,
   _2F_1\left(\frac{1}{2},\frac{1}{m};1+\frac{1}{m};-\frac{b^m}{z}\right)}{\sqrt{z}}
\end{equation}
where $Re(m)>0$.
\end{example}
\begin{figure}[H]
\caption{Plot of $Re\left(\frac{1}{\sqrt{x^m-1}}\right)$}
\includegraphics[width=8cm]{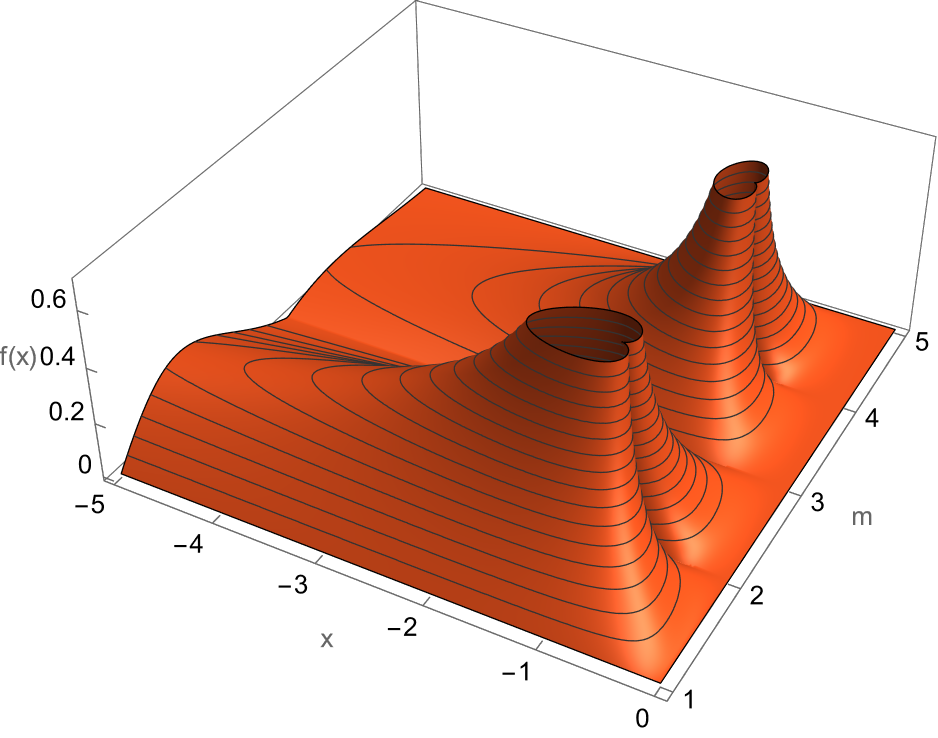}
\end{figure}
\begin{figure}[H]
\caption{Plot of $Im\left(\frac{1}{\sqrt{x^m-1}}\right)$}
\includegraphics[width=8cm]{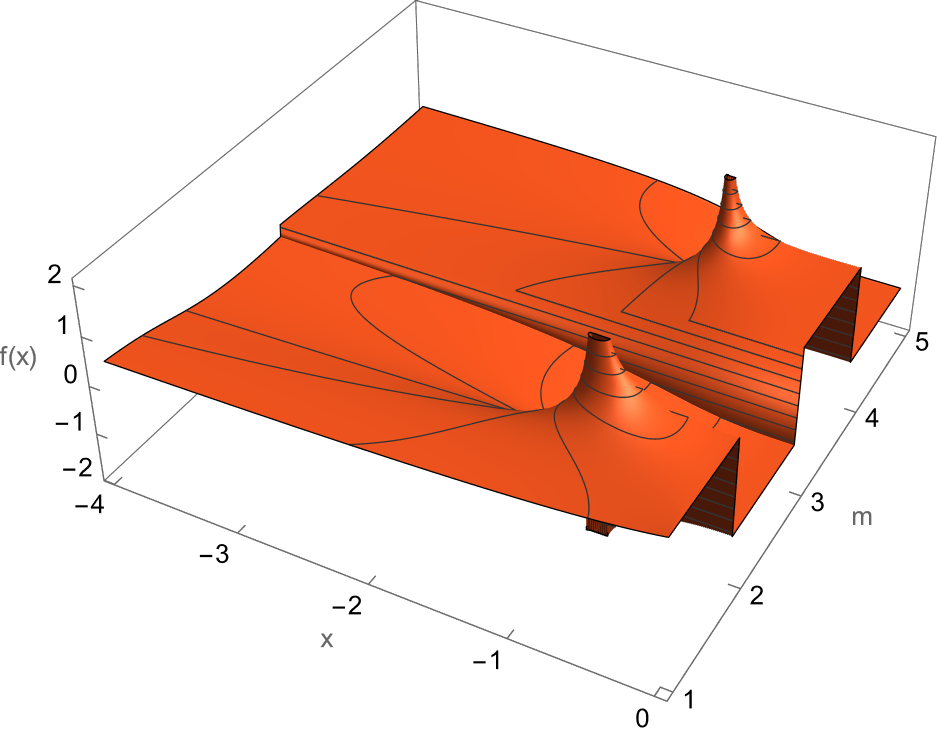}
\end{figure}
\begin{example}
 Generalized form of equation (2.6.8.1) in \cite{prud1}. Use equation (\ref{eq:10.4}) and set $k\to 1,a\to 1,f\to \frac{1}{2},d\to \frac{1}{2},c\to 2,s\to 2,m\to 0$ and form a second equation by replacing $b\to a$ and take their difference and simplify.
 \begin{multline}\label{eq:11.14}
\int_a^b \frac{\log (x)}{\sqrt{1+x^2 z} \sqrt{1+x^2 \alpha }} \, dx\\
=\sum _{j=0}^{\infty } \sum _{l=0}^{\infty
   } \frac{z^j \alpha ^l \binom{-\frac{1}{2}}{j} \binom{-\frac{1}{2}}{l} (\Gamma (2,-((1+2 j+2 l) \log (a)))-\Gamma
   (2,-((1+2 j+2 l) \log (b))))}{(1+2 j+2 l)^2}
\end{multline}
where $Re(z)>0,Re(\alpha)>0$.
\end{example}
\begin{example}
Use equation (\ref{eq:11.14}) and rewrite in the form as the left-hand side as equation (2.6.8.1) in \cite{prud1}. 
\begin{multline}
\int_a^b \frac{\log (x)}{\sqrt{x^2-z^2} \sqrt{\alpha ^2-x^2}} \, dx\\
=-\sum _{j=0}^{\infty } \sum _{l=0}^{\infty } \frac{i (-1)^{-j+l} z^{-1-2 j} \alpha ^{-1-2 l}
   \binom{-\frac{1}{2}}{j} \binom{-\frac{1}{2}}{l} }{(1+2 j+2 l)^2}\\
(\Gamma (2,-((1+2 j+2 l) \log (a)))-\Gamma (2,-((1+2 j+2 l) \log (b))))
\end{multline}
where $Re(a)>Re(b), Re(\alpha)>Re(z)>3/2$.
\end{example}
\begin{example}
A functional equation in terms of the complete elliptic integral of the first kind [Wolfram,\href{https://mathworld.wolfram.com/CompleteEllipticIntegraloftheFirstKind.html}{1}], using equation (2.6.8.10) in \cite{prud1}. 
Use equation (\ref{eq:10.4}) and set $k\to 1,a\to 1,f\to \frac{1}{2},d\to \frac{1}{2},c\to 2,s\to 2,m\to 0,b\to 1$ and simplify.
\begin{multline}\label{eq:11.16}
\int_0^1 \frac{\log (x)}{\sqrt{1+x^2 y^2} \sqrt{1+x^2 z^2}} \, dx=-\sum _{l=0}^{\infty } \frac{y^{2 l} \binom{-\frac{1}{2}}{l} \,
   _3F_2\left(\frac{1}{2},\frac{1}{2}+l,\frac{1}{2}+l;\frac{3}{2}+l,\frac{3}{2}+l;-z^2\right)}{(1+2 l)^2}
\end{multline}
where $Re(y)<1$.
\begin{multline}\label{eq:11.17}
\int_0^1 \frac{\log (x)}{\sqrt{x^2+y^2} \sqrt{x^2+z^2}} \, dx=-\sum _{j=0}^{\infty } \frac{z^{-1-2 j} \binom{-\frac{1}{2}}{j} \,
   _3F_2\left(\frac{1}{2},\frac{1}{2}+j,\frac{1}{2}+j;\frac{3}{2}+j,\frac{3}{2}+j;-\frac{1}{y^2}\right)}{(1+2 j)^2 y}
\end{multline}
where $Re(y)<1$.
Next rewrite equation (2.6.8.10) in \cite{prud1} over the domain $x\in[0,1]$ and simplify to get;
\begin{multline}\label{eq:11.18}
\int_0^1 \left(-\frac{1}{\sqrt{x^2+y^2} \sqrt{x^2+z^2}}+\frac{1}{\sqrt{1+x^2 y^2} \sqrt{1+x^2 z^2}}\right) \log (x) \, dx=-\frac{ \log
   (y z)}{2 y}K\left(1-\frac{z^2}{y^2}\right)
\end{multline}
where $Re(y)<1,Re(z)<1$. Then we substitute equations (\ref{eq:11.16}) and (\ref{eq:11.17}) into (\ref{eq:11.18}) and simplify.
\begin{multline}\label{eq:11.19}
\sum _{l=0}^{\infty } \frac{y^{2 l} \binom{-\frac{1}{2}}{l} \, _3F_2\left(\frac{1}{2},\frac{1}{2}+l,\frac{1}{2}+l;\frac{3}{2}+l,\frac{3}{2}+l;-z^2\right)}{(1+2
   l)^2}\\
-\sum _{j=0}^{\infty } \frac{z^{-1-2 j} \binom{-\frac{1}{2}}{j} \,
   _3F_2\left(\frac{1}{2},\frac{1}{2}+j,\frac{1}{2}+j;\frac{3}{2}+j,\frac{3}{2}+j;-\frac{1}{y^2}\right)}{(1+2 j)^2 y}=\frac{K\left(1-\frac{z^2}{y^2}\right) \log (y z)}{2
   y}
\end{multline}
where $Re(y)<1$.
\end{example}
\begin{example}
Use equation (\ref{eq:11.19}) and set $z=i,y=1/\sqrt{2}$ and simplify.
\begin{multline}
-\sum _{j=0}^{\infty } \frac{i^{-1-2 j} \sqrt{2} \binom{-\frac{1}{2}}{j} \,
   _3F_2\left(\frac{1}{2},\frac{1}{2}+j,\frac{1}{2}+j;\frac{3}{2}+j,\frac{3}{2}+j;-2\right)}{(1+2 j)^2}\\
-\sum _{l=0}^{\infty } \frac{2^{-1-l} \sqrt{\pi }
   \binom{-\frac{1}{2}}{l} \Gamma \left(\frac{3}{2}+l\right) \left(\psi ^{(0)}\left(\frac{1}{2}+l\right)-\psi ^{(0)}(1+l)\right)}{(1+2 l) \Gamma (1+l)}\\
=\frac{K(3) \log
   \left(\frac{i}{\sqrt{2}}\right)}{\sqrt{2}}
\end{multline}
\end{example}
\begin{example}
Appell hypergeometric function of two variables [Wolfram,\href{https://mathworld.wolfram.com/AppellHypergeometricFunction.html}{1}], in terms of the hypergeometric function [Wolfram,\href{https://mathworld.wolfram.com/HypergeometricFunction.html}{8}]. Use equation (\ref{eq:10.4}) and set $k\to 0,a\to 1,m\to 0,d\to -m,f\to m+1,c\to 2,s\to 1$ and simplify.
\begin{multline}
\sum _{l=0}^{\infty } \frac{b^{1+l} \alpha ^l \binom{-1+m}{l} \,
   _2F_1\left(\frac{1+l}{2},m;\frac{3+l}{2};-s^2\right)}{1+l}\\
=\frac{1}{m \alpha }\left( -\left(\frac{b^2 \alpha ^2}{s^2+b^2 \alpha
   ^2}\right)^m F_1\left(m;m,m;1+m;\frac{i s}{i s-b \alpha },\frac{i s}{i s+b \alpha }\right)\right. \\ \left.
+\left(\frac{s^2+b^2
   \alpha ^2}{b^2 \alpha ^2 (1+b \alpha )}\right)^{-m} F_1\left(m;m,m;1+m;\frac{i s+i b s \alpha }{i s-b \alpha
   },\frac{i s+i b s \alpha }{i s+b \alpha }\right)\right)
\end{multline}
where $Re(z)>0,Re(\alpha)>0$.
\end{example}
\section{Special case of a generalized definite integral}
In this section we look at a special case of equation (\ref{eq:10.4}) expressed in terms of the Hurwitz-Lerch zeta function [Wolfram,\href{https://mathworld.wolfram.com/LerchTranscendent.html}{2}].
\begin{example}
Use equation (\ref{eq:10.4}) and set $a\to 1,b\to 1,f\to 1,d\to -d$ and simplify.
\begin{multline}\label{eq:12.1}
\int_0^1 \frac{x^m \left(1+x^c z\right)^d }{1+x^s \alpha }\log ^k\left(\frac{1}{x}\right) \,
   dx=\frac{\Gamma (1+k) }{s^{k+1}}\sum _{j=0}^{\infty } z^j \binom{d}{j} \Phi \left(-\alpha ,1+k,\frac{1+c
   j+m}{s}\right)
\end{multline}
where $Re(c)>0, Re(s)>0, |Re(m)|<1$.
\end{example}
\begin{example}
Use equation (\ref{eq:12.1}) and set $\alpha=1$ and simplify the right-hand side using equation [DLMF,\href{https://dlmf.nist.gov/25.14.E2}{25.14.2}].
\begin{multline}\label{eq:12.2}
\int_0^1 \frac{x^m \left(1+x^c z\right)^d \log ^k\left(\frac{1}{x}\right)}{1+x^s} \, dx=\frac{\Gamma (1+k)
   }{(2 s)^{k+1}}\\
\sum _{j=0}^{\infty } z^j \binom{d}{j} \left(\zeta \left(1+k,\frac{1+c j+m}{2 s}\right)-\zeta \left(1+k\frac{1+c j+m+s}{2 s}\right)\right)
\end{multline}
where $Re(c)>0, Re(s)>0, |Re(m)|<1$.
\end{example}
\begin{example}
Use equation (\ref{eq:12.2}) and simplify using [DLMF,\href{https://dlmf.nist.gov/25.12.E13}{25.12.13}]
\begin{multline}
\int_0^1 \frac{x^m \left(1+x^c z\right)^d \log ^k\left(\frac{1}{x}\right)}{1+x^s} \, dx\\
=\frac{\left(\frac{\pi
   }{s}\right)^{1+k} }{2 i^k \sin (k \pi )}\sum _{j=0}^{\infty } z^j \binom{d}{j} \left(\text{Li}_{-k}\left(-e^{-\frac{i (1+c j+m) \pi}{s}}\right)-\text{Li}_{-k}\left(e^{-\frac{i (1+c j+m) \pi }{s}}\right)\right. \\ \left.
+e^{i k \pi }
   \left(\text{Li}_{-k}\left(-e^{\frac{i (1+c j+m) \pi }{s}}\right)-\text{Li}_{-k}\left(e^{\frac{i (1+c j+m) \pi }{s}}\right)\right)\right)
\end{multline}
where $Re(c)>0, Re(s)>2, |Re(m)|<1$.
\end{example}
\begin{example}
Use equation (\ref{eq:12.2}) and take the first partial derivative with respect to $k$ and set $k=0$ by applying l'Hopital's rule.
 \begin{multline}
\int_0^1 \frac{x^m \left(1+x^c z\right)^d \log \left(\log \left(\frac{1}{x}\right)\right)}{1+x^s} \, dx\\
=\sum
   _{j=0}^{\infty } \frac{z^j \Gamma (1+d) }{2 s \Gamma (1+d-j) \Gamma (1+j)}\left((\gamma +\log (2 s)) \left(\psi ^{(0)}\left(\frac{1+c j+m}{2
   s}\right)\right.\right. \\ \left.\left.
-\psi ^{(0)}\left(\frac{1+c j+m+s}{2 s}\right)\right)-\gamma _1\left(\frac{1+c j+m}{2 s}\right)+\gamma
   _1\left(\frac{1+c j+m+s}{2 s}\right)\right)
\end{multline}
where $Re(c)>0, Re(s)>2, |Re(m)|<1$.
\end{example}
\begin{example}
Use equation (\ref{eq:12.2}) and form a second equation by replacing $m\to r$ and take their difference and take the first partial derivative with respect to $k$ and apply l'Hopital's rule as $k\to -1$ and simplify using [DLMF,\href{https://dlmf.nist.gov/25.11.E18}{25.11.18}] .
\begin{multline}
\int_0^1 \frac{\left(-x^m+x^r\right) \left(1+x^c z\right)^d \log \left(\log
   \left(\frac{1}{x}\right)\right)}{\left(1+x^s\right) \log \left(\frac{1}{x}\right)} \, dx\\
=\sum _{j=0}^{\infty }
   \frac{z^j \Gamma (1+d) }{2 \Gamma (1+d-j) \Gamma
   (1+j)}\left(2 (\gamma +\log (2)+\log (s)) \left(\log \left(-2+\frac{1+c j+m}{s}\right)\right.\right. \\ \left.\left.
-\log \left(-1+\frac{1+c j+m}{s}\right)-\log \left(-2+\frac{1+c j+r}{s}\right)+\log \left(-1+\frac{1+c
   j+r}{s}\right)\right.\right. \\ \left.\left.
+\text{log$\Gamma $}\left(-1+\frac{1+c j+m}{2 s}\right)-\text{log$\Gamma $}\left(-1+\frac{1+cj+r}{2 s}\right)-\text{log$\Gamma $}\left(\frac{1+c j+m-s}{2 s}\right)\right.\right. \\ \left.\left.
+\text{log$\Gamma $}\left(\frac{1+c j+r-s}{2 s}\right)\right)-\zeta''\left(0,\frac{1+c j+m}{2 s}\right)\right. \\ \left.
+\zeta''\left(0,\frac{1+c j+r}{2s}\right)+\zeta''\left(0,\frac{1+c j+m+s}{2 s}\right)\right. \\ \left.
-\zeta''\left(0,\frac{1+c j+r+s}{2 s}\right)\right)
\end{multline}
where $Re(c)>0, Re(s)>2, |Re(m)|<1,|Re(r)|<1$.
\end{example}
\begin{theorem}
For all $|Re(m)|<1$ then,
\begin{multline}\label{eq:12.6}
\int_0^b x^m \left(1+x^c z\right)^{-d} \left(1+x^s \alpha \right)^{-f} \left(1+x^{\gamma } \beta \right)^{-g}
   \log ^k(a x) \, dx\\
=\sum _{j=0}^{\infty } \sum _{l=0}^{\infty } \sum _{h=0}^{\infty } b^{1+c j+m+l s+h \gamma } (a
   b)^{-1-c j-m-l s-h \gamma } z^j \alpha ^l \beta ^h (-1-c j-m-l s-h \gamma )^{-1-k}\\
 \binom{-d}{j} \binom{-f}{l}
 \binom{-g}{h} \Gamma (1+k,-((1+c j+m+l s+h \gamma ) \log (a b)))
\end{multline}
\end{theorem}
\begin{proof}
Apply the method in (\ref{eq:10.4}) to include another binomial coefficient and rational function.
\end{proof}
\begin{example}
Extended Byrd-Elliptic integral form. (800.07) in \cite{byrd}. Use equation (\ref{eq:12.6}) and set $k\to 0,a\to 1,d\to \frac{1}{2},z\to -1,c\to 2,g\to \frac{1}{2},\gamma \to 2,\beta \to -k^2,f\to 1,s\to
   2$ then take the indefinite integral with respect to $\alpha$ and replace $\alpha\to -\alpha^2,m=2$ and simplify. Derivation of equations (4.295.31-33) in \cite{grad} are special cases of equation (\ref{eq:12.7}).
   \begin{multline}\label{eq:12.7}
\int_0^b \frac{\log \left(1-x^2 \alpha ^2\right)}{\sqrt{1-x^2} \sqrt{1-k^2 x^2}} \, dx\\
=-\sum _{l=0}^{\infty }
   \sum _{h=0}^{\infty } \frac{(-1)^{h+2 l} b^{3+2 h+2 l} k^{2 h} \alpha ^{2+2 l} \binom{-\frac{1}{2}}{h} \,
   _2F_1\left(\frac{1}{2},\frac{3}{2}+h+l;\frac{5}{2}+h+l;b^2\right)}{(1+l) (3+2 h+2 l)}\\
=-\sum _{j=0}^{\infty } \sum
   _{h=0}^{\infty } \frac{(-1)^{h+j} b^{1+2 h+2 j} k^{2 h} \pi  \left(-\frac{4 b^2 \alpha ^2 \,
   _2F_1\left(1,\frac{3}{2}+h+j;\frac{5}{2}+h+j;b^2 \alpha ^2\right)}{(1+2 h+2 j) (3+2 h+2 j)}-\frac{2 \log
   \left(1-b^2 \alpha ^2\right)}{1+2 h+2 j}\right)}{2 \Gamma \left(\frac{1}{2}-h\right) \Gamma (1+h) \Gamma
   \left(\frac{1}{2}-j\right) \Gamma (1+j)}
\end{multline}
where $0< \alpha^2 < k^2$.
\end{example}
\begin{example}
Generalized form of equation (4.295.34) in \cite{grad}. Use equation (\ref{eq:12.6}) and set $k\to 0,a\to 1,d\to \frac{1}{2},z\to -1,c\to 2,g\to -\frac{1}{2},\gamma \to 2,\beta \to -k^2,f\to 1,s\to
   2$ then take the indefinite integral with respect to $\alpha$ and replace $\alpha\to -\alpha^2,m=2$ and simplify.
   \begin{multline}
\int_0^b \frac{\sqrt{1-k^2 x^2} \log \left(1-x^2 \alpha ^2\right)}{\sqrt{1-x^2}} \, dx\\
=-\sum _{l=0}^{\infty }
   \sum _{h=0}^{\infty } \frac{(-1)^l b^{3+2 h+2 l} \left(-k^2\right)^h \alpha ^2 \left(-\alpha ^2\right)^l
   \binom{\frac{1}{2}}{h} \, _2F_1\left(\frac{1}{2},\frac{3}{2}+h+l;\frac{5}{2}+h+l;b^2\right)}{(1+l) (3+2 h+2
   l)}
\end{multline}
where $0< \alpha^2 < k^2$.
\end{example}
\begin{example}
Generalized form of equation (4.295.35)  in \cite{grad}. Use equation (\ref{eq:12.6}) and set $k\to 0,a\to 1,d\to -\frac{1}{2},c\to 2,g\to \frac{1}{2},\gamma \to 2,\beta \to -k^2,f\to 1,s\to 2,b\to
   1$ then take the indefinite integral with respect to $\alpha$ and replace $\alpha\to -\alpha^2,m=2,z=-1$ and simplify.
   \begin{multline}
\int_0^1 \frac{\sqrt{1-x^2} \log \left(1-x^2 \alpha ^2\right)}{\sqrt{1-k^2 x^2}} \, dx=-\frac{\alpha ^2 \sqrt{\pi }}{2}
   \sum _{l=0}^{\infty } \frac{\alpha ^{2 l} \Gamma
   \left(\frac{1}{2} (5+2 l)\right) \, _2F_1\left(\frac{1}{2},\frac{3}{2}+l;3+l;k^2\right)}{(1+l) (3+2 l) \Gamma
   (3+l)}
\end{multline}
where $0< \alpha^2 < k^2$.
\end{example}
\begin{example}
Generalized form of equation (4.297.9) in \cite{grad}. Use equation (\ref{eq:12.6}) and set $k\to 0,a\to 1,d\to \frac{1}{2},c\to 2,f\to \frac{1}{2},s\to 2,g\to 1,\gamma \to 1$ then take the indefinite integral with respect to $\beta$ and replace $m=1$ and form a second equation by replacing $\beta\to -\beta$ and take their difference and simplify.
\begin{multline}\label{eq:12.10}
\int_0^b \frac{\tanh ^{-1}(x \beta )}{\sqrt{1+x^2 z} \sqrt{1+x^2 \alpha }} \, dx=\sum _{l=0}^{\infty }
   \frac{b^{1+2 l} \sqrt{\pi } \alpha ^l \tanh ^{-1}(b \beta ) \,
   _2F_1\left(\frac{1}{2},\frac{1}{2}+l;\frac{3}{2}+l;-b^2 z\right)}{(1+2 l) \Gamma \left(\frac{1}{2} (1-2 l)\right)
   \Gamma (1+l)}\\
-\sum _{j=0}^{\infty } \sum _{l=0}^{\infty } \frac{b^{2+2 j+2 l} \pi  z^j \alpha ^l \beta  }{4 (1+j+l) (1+2 j+2 l) \Gamma
   \left(\frac{1}{2}-j\right) \Gamma (1+j) \Gamma \left(\frac{1}{2}-l\right) \Gamma (1+l)}\\ \times
(\,
   _2F_1(1,2 (1+j+l);3+2 j+2 l;-b \beta )+\, _2F_1(1,2 (1+j+l);3+2 j+2 l;b \beta ))
\end{multline}
where $|Re(k)|<1,|Re(\alpha)|<1$.
\end{example}
\begin{example}
Generalized form of equation (4.297.9) in \cite{grad}. Use the method in equation (\ref{eq:12.10}) and simplify the series on the right-hand side starting with $j\in[0,\infty)$.
\begin{multline}
\int_0^b \frac{\tanh ^{-1}(x \beta )}{\sqrt{1+x^2 z} \sqrt{1+x^2 \alpha }} \, dx\\
=\sum _{h=0}^{\infty } \sum
   _{l=0}^{\infty } \frac{(-1)^h b^{2+h+2 l} \alpha ^l \beta  \left((-\beta )^h+\beta ^h\right)
   \binom{-\frac{1}{2}}{l} \, _2F_1\left(\frac{1}{2},1+\frac{h}{2}+l;2+\frac{h}{2}+l;-b^2 z\right)}{2 (1+h) (2+h+2
   l)}
\end{multline}
where $|Re(k)|<1,|Re(\alpha)|<1$.
\end{example}
\begin{example}
Generalized form of equation (4.297.9) in \cite{grad}. Use equation (\ref{eq:12.6}) and set $k\to 0,a\to 1,d\to \frac{1}{2},c\to 2,f\to \frac{1}{2},s\to 2,g\to 1,\gamma \to 1$ and form a second equation by replacing $\beta \to -\beta$ and take their difference and simplify.
\begin{multline}\label{eq:12.12}
\int_0^b \frac{\tanh ^{-1}(x \beta )}{\sqrt{1+x^2 z} \sqrt{1+x^2 \alpha }} \, dx\\
=\sum _{h=0}^{\infty } \sum
   _{l=0}^{\infty } \frac{(-1)^h b^{2+h+2 l} \alpha ^l \beta  \left((-\beta )^h+\beta ^h\right)
   \binom{-\frac{1}{2}}{l} \, _2F_1\left(\frac{1}{2},1+\frac{h}{2}+l;2+\frac{h}{2}+l;-b^2 z\right)}{2 (1+h) (2+h+2
   l)}
\end{multline}
where $|Re(k)|<1,|Re(\alpha)|<1$.
\end{example}
\begin{example}
Generalized form of equation (4.297.11) in \cite{grad}.  Use equation (\ref{eq:12.10}) and set $k\to 0,a\to 1,d\to \frac{1}{2},c\to 2,f\to \frac{1}{2},s\to 2,g\to 1,\gamma \to 1$ then take the indefinite integral with respect to $\beta$ and form a second equation by replacing $\beta \to -\beta$ and take the difference and simplify.
\begin{multline}\label{eq:12.13}
\int_0^b \frac{x^{-1+m} \tanh ^{-1}(x \beta )}{\sqrt{1+x^2 z} \sqrt{1+x^2 \alpha }} \, dx\\
=\sum _{j=0}^{\infty
   } \sum _{l=0}^{\infty } \frac{b^{2 j+2 l+m} \pi  z^j \alpha ^l }{(2 j+2 l+m) (1+2
   j+2 l+m) \Gamma \left(\frac{1}{2}-j\right) \Gamma (1+j) \Gamma \left(\frac{1}{2}-l\right) \Gamma (1+l)}\\ \times 
\left(\tanh ^{-1}(b \beta )+2 j \tanh ^{-1}(b
   \beta )+2 l \tanh ^{-1}(b \beta )+m \tanh ^{-1}(b \beta )\right. \\ \left. -b \beta  \,
   _2F_1\left(1,\frac{1}{2}+j+l+\frac{m}{2};\frac{3}{2}+j+l+\frac{m}{2};b^2 \beta ^2\right)\right)
\end{multline}
where $|Re(k)|<1,|Re(\alpha)|<1,|Re(m)|<1$.
\end{example}
\begin{example}
Use equation (\ref{eq:12.13}) and simplify the right-hand side with $j\in[0,\infty)$ and simplify.
\begin{multline}\label{eq:thm_gen}
\int_0^b \frac{x^{-1+m} \tanh ^{-1}(x \beta )}{\sqrt{1+x^2 z} \sqrt{1+x^2 \alpha }} \, dx\\
=\sum _{h=0}^{\infty
   } \sum _{l=0}^{\infty } \frac{(-1)^h \left(1+(-1)^h\right) b^{1+h+2 l+m} \sqrt{\pi } \alpha ^l \beta ^{1+h} }{2 (1+h) (1+h+2 l+m) \Gamma
   \left(\frac{1}{2}-l\right) \Gamma (1+l)}\\ \times
\,
   _2F_1\left(\frac{1}{2},\frac{1}{2} (1+h+2 l+m);\frac{1}{2} (3+h+2 l+m);-b^2 z\right)
\end{multline}
where $|Re(k)|<1,|Re(\alpha)|<1,|Re(m)|<1$.
\end{example}
\subsection{The incomplete elliptic integral of the third kind}
 This example is an extension of the Ward form in \cite{ward}. In 1960, Morgan Ward published work on the evaluation of an elliptic integral of the third kind and the difficulties associated with these types of integrals. In this example we present a generalized closed form solution for this elliptic integral in terms of the hypergeometric function allowing for easy evaluations after parameter substitutions. A similar form is also published on page 523 of Watson \cite{watson}. In 2013, Fukushima \cite{Fukushima} published work on a novel method to calculate an associate complete elliptic integral of the third kind. In 2017, Anakhev \cite{Anakhaev}, published work entitled, "Complete elliptic integrals of the third kind in problems of mechanics", where elementary analytical expressions for determining the values of the specified integrals, which is of importance for revealing the physical cause effect relationships in the problems under solution and in their subsequent mathematical transformations. In 1904, Greenhill \cite{greenhill} published work on the elliptic integral of the third kind, which makes its appearance in a dynamical problem, is of the circular form in Legendre's classification, and thus the Jacobian parameter is a fraction of the imaginary period, so that the expression by-means of the theta function can no longer be considered as reducing the variable elements from three to two. In 1965, Fettis \cite{fettis} worked on the "Calculation of Elliptic Integrals of the Third Kind by Means of Gauss' Transformation", where this work is  based on Gauss' modulus-reducing transformation, enabling one to write an efficient program for direct computation of the integral for any given values of its three arguments. In 1988, Carlson \cite{carlson1} published work on a collection of 72 elliptic integrals of the third kind in terms of $R$-functions. Some methods of computing elliptic integrals include the duplication method [DLMF,\href{https://dlmf.nist.gov/19.36.i}{19.36(i)}], quadratic transformations [DLMF,\href{https://dlmf.nist.gov/19.36.ii}{19.36(ii)}], theta functions [DLMF,\href{https://dlmf.nist.gov/19.36.iii}{19.36(iii)}], other methods [DLMF,\href{https://dlmf.nist.gov/19.36.iv}{19.36(iv)}], such as numerical quadrature, series expansions of Legendre's integrals, and addition theorems. In 1991, Jha et al., \cite{jha}, produced work on the "Computation of elliptic integrals of the third kind in high-frequency EM scattering problems," where they determined the computational efficiency of various formulae in relation to the nature of the scattering structures.
\begin{example}
The incomplete elliptic integral of the third kind see [Wolfram,\href{https://mathworld.wolfram.com/EllipticIntegraloftheThirdKind.html}{2}]. Use equation (\ref{eq:12.6}) and set $k\to 0,a\to 1,d\to 1,c\to 2,f\to \frac{1}{2},s\to 2,g\to \frac{1}{2},\gamma \to 2,m\to 0$ and simplify.\\\\
\begin{multline}\label{eq:12.15}
\int_0^b \frac{1}{\left(1+x^2 z\right) \sqrt{1+x^2 \alpha } \sqrt{1+x^2 \beta }} \, dx\\
=\sum _{h=0}^{\infty }
   \sum _{l=0}^{\infty } \frac{b^{1+2 h+2 l} \alpha ^l \beta ^h \binom{-\frac{1}{2}}{h} \binom{-\frac{1}{2}}{l} \,
   _2F_1\left(1,\frac{1}{2}+h+l;\frac{3}{2}+h+l;-b^2 z\right)}{1+2 h+2 l}
\end{multline}
where $|Re(z)|<1,|Re(\alpha)|<1,|Re(\beta)|<1$.
\end{example}

\begin{example}
Use equation (\ref{eq:12.15}) and simplify the right-hand side series starting with $h\in[0,\infty)$ and simplify.
\begin{multline}
\int_0^b \frac{1}{\left(1+x^2 z\right) \sqrt{1+x^2 \alpha } \sqrt{1+x^2 \beta }} \, dx\\
=\sum _{j=0}^{\infty }
   \sum _{l=0}^{\infty } \frac{(-1)^j b^{1+2 j+2 l} z^j \alpha ^l \binom{-\frac{1}{2}}{l} \,
   _2F_1\left(\frac{1}{2},\frac{1}{2}+j+l;\frac{3}{2}+j+l;-b^2 \beta \right)}{1+2 j+2 l}
\end{multline}
where $|Re(z)|<1,|Re(\alpha)|<1,|Re(\beta)|<1$.
\end{example}
\section{A generalized form of definite integral over a finite domain}
In this theorem we iterated the process of deriving the contour integral method in Theorem (\ref{eq:10.4}).
\begin{theorem}
For all $|Re(m)|<1$ then,
\begin{multline}
\int_0^b \sum\limits_{j_{1},..,j_{n}\geq 0} x^{m+\sum _{l=1}^n a_l j_l} \log ^k(a x) \prod _{l=1}^n \binom{-b_l}{j_l} c_l^{j_l} \,
   dx\\
=-\sum\limits_{j_{1},..,j_{n}\geq 0} a^{-1-m-\sum _{l=1}^n a_l j_l} \Gamma \left(1+k,-\log (a b) \left(1+m+\sum _{l=1}^n a_l j_l\right)\right)
   \left(\prod _{l=1}^n \binom{-b_l}{j_l} c_l^{j_l}\right)\\ \left(-1-m-\sum _{l=1}^n a_l j_l\right){}^{-1-k}
\end{multline}
\end{theorem}
\begin{example}
The elliptic integral of the first kind. Use equation (\ref{eq:10.4}) and set $k\to 0,a\to 1,d\to \frac{1}{2},f\to \frac{1}{2},g\to \frac{1}{2},c\to 1,s\to 1,\gamma \to 1,z\to -1,\alpha\to -\frac{1}{2},\beta \to -\frac{1}{3},m\to 0$ and simplify.
   \begin{multline}
F\left(\left.i \sinh ^{-1}(1)\right|2\right)-F\left(\left.i
   \text{csch}^{-1}\left(\sqrt{1-b}\right)\right|2\right)\\
=-\sum _{h=0}^{\infty } \sum _{l=0}^{\infty } \frac{i
   (-1)^{h+l} 2^{-\frac{3}{2}-l} 3^{-\frac{1}{2}-h} b^{1+h+l} \binom{-\frac{1}{2}}{h} \binom{-\frac{1}{2}}{l} \, _2F_1\left(\frac{1}{2},1+h+l;2+h+l;b\right)}{1+h+l}
\end{multline}
\end{example}
\begin{example}
Derivation of (4.1.5.6) in \cite{brychkov}. Use equation (\ref{eq:10.4}) and set $k\to 0,a\to 1,b\to 1,d\to 1-s,c\to 1,z\to -1,f\to s+1,s\to 1,\alpha \to -z,g\to 1,\gamma \to 1$ then take the indefinite integral with respect to $\beta$ and set $\beta=-1,m=1$ and simplify.
\begin{multline}
\int_0^1 \frac{(1-x)^{s-1} \log (1-x)}{(1-x z)^{s+1}} \, dx\\
=-\sum _{h=0}^{\infty } \frac{\Gamma (3+h) \Gamma
   (s) \, _2F_1(2+h,1+s;2+h+s;z)}{(1+h) (2+h) \Gamma (2+h+s)}\\
=\frac{\, _2F_1(1,s;s+1;z)}{(z-1) s^2}
\end{multline}
where $Re(s)>0,|1-z|<\pi$.
\end{example}
\begin{example}
Derivation of a generalized form of (4.1.5.45) in \cite{brychkov}. Use equation (\ref{eq:10.4}) and set $k\to 0,a\to 1,b\to 1,d\to 1-s,c\to 2,z\to -1,f\to s+\frac{1}{2},s\to 2,\alpha \to -a,g\to 1,\gamma \to
   1$ then take the indefinite integral with respect to $\beta$ and set $\beta=-\beta$ to form a second equation, then take their difference and simplify.
   \begin{multline}
\int_0^1 \frac{x^{m-1} \left(1-x^2\right)^{s-1} \tanh ^{-1}(x \beta )}{\left(1-a
   x^2\right)^{s+\frac{1}{2}}} \, dx\\
=\sum _{l=0}^{\infty } \frac{(-1)^l a^l \beta  \binom{-\frac{1}{2}-s}{l} \Gamma
   \left(\frac{1}{2} (3+2 l+m)\right) \Gamma (s) \,
   _3F_2\left(\frac{1}{2},1,\frac{1}{2}+l+\frac{m}{2};\frac{3}{2},\frac{1}{2}+l+\frac{m}{2}+s;\beta ^2\right)}{(1+2
   l+m) \Gamma \left(\frac{1}{2} (1+2 l+m+2 s)\right)}
\end{multline}
where $Re(s)>0,|1-z|<\pi$.
\end{example}
\begin{example}
Generalized form of equation (4.531.13) in \cite{grad}. Use equation (\ref{eq:10.4}) and set $k\to 0,a\to 1,b\to 1,d\to \frac{1}{2},c\to 2,f\to \frac{1}{2},s\to 2,g\to 1,\gamma \to 1$ then simplify the series over $j\in[0,\infty)$ then take the indefinite integral with respect to $\beta$. Next form a second equation by replacing $\beta\to -\beta$ and take their difference and simplify. To get the second double series, after the initial parameter substitutions simplify the series in terms of $h\in [0,\infty)$ and repeat the steps above and simplify.
\begin{multline}
\int_0^1 \frac{x^{-1+m} \tanh ^{-1}(x \beta )}{\sqrt{1+x^2 z} \sqrt{1+x^2 \alpha }} \, dx
=\sum _{h=0}^{\infty
   } \sum _{l=0}^{\infty } \frac{(-1)^h \left(1+(-1)^h\right) \alpha ^l \beta ^{1+h} \binom{-\frac{1}{2}}{l} }{2 (1+h) (1+h+2l+m)}\\ \times \,
   _2F_1\left(\frac{1}{2},\frac{1}{2} (1+h+2 l+m);\frac{1}{2} (3+h+2 l+m);-z\right)\\
=\frac{1}{2} \left(\sum _{l=0}^{\infty } \frac{2 \sqrt{\pi } \alpha ^l \tanh ^{-1}(\beta ) \,
   _2F_1\left(\frac{1}{2},l+\frac{m}{2};1+l+\frac{m}{2};-z\right)}{(2 l+m) \Gamma \left(\frac{1}{2} (1-2 l)\right)
   \Gamma (1+l)}\right. \\ \left.
-\sum _{j=0}^{\infty } \sum _{l=0}^{\infty } \frac{\pi  z^j \alpha ^l \beta  }{(2 j+2 l+m) (1+2 j+2 l+m) \Gamma
   \left(\frac{1}{2}-j\right) \Gamma (1+j) \Gamma \left(\frac{1}{2}-l\right) \Gamma (1+l)}\right. \\ \left. \times(\, _2F_1(1,1+2 j+2l+m;2 (1+j+l)+m;-\beta )\right. \\ \left.
+\, _2F_1(1,1+2 j+2 l+m;2 (1+j+l)+m;\beta ))\right)
\end{multline}
where $Re(m)>0,|Re(z)|<1,|Re(\alpha)|<1$.
\end{example}
\begin{example}
A generalized elliptic integral form involving the product of logarithm and arctangent functions. Use equation (\ref{eq:10.4}) and set $a\to 1,b\to 1,d\to \frac{1}{2},c\to 2,f\to \frac{1}{2},s\to 2,g\to 1,\gamma \to 1$ then take the indefinite integral with respect to $\beta$ and simplify.
\begin{multline}
\int_0^1 \frac{x^{-1+m} \tanh ^{-1}(x \beta ) \log ^k(x)}{\sqrt{1+x^2 z} \sqrt{1+x^2 \alpha }} \,
   dx\\
=\frac{\beta}{2}   \Gamma (1+k) \sum _{h=0}^{\infty } \sum _{j=0}^{\infty } \sum _{l=0}^{\infty }
   \frac{(-1)^h \left(1+(-1)^h\right) z^j \alpha ^l \beta ^h \binom{-\frac{1}{2}}{j} \binom{-\frac{1}{2}}{l}}{(1+h)
   (-1-h-2 j-2 l-m)^{k+1}}
\end{multline}
where $Re(m)>0,|Re(z)|<1,|Re(\alpha)|<1$.
\end{example}
\begin{example}
Generalized form of equation (2.6.12.9) in \cite{prud}. Use equation (\ref{eq:10.4}) and set $k\to 0,a\to 1,g\to 1,\gamma \to 1$ then take the indefinite integral with respect to $\beta$ and replace $d\to -\gamma,f\to \mu$ and simplify.
\begin{multline}
\int_0^b \frac{x^{m-1} \left(1+x^c z\right)^{\gamma -1} \log (1+x \beta )}{\left(1+x^s \alpha \right)^{\mu }}
   \, dx\\
=\sum _{h=0}^{\infty } \sum _{l=0}^{\infty } \frac{(-1)^h b^{1+h+m+l s} \alpha ^l \beta ^{1+h} \binom{-\mu
   }{l}}{(1+h) (1+h+m+l s)}\\
 \times  \, _2F_1\left(\frac{1}{c}+\frac{h}{c}+\frac{m}{c}+\frac{l s}{c},1-\gamma
   ;1+\frac{1}{c}+\frac{h}{c}+\frac{m}{c}+\frac{l s}{c};-b^c z\right)
\end{multline}
where $Re(m)>0,|Re(b)|\leq 1$.
\end{example}
\begin{example}
Bierens de haan form for equation (2.6.17.9) in \cite{prud}. Use equation (\ref{eq:10.4}) and set $a\to 1,b\to 1,g\to 0,m\to m-1,d\to -d,f\to -f,\beta \to 1,\gamma \to 0$ next apply l'Hopital's rule as $k\to -1$ then form a second equation by replacing $m\to r$ and take their difference and simplify.
\begin{multline}
\int_0^1 \frac{\left(x^m-x^r\right) \left(1+x^c z\right)^d \left(1+x^s \alpha \right)^f}{x \log (x)} \,
   dx=\sum _{j=0}^{\infty } \sum _{l=0}^{\infty } z^j \alpha ^l \binom{d}{j} \binom{f}{l} \log \left(\frac{-c j-m-l
   s}{-c j-r-l s}\right)
\end{multline}
where $|Re(m)|<1,|Re(r)|<1$.
\end{example}
\section{Definite integrals involving series of the Hurwitz-Lerch zeta function}
In this section we look at a special case of equation (\ref{eq:10.4}) involving the definite integral of the product of rational and logarithm functions in terms of the Hurwitz-Lerch zeta function. Special cases in terms of other special functions and fundamental constants.
\begin{example}
Definite integral in terms of the infinite sum of the Hurwitz-Lerch zeta function. Use equation (\ref{eq:10.4}) and set $a=b=d=1,f\to -f,g\to -g$ and simplify using [DLMF,\href{https://dlmf.nist.gov/25.14.i}{25.14.1}].
\begin{multline}\label{eq:15.1}
\int_0^1 \frac{x^{m-1} \left(1+x^s \alpha \right)^f \left(1+x^{\gamma } \beta \right)^g \log
   ^k\left(\frac{1}{x}\right)}{1+x^c z} \, dx\\
=\frac{\Gamma (1+k) }{c^{k+1}}\sum _{h=0}^{\infty } \sum _{l=0}^{\infty } \alpha
   ^l \beta ^h \binom{f}{l} \binom{g}{h} \Phi \left(-z,1+k,-\frac{-m-l s-h \gamma }{c}\right)
\end{multline}
where $|Re(m)|<1,Re(c)>0, Re(s)>0, Re(\gamma)>0$.
\end{example}
\begin{example}
The Hurwitz-zeta function. Use equation (\ref{eq:15.1}) and set $z=1$ and simplify using [DLMF,\href{https://dlmf.nist.gov/25.14.E2}{25.14.2}].
\begin{multline}\label{eq:15.2}
\int_0^1 \frac{x^{-1+m} \left(1+x^s \alpha \right)^f \left(1+x^{\gamma } \beta \right)^g \log
   ^k\left(\frac{1}{x}\right)}{1+x^c} \, dx\\
=c^{-1-k}\Gamma (1+k)\sum _{h=0}^{\infty } \sum _{l=0}^{\infty }  \alpha ^l \beta ^h
   \binom{f}{l} \binom{g}{h}  \left(2^{-1-k} \zeta \left(1+k,-\frac{-m-l s-h \gamma }{2
   c}\right)\right. \\ \left.
-2^{-1-k} \zeta \left(1+k,\frac{1}{2} \left(1-\frac{-m-l s-h \gamma }{c}\right)\right)\right)
\end{multline}
where $|Re(m)|<1,Re(c)>0, Re(s)>0, Re(\gamma)>0$.
\end{example}
\begin{example}
Malmsten-type definite integral. Use equation (\ref{eq:15.2}) and take the first partial derivative with respect to $k$ and set $k=0$ and simplify using [DLMF,\href{https://dlmf.nist.gov/25.11.E17}{25.11.18 	}].
\begin{multline}
\int_0^1 \frac{x^{-1+m} \left(1+x^s \alpha \right)^f \left(1+x^{\gamma } \beta \right)^g \log \left(\log
   \left(\frac{1}{x}\right)\right)}{1+x^c} \, dx\\
=\sum _{h=0}^{\infty } \sum _{l=0}^{\infty } \frac{\alpha ^l \beta
   ^h \Gamma (1+f) \Gamma (1+g) }{2 c \Gamma (1+g-h) \Gamma (1+h) \Gamma(1+f-l) \Gamma (1+l)}\\ \times \left((\gamma +\log (2 c)) \left(\psi ^{(0)}\left(\frac{m+l s+h \gamma }{2c}\right)-\psi ^{(0)}\left(\frac{c+m+l s+h \gamma }{2 c}\right)\right)\right. \\ \left.
-\gamma _1\left(\frac{m+l s+h \gamma }{2c}\right)+\gamma _1\left(\frac{c+m+l s+h \gamma }{2 c}\right)\right)
\end{multline}
where $|Re(m)|<1,Re(c)>0, Re(s)>0, Re(\gamma)>0$.
\end{example}
\begin{example}
Definite integral in terms of the finite sum of the Hurwitz-Lerch zeta function. Use equation (\ref{eq:15.1}) and simplify the series on the right-hand side.
\begin{multline}
\int_0^1 \frac{x^{m-1} \left(1+x^s \alpha \right)^f \left(1+x^{\gamma } \beta \right)^g \log
   ^k\left(\frac{1}{x}\right)}{1+x^c z} \, dx\\
=\frac{\Gamma (1+k) }{c^{k+1}}\sum _{h=0}^g \sum _{l=0}^f \alpha ^l \beta ^h
   \binom{f}{l} \binom{g}{h} \Phi \left(-z,1+k,-\frac{-m-l s-h \gamma }{c}\right)
\end{multline}
where $|Re(m)|<1,Re(c)>0, Re(s)>0, Re(\gamma)>0$.
\end{example}
\begin{example}
Definite integral involving the reciprocal ,logarithm function. Use equation (\ref{eq:10.4}) and set $k=-1,b=1$ and simplify.
\begin{multline}
\int_0^1 \frac{x^m \left(1+x^c z\right)^{-d} \left(1+x^s \alpha
   \right)^{-f} \left(1+x^{\gamma } \beta \right)^{-g}}{\log (a)+\log (x)} \,
   dx\\
=-\sum _{j=0}^{\infty } \sum _{h=0}^{\infty } \sum _{l=0}^{\infty }
   a^{-1-c j-m-l s-h \gamma } z^j \alpha ^l \beta ^h \binom{-d}{j}
   \binom{-f}{l} \binom{-g}{h} \\
\Gamma (0,-((1+c j+m+l s+h \gamma ) \log
   (a)))
\end{multline}
where $|Re(m)|<1,Re(c)>0, Re(s)>0, Re(\gamma)>0$.
\end{example}
\begin{example}
Use equation (\ref{eq:15.2}) and simplify using equation (\ref{eq:hurwitz}). When using the polylogarithm function definition, then $Re(c)>1$. One advantage of the polylogarithm form, is the lack of the gamma function. This allows for negative integers in the parameter $k$ by applying l'Hopital's rule.
 \begin{multline}
\int_0^1 \frac{x^{-1+m} \left(1+x^s \alpha \right)^f \left(1+x^{\gamma } \beta \right)^g \log
   ^k\left(\frac{1}{x}\right)}{1-x^c} \, dx\\
=\sum _{h=0}^{\infty } \sum _{l=0}^{\infty } c^{-1-k} \alpha ^l \beta ^h
   \binom{f}{l} \binom{g}{h} \Gamma (1+k) \zeta \left(1+k,-\frac{-m-l s-h \gamma }{c}\right)\\
=-(2 i)^k c^{-1-k} \pi
   ^{1+k} \csc (k \pi ) \sum _{h=0}^{\infty } \sum _{l=0}^{\infty } \alpha ^l \beta ^h \binom{f}{l} \binom{g}{h}\\
   \left(\text{Li}_{-k}\left(-e^{-\frac{i \pi  (c-2 (m+l s+h \gamma ))}{c}}\right)+(-1)^{-k}
   \text{Li}_{-k}\left(-e^{\frac{i \pi  (c-2 (m+l s+h \gamma ))}{c}}\right)\right)
\end{multline}
where $|Re(m)|<1,Re(c)>0, Re(s)>0, Re(\gamma)>0$.
\end{example}
\begin{example}
Mellin transform of the logarithmic form of the elliptic integral of the third kind in terms of the Hurwitz-Lerch zeta function. Use equation (\ref{eq:10.4}) and set $d\to \frac{1}{2},f\to \frac{1}{2},g\to 1,a\to 1,b\to 1,c\to 2,s\to 2,\gamma \to 2,m\to m-1$ and simplify.
\begin{multline}\label{eq:15.7}
\int_0^1 \frac{x^{-1+m} \log ^k\left(\frac{1}{x}\right)}{\sqrt{1+x^2 z} \sqrt{1+x^2 \alpha } \left(1+x^2
   \beta \right)} \, dx\\
=2^{-1-k} \Gamma (1+k) \sum _{j=0}^{\infty } \sum _{l=0}^{\infty } z^j \alpha ^l
   \binom{-\frac{1}{2}}{j} \binom{-\frac{1}{2}}{l} \Phi \left(-\beta ,1+k,\frac{1}{2} (2 j+2 l+m)\right)
\end{multline}
where $|Re(m)|<1,|Re(z)|<1,|Re(\alpha)|<1,|Re(\beta)|<1$.
\end{example}
\begin{example}
Malmsten-form of the elliptic integral of the third kind, where $Re(\beta)>1,0< Re(m)<1$. Use equation (\ref{eq:15.7}) and take the first partial derivative with respect to $k$ and set $k=0$ and simplify.
\begin{multline}
\int_0^1 \frac{x^{-1+m} \log \left(\log \left(\frac{1}{x}\right)\right)}{\sqrt{1+x^2 z} \sqrt{1+x^2 \alpha }
   \left(1+x^2 \beta \right)} \, dx\\
=-\frac{1}{2} \sum _{j=0}^{\infty } \sum _{l=0}^{\infty } z^j \alpha ^l
   \binom{-\frac{1}{2}}{j} \binom{-\frac{1}{2}}{l} \left(\Phi \left(-\beta ,1,j+l+\frac{m}{2}\right) (\gamma +\log
   (2))\right. \\ \left.
-\Phi'\left(-\beta ,1,j+l+\frac{m}{2}\right)\right)
\end{multline}
where $0< Re(m)<1,|Re(z)|<1,|Re(\alpha)|<1,|Re(\beta)|>1$.
\end{example}
\section{Exercises}
Use the method in sections (11.1-11.2) with the appropriate coefficient functions and variables to derive the following examples. 
\begin{example}
The logarithmic transform of the hypergeometric function.
\begin{multline}
\int_0^b x^m \, _2F_1(\alpha ,\beta ;\gamma ;x) \log ^k(a x) \, dx=\sum _{l=0}^{\infty } \frac{a^{-1-l-m}
   \Gamma (1+k,-((1+l+m) \log (a b))) (\alpha )_l (\beta )_l}{l! (\gamma )_l (-1-l-m)^{k+1}}
\end{multline}
where $0< Re(m)<1,Re(\alpha)>0,Re(\beta)>0,Re(\gamma)>0$.
\end{example}
\begin{example}
Malmsten type integral involving the hypergeometric function.
\begin{multline}
\int_0^1 \, _2F_1(\alpha ,\beta ;\gamma ;x) \log \left(\log \left(\frac{1}{x}\right)\right) \,
   dx\\
=-\frac{\gamma  (-1+\gamma ) (-1+\, _2F_1(-1+\alpha ,-1+\beta ;-1+\gamma ;1))}{(-1+\alpha ) (-1+\beta )}\\
-\sum
   _{l=0}^{\infty } \frac{\log (1+l) (\alpha )_l (\beta )_l}{(1+l) l! (\gamma )_l}
\end{multline}
where $0< Re(\alpha)<1,0< Re(\beta)<1,Re(\gamma)>2$.
\end{example}
\begin{example}
The logarithmic transform of the Hurwitz-Lerch zeta function.
\begin{multline}
\int_0^b x^m \Phi (x,s,v) (-\log (a x))^k \, dx=\frac{1}{a}\sum _{l=0}^{\infty } \frac{\Gamma (1+k,-((1+l+m) \log
   (a b)))}{a^{l+m} (1+l+m)^{k+1} (l+v)^s}
\end{multline}
where $Re(k)>0,Re(s)>0,|Re(m)|<1$.
\end{example}
\begin{example}
Definite integral representation for the Hurwitz zeta function.
\begin{equation}
\frac{1}{\Gamma (1+k)}\int_0^1 x^{-1+v} \Phi (x,s,v) \log ^k\left(\frac{1}{x}\right) \, dx=\zeta
   (1+k+s,v)
\end{equation}
where $Re(k)>0,Re(s)>0$.
\end{example}
\begin{example}
Definite integral representation for the  Riemann zeta function. Similar form is given by Arakawa et al. in \cite{arakawa}.
\begin{equation}
\frac{1}{\Gamma (k)}\int_0^1 \frac{\log ^{k-1}\left(\frac{1}{x}\right) \text{Li}_s(x)}{x} \, dx=\zeta
   (k+s)
\end{equation}
where $Re(k)>0,Re(s)>0$.
\end{example}
\section{Special functions and integrals of the general form}
\subsection{The hypergeometric function $\, _2F_1(\alpha ,\beta ;\gamma ;x)$}
The contour integral representation form involving the hypergeometric function is given by;
\begin{multline}\label{eq:h1}
\frac{1}{2\pi i}\int_{C}\int_{0}^{b}a^w w^{-k-1} x^{m+w} \, _2F_1(\alpha ,\beta ;\gamma ;x)dxdw\\=\frac{1}{2\pi i}\int_{C}\frac{a^w w^{-k-1} b^{m+w+1} \, _3F_2(m+w+1,\alpha,\beta ;m+w+2,\gamma ;b)}{m+w+1}dw
\end{multline}
where $Re(\alpha+\beta)>0$.
\subsubsection{Left-hand side contour integral representation}
Using a generalization of Cauchy's integral formula \ref{intro:cauchy}, we form the definite integral by replacing $y$ by $\log{ax}$ and multiply both sides by $\frac{x^{l+m} (\alpha )_l (\beta )_l}{l! (\gamma )_l}$ then take the infinite sum of both sides over $l \in [0,\infty)$ to get;
\begin{multline}\label{eq:h2}
\int_{0}^{b}\frac{x^m \log ^k(a x) \, _2F_1(\alpha ,\beta ;\gamma ;x)}{k!}dx=\frac{1}{2\pi i}\int_{0}^{b}\int_{C}w^{-k-1} x^m (a x)^w \, _2F_1(\alpha ,\beta;\gamma ;x)dwdx\\
=\frac{1}{2\pi i}\int_{C}\int_{0}^{b}w^{-k-1} x^m (a x)^w \, _2F_1(\alpha ,\beta;\gamma ;x)dxdw
\end{multline}
where $Re(\alpha+\beta)>0$.
We are able to switch the order of integration over $t$ and $w$ using Fubini's theorem for multiple integrals see page 178 in \cite{gelca}, since the integrand is of bounded measure over the space $\mathbb{C} \times [0,1]$.
\subsubsection{Right-hand side contour integral representation}
Use equation (\ref{eq:4.4}) and set $y=0$, replace $m\to l+m+1,a\to a b$ and multiply both sides by $\frac{b^{l+m+1} (\alpha )_l (\beta )_l}{l! (\gamma )_l}$. Next take the infinite sum of both sides over $l \in [0,\infty)$ and simplify to get;
\begin{multline}\label{eq:h3}
\sum _{l=0}^{\infty } \frac{b^{1+l+m} (a b)^{-1-l-m} (-1-l-m)^{-1-k} \Gamma (1+k,-((1+l+m) \log (a b)))
   (\alpha )_l (\beta )_l}{k! l! (\gamma )_l}\\
=-\frac{1}{2\pi i}\sum _{l=0}^{\infty } \int_{C}\frac{w^{-k-1} (a b)^w b^{l+m+1} (\alpha )_l (\beta )_l}{l! (l+m+w+1) (\gamma )_l}dw\\
=-\frac{1}{2\pi i} \int_{C}\sum _{l=0}^{\infty }\frac{w^{-k-1} (a b)^w b^{l+m+1} (\alpha )_l (\beta )_l}{l! (l+m+w+1) (\gamma )_l}dw\\
=-\frac{1}{2\pi i}\int_{C}\frac{b^{1+m} (a b)^w w^{-1-k} \, _3F_2(1+m+w,\alpha ,\beta
   ;2+m+w,\gamma ;b)}{1+m+w}dw
\end{multline}
from equation [Wolfram, \href{http://functions.wolfram.com/07.27.02.0002.01}{07.27.02.0002.01}] where $|Re(b)|<1$. We are able to switch the order of integration and summation over $w$ using Tonellii's theorem for  integrals and sums see page 177 in \cite{gelca}, since the summand is of bounded measure over the space $\mathbb{C} \times [0,\infty)$
\begin{theorem}
Finite integral of the hypergeometric function.
\begin{multline}\label{eq:h4}
\int_0^b x^m \, _2F_1(\alpha ,\beta ;\gamma ;x) \log ^k(a x) \, dx=\sum _{l=0}^{\infty } \frac{a^{-1-l-m}
   \Gamma (1+k,-((1+l+m) \log (a b))) (\alpha )_l (\beta )_l}{l! (\gamma )_l (-1-l-m)^{k+1}}
\end{multline}
where $0< Re(m)<1,Re(\alpha)>0,Re(\beta)>0,Re(\gamma)>0$.
\end{theorem}
\begin{proof}
Since the right-hand sides of (\ref{eq:h2}) and (\ref{eq:h3}) are equivalent relative to (\ref{eq:h1}) we may equate the left-hand sides and simplify the gamma function to yield the stated result.
\end{proof}
\begin{example}
 Use equation (\ref{eq:h4}) and set $a=b=1$ and simplify.
\begin{multline}
\int_0^1 x^m \, _2F_1(\alpha ,\beta ;\gamma ;x) \log ^k\left(\frac{1}{x}\right) \, dx=\Gamma (1+k) \sum
   _{l=0}^{\infty } \frac{(\alpha )_l (\beta )_l}{l! (\gamma )_l (l+m+1)^{k+1}}
\end{multline}
where $Re(\alpha+\beta)>0$.
\end{example}
\begin{example}
Malmsten logarithm form of a definite integral see [Blagouchine, \href{https://doi.org/10.1007/s11139-013-9528-5}{1}]. Use equation (\ref{eq:h4}) and set $m=0,a=b=1$ then take the first partial derivative with respect to $k$ and set $k=0$ and simplify.
\begin{multline}
\int_0^1 \, _2F_1(\alpha ,\beta ;\gamma ;x) \log \left(\log \left(\frac{1}{x}\right)\right) \,
   dx\\
=-\frac{\gamma  (-1+\gamma ) (-1+\, _2F_1(-1+\alpha ,-1+\beta ;-1+\gamma ;1))}{(-1+\alpha ) (-1+\beta )}\\
-\sum
   _{l=0}^{\infty } \frac{\log (1+l) (\alpha )_l (\beta )_l}{(1+l) l! (\gamma )_l}
\end{multline}
where $0< Re(\alpha)<1,0< Re(\beta)<1,Re(\gamma)>2$.
\end{example}
\subsection{The Hurwitz-Lerch zeta function $\Phi(z,s,a)$}
The contour integral representation form involving the Hurwitz-Lerch zeta function function [DLMF, \href{https://dlmf.nist.gov/25.14}{25.14}] is given by;
.\begin{multline}\label{eq:hu1}
\frac{1}{2\pi i}\int_{C}\int_{0}^{b}a^w w^{-1-k} x^{m+w} \Phi (x,s,v)dxdw=\frac{1}{2\pi i}\int_{C}\sum _{l=0}^{\infty } \frac{a^w b^{1+l+m+w} (l+v)^{-s}
   w^{-1-k}}{1+l+m+w}dw
\end{multline}
where $Re(v)>0$.
\subsubsection{Left-hand side contour integral representation}
Using a generalization of Cauchy's integral formula \ref{intro:cauchy}, we form the definite integral by replacing $y$ by $\log{ax}$ and multiply both sides by $x^{l+m} (l+v)^{-s}$ then take the infinite sum of both sides over $l \in [0,\infty)$ to get;
\begin{multline}\label{eq:hu2}
\int_{0}^{b}\frac{x^m \log ^k(a x) \Phi (x,s,v)}{k!}dx\\
=\frac{1}{2\pi i}\sum_{l=0}^{\infty}\int_{0}^{b}\int_{C}w^{-k-1} (a x)^w x^{l+m} (l+v)^{-s}dwdx\\
=\frac{1}{2\pi i}\int_{0}^{b}\int_{C}\sum_{l=0}^{\infty}w^{-k-1} (a x)^w x^{l+m} (l+v)^{-s}dwdx\\
=\frac{1}{2\pi i}\int_{0}^{b}\int_{C}w^{-k-1} x^m (a x)^w \Phi (x,s,v)dwdx\\
=\frac{1}{2\pi i}\int_{C}\int_{0}^{b}w^{-k-1} x^m (a x)^w \Phi (x,s,v)dxdw
\end{multline}
where $Re(v)>0$.
We are able to switch the order of integration over $t$ and $w$ using Fubini's theorem for multiple integrals see page 178 in \cite{gelca}, since the integrand is of bounded measure over the space $\mathbb{C} \times [0,1]$.
\subsubsection{Right-hand side contour integral representation}
Use equation (\ref{eq:4.4}) and set $y=0$, replace $m\to l+m+1,a\to a b$ and multiply both sides by $b^{l+m+1} (l+v)^{-s}$. Next take the infinite sum of both sides over $l \in [0,\infty)$ and simplify to get;
\begin{multline}\label{eq:hu3}
\frac{1}{2\pi i}\int_{C}\sum _{l=0}^{\infty } \frac{b^{1+l+m} (a b)^{-1-l-m} (-1-l-m)^{-1-k} (l+v)^{-s} \Gamma (1+k,-((1+l+m) \log (ab)))}{k!}\\
=-\frac{1}{2\pi i}\sum _{l=0}^{\infty }\int_{C} \frac{b^{1+l+m} (a b)^w (l+v)^{-s} w^{-1-k}}{1+l+m+w}\\
=-\frac{1}{2\pi i}\int_{C}\sum _{l=0}^{\infty } \frac{b^{1+l+m} (a b)^w (l+v)^{-s} w^{-1-k}}{1+l+m+w}
\end{multline}
where $Re(v)>0$.
where $|Re(b)|<1$. We are able to switch the order of integration and summation over $w$ using Tonellii's theorem for  integrals and sums see page 177 in \cite{gelca}, since the summand is of bounded measure over the space $\mathbb{C} \times [0,\infty)$
\begin{theorem}
\begin{multline}\label{eq:hu4}
\int_0^b x^m \Phi (x,s,v) (-\log (a x))^k \, dx=\frac{1}{a}\sum _{l=0}^{\infty } \frac{\Gamma (1+k,-((1+l+m) \log
   (a b)))}{a^{l+m} (1+l+m)^{k+1} (l+v)^s}
\end{multline}
where $Re(k)>0,Re(s)>0,|Re(m)|<1$.
\end{theorem}
\begin{proof}
Since the right-hand sides of (\ref{eq:hu2}) and (\ref{eq:hu3}) are equivalent relative to (\ref{eq:hu1}) we may equate the left-hand sides and simplify the gamma function to yield the stated result.
\end{proof}
\begin{example}
 The Hurwitz zeta function $\zeta(s,a)$. Use equation (\ref{eq:hu4}) and set $a=b=1,m\to v-1$ and simplify.
\begin{equation}\label{eq:hur1}
\frac{1}{\Gamma (1+k)}\int_0^1 x^{-1+v} \Phi (x,s,v) \log ^k\left(\frac{1}{x}\right) \, dx=\zeta
   (1+k+s,v)
\end{equation}
where $Re(k)>0,Re(s)>0$.
\end{example}
\begin{example}
The Riemann zeta function $\zeta(k+s)$ see \cite{arakawa} Use equation (\ref{eq:hur1}) and set $v=1$ and simplify.
\begin{equation}
\frac{1}{\Gamma (k)}\int_0^1 \frac{\log ^{k-1}\left(\frac{1}{x}\right) \text{Li}_s(x)}{x} \, dx=\zeta
   (k+s)
\end{equation}
where $Re(k)>0,Re(s)>0$.
\end{example}
\subsection{The product of Bessel functions}
The contour integral representation form involving the  function is given by;
\begin{multline}\label{eq:b1}
\frac{1}{2\pi i}\int_{C}\int_{0}^{b}\frac{2^{-v-\mu } a^w w^{-1-k} x^{m+w} \left(1+c x^p\right)^d z^{v+\mu } }{\Gamma (1+v) \Gamma (1+\mu )}\\ \times \,_2F_3\left(\frac{1}{2}+\frac{v}{2}+\frac{\mu }{2},1+\frac{v}{2}+\frac{\mu }{2};1+v,1+\mu ,1+v+\mu ;-x
   z^2\right)dxdw\\
=\frac{1}{2\pi i}\int_{C}\sum _{l=0}^{\infty } \frac{2^{-v-\mu } a^w b^{1+m+l p+w} c^l w^{-1-k}
   z^{v+\mu } \binom{d}{l} }{(1+m+l p+w) \Gamma (1+v) \Gamma (1+\mu )}\\ \times \, _3F_4\left(1+m+l p+w,\frac{1}{2}+\frac{v}{2}+\frac{\mu }{2},1+\frac{v}{2}+\frac{\mu }{2};1+v,2+m+l p+w,1+\mu ,1+v+\mu ;-b z^2\right)dw
\end{multline}
where $|Re(z)|<1$.
\subsubsection{Left-hand side contour integral representation}
Using a generalization of Cauchy's integral formula \ref{intro:cauchy}, we form the definite integral by replacing $y$ by $\log{ax}$ and multiply both sides by $\frac{(-1)^j c^l \binom{d}{l} 2^{-2 j-\mu -v} (j+v+\mu +1)_j z^{2 j+\mu +v} x^{j+l p+m}}{j! \Gamma (j+\mu +1) \Gamma
   (j+v+1)}$ then take the infinite sums of both sides over $j \in [0,\infty),l \in [0,\infty)$ to get;
\begin{multline}\label{eq:b2}
\int_{0}^{b}\frac{x^m z^{\mu +v} \log ^k(a x) \left(c x^p+1\right)^d \left(\sqrt{x} z\right)^{-\mu -v} J_v\left(\sqrt{x}
   z\right) J_{\mu }\left(\sqrt{x} z\right)}{\Gamma (k+1)}dx\\
=\frac{1}{2\pi i}\sum_{j=0}^{\infty}\sum_{l=0}^{\infty}\int_{0}^{b}\int_{C}\frac{(-1)^j c^l w^{-k-1} (a x)^w \binom{d}{l} 2^{-2 j-\mu -v} (j+v+\mu +1)_j
   z^{2 j+\mu +v} x^{j+l p+m}}{j! \Gamma (j+\mu +1) \Gamma (j+v+1)}dwdx\\
=\frac{1}{2\pi i}\int_{0}^{b}\int_{C}\sum_{j=0}^{\infty}\sum_{l=0}^{\infty}\frac{(-1)^j c^l w^{-k-1} (a x)^w \binom{d}{l} 2^{-2 j-\mu -v} (j+v+\mu +1)_j
   z^{2 j+\mu +v} x^{j+l p+m}}{j! \Gamma (j+\mu +1) \Gamma (j+v+1)}dwdx\\
=\frac{1}{2\pi i}\int_{0}^{b}\int_{C}w^{-k-1} x^m (a x)^w z^{\mu +v} \left(c x^p+1\right)^d\left(\sqrt{x} z\right)^{-\mu -v} \left(x z^2\right)^{\frac{\mu }{2}+\frac{1}{2} (-\mu -v)+\frac{v}{2}}J_v\left(\sqrt{x} z\right) J_{\mu }\left(\sqrt{x} z\right)dwdx\\
=\frac{1}{2\pi i}\int_{C}\int_{0}^{b}w^{-k-1} x^m (a x)^w z^{\mu +v} \left(c x^p+1\right)^d\left(\sqrt{x} z\right)^{-\mu -v} \left(x z^2\right)^{\frac{\mu }{2}+\frac{1}{2} (-\mu -v)+\frac{v}{2}}J_v\left(\sqrt{x} z\right) J_{\mu }\left(\sqrt{x} z\right)dxdw
\end{multline}
where $|Re(z)|<1$.
We are able to switch the order of integration over $t$ and $w$ using Fubini's theorem for multiple integrals see page 178 in \cite{gelca}, since the integrand is of bounded measure over the space $\mathbb{C} \times [0,b]$.
\subsubsection{Right-hand side contour integral representation}
Use equation (\ref{eq:4.4}) and set $y=0$, replace $m -> m + 1 + j + l p, a -> a b$ and multiply both sides by $\frac{(-1)^j c^l \binom{d}{l} 2^{-2 j-\mu -v} (j+v+\mu +1)_j z^{2 j+\mu +v} b^{j+l p+m+1}}{j! \Gamma (j+\mu +1)
   \Gamma (j+v+1)}$. Next take the infinite sums of both sides over $j \in [0,\infty),l \in [0,\infty)$ and simplify to get;
\begin{multline}\label{eq:b3}
\sum _{l=0}^{\infty } \sum _{j=0}^{\infty } \frac{(-1)^j 2^{-2 j-v-\mu }c^l  z^{2 j+v+\mu } \binom{d}{l} \Gamma (1+k,-((1+j+m+l p) \log (a b))) (1+j+v+\mu )_j}{j! k! \Gamma (1+j+v)\Gamma (1+j+\mu )(-1-j-m-lp)^{k+1} a^{1+j+m+l p} }\\
=-\frac{1}{2\pi i}\sum _{l=0}^{\infty } \sum _{j=0}^{\infty }\int_{C}\frac{(-1)^j c^l
   w^{-k-1} (a b)^w \binom{d}{l} 2^{-2 j-\mu -v} (j+v+\mu +1)_j z^{2 j+\mu +v} b^{j+l p+m+1}}{j! \Gamma (j+\mu +1)\Gamma (j+v+1) (j+l p+m+w+1)}\\
=-\frac{1}{2\pi i}\int_{C}\sum _{l=0}^{\infty } \sum _{j=0}^{\infty }\frac{(-1)^j c^l
   w^{-k-1} (a b)^w \binom{d}{l} 2^{-2 j-\mu -v} (j+v+\mu +1)_j z^{2 j+\mu +v} b^{j+l p+m+1}}{j! \Gamma (j+\mu +1)\Gamma (j+v+1) (j+l p+m+w+1)}\\
=-\frac{1}{2\pi i}\int_{C}\sum _{l=0}^{\infty } \frac{2^{-v-\mu } b^{1+m+l p} (a b)^w c^l w^{-1-k} z^{v+\mu }
   \binom{d}{l} }{(1+m+l p+w) \Gamma (1+v) \Gamma (1+\mu )}\\ \times \, _3F_4\left(1+m+l p+w,\frac{1}{2}+\frac{v}{2}+\frac{\mu }{2},1+\frac{v}{2}+\frac{\mu}{2}\right. \\ \left.;1+v,2+m+l p+w,1+\mu ,1+v+\mu ;-b z^2\right)dw
\end{multline}
where $|Re(z)|<1$.
where $|Re(b)|<1$. We are able to switch the order of integration and summation over $w$ using Tonellii's theorem for  integrals and sums see page 177 in \cite{gelca}, since the summand is of bounded measure over the space $\mathbb{C} \times [0,\infty)$
\begin{theorem}
The product of Bessel functions.
\begin{multline}\label{eq:b4}
\int_0^b x^{m-\frac{v}{2}-\frac{\mu }{2}} \left(1+c x^p\right)^d J_v\left(\sqrt{x} z\right) J_{\mu
   }\left(\sqrt{x} z\right) (-\log (a x))^k \, dx,\\
=\sum _{j=0}^{\infty } \sum _{l=0}^{\infty } \frac{(-1)^j c^l z^{2
   j+v+\mu } \binom{d}{l} \Gamma (1+k,-((1+j+m+l p) \log (a b))) (1+j+v+\mu )_j}{j! \Gamma (1+j+v) \Gamma (1+j+\mu )
   (1+j+m+l p)^{k+1} a^{1+j+m+l p} 2^{2 j+v+\mu }}
\end{multline}
where $Re(b)>0,|Re(m)|<1$.
\end{theorem}
\begin{proof}
Since the right-hand sides of (\ref{eq:b2}) and (\ref{eq:b3}) are equivalent relative to (\ref{eq:b1}) we may equate the left-hand sides and simplify the gamma function to yield the stated result.
\end{proof}
\begin{example}
 The form in equation (\ref{eq:b4}) can be used to derive (2.12.32.2-7) in \cite{prud2}, the integral form over $x\in]1,\infty)$ for equations (2.12.32.8-10,13) in \cite{prud2}. Equations (13.1.66-67) in \cite{erdt2}.
\end{example}
\begin{example}
 Use equation (\ref{eq:b4}) and set $k=0,a=1$ and simplify.
\begin{multline}
\sum _{j=0}^{\infty } \frac{(-b)^j \left(\frac{z}{2}\right)^{2 j} \Gamma (1+2 j+v+\mu ) \,
   _2F_1\left(-d,\frac{1}{p}+\frac{j}{p}+\frac{m}{p};1+\frac{1}{p}+\frac{j}{p}+\frac{m}{p};-b^p c\right)}{(1+j+m)
   \Gamma (1+j) \Gamma (1+j+v) \Gamma (1+j+\mu ) \Gamma (1+j+v+\mu )}\\
=\sum _{l=0}^{\infty } \frac{\left(c
   b^p\right)^l \binom{d}{l} \, _3F_4\left(1+m+l p,\frac{1}{2}+\frac{v}{2}+\frac{\mu }{2},1+\frac{v}{2}+\frac{\mu
   }{2};2+m+l p,1+v,1+\mu ,1+v+\mu ;-b z^2\right)}{(1+m+l p) \Gamma (1+v) \Gamma (1+\mu )}
\end{multline}
where $Re(b)>0,|Re(m)|<1$.
\end{example}
\begin{example}
 Use equation (\ref{eq:b4}) and set $k\to 0,a\to 1,p\to 1,c\to 1,d\to 0$ and simplify.
\begin{multline}
\int_0^b x^{1+2 m} J_v(x z) J_{\mu }(x z) \, dx
=\frac{2^{-1-v-\mu } b^{2+2 m+v+\mu } z^{v+\mu }
 }{\left(1+m+\frac{v}{2}+\frac{\mu}{2}\right) \Gamma (1+v) \Gamma (1+\mu )}\\ \times \,_3F_4\left(\frac{1}{2}+\frac{v}{2}+\frac{\mu }{2},1+\frac{v}{2}+\frac{\mu }{2},1+m+\frac{v}{2}+\frac{\mu
   }{2}\right. \\ \left.
;1+v,2+m+\frac{v}{2}+\frac{\mu }{2},1+\mu ,1+v+\mu ;-b^2 z^2\right)
\end{multline}
where $|Re(z)|<1$.
\end{example}
\subsection{The product of the Bessel and logarithmic functions}
The contour integral representation form involving the product of the Bessel and logarithmic functions is given by;
\begin{multline}\label{eq:bl1}
\frac{1}{2\pi i}\int_{C}\int_{0}^{b}2^v a^w w^{-k-1} J_v(x) x^{2 m+v+2 w}dxdw\\
=\frac{1}{2\pi i}\int_{C}\frac{a^w w^{-k-1} b^{2 m+2 v+2 w+1} \,
   _1F_2\left(m+v+w+\frac{1}{2};v+1,m+v+w+\frac{3}{2};-\frac{b^2}{4}\right)}{\Gamma (v+1) (2 m+2 v+2 w+1)}dw
\end{multline}
where $|Re(b)|<1$.
\subsubsection{Left-hand side contour integral representation}
Using a generalization of Cauchy's integral formula \ref{intro:cauchy}, we form the definite integral by replacing $y$ by $\log{ax}$ and multiply both sides by $\frac{\left(-\frac{1}{4}\right)^j x^{2 (j+m+v)}}{j! \Gamma (j+v+1)}$ then take the infinite sum of both sides over $j \in [0,\infty)$ to get;
\begin{multline}\label{eq:bl2}
\int_{0}^{b}\frac{2^v x^{2 m+v} J_v(x) \log ^k\left(a x^2\right)}{k!}dx
=\frac{1}{2\pi i}\sum_{j=0}^{\infty}\int_{0}^{b}\int_{C}\frac{\left(-\frac{1}{4}\right)^j w^{-k-1} \left(a x^2\right)^w x^{2 (j+m+v)}}{j! \Gamma(j+v+1)}dwdx\\
=\frac{1}{2\pi i}\int_{0}^{b}\int_{C}\sum_{j=0}^{\infty}\frac{\left(-\frac{1}{4}\right)^j w^{-k-1} \left(a x^2\right)^w x^{2 (j+m+v)}}{j! \Gamma(j+v+1)}dwdx\\
=\frac{1}{2\pi i}\int_{0}^{b}\int_{C}2^v w^{-k-1} \left(a x^2\right)^w x^{2 m+v}J_v(x)dwdx\\
=\frac{1}{2\pi i}\int_{C}\int_{0}^{b}2^v w^{-k-1} \left(a x^2\right)^w x^{2 m+v}J_v(x)dxdw
\end{multline}
where $|Re(b)|<1$.
We are able to switch the order of integration over $t$ and $w$ using Fubini's theorem for multiple integrals see page 178 in \cite{gelca}, since the integrand is of bounded measure over the space $\mathbb{C} \times [0,b]$.
\subsubsection{Right-hand side contour integral representation}
Use equation (\ref{eq:4.4}) and set $y=0$, replace $m\to j+m+v+\frac{1}{2},a\to a b^2$ and multiply both sides by $\frac{(-1)^j 2^{-2 j-1} b^{2 (j+m+v)+1}}{j! \Gamma (j+v+1)}$. Next take the infinite sum of both sides over $j \in [0,\infty)$ and simplify to get;
\begin{multline}\label{eq:bl3}
\sum _{j=0}^{\infty } \frac{(-1)^j 2^{-2 j+k}  b^{1+2 j+2 m+2 v}
   \left(b^2\right)^{-\frac{1}{2}-j-m-v}  \Gamma
   \left(1+k,-\left(\left(\frac{1}{2}+j+m+v\right) \log \left(a b^2\right)\right)\right)}{j! k! \Gamma
   (1+j+v)(-1-2 j-2 m-2 v)^{k+1}a^{\frac{1}{2}+j+m+v}}\\
=\frac{1}{2\pi i}\sum_{j=0}^{\infty}\int_{C}\frac{\left(-\frac{1}{4}\right)^j a^w w^{-k-1} b^{2 j+2 m+2 v+2 w+1}}{j! \Gamma (j+v+1) (2 j+2 m+2 v+2 w+1)}dw\\
=\frac{1}{2\pi i}\int_{C}\sum_{j=0}^{\infty}\frac{\left(-\frac{1}{4}\right)^j a^w w^{-k-1} b^{2 j+2 m+2 v+2 w+1}}{j! \Gamma (j+v+1) (2 j+2 m+2 v+2 w+1)}dw\\
=-\frac{1}{2\pi i}\int_{C}\frac{a^w b^{1+2 m+2 v+2 w} w^{-1-k} \,
   _1F_2\left(\frac{1}{2}+m+v+w;1+v,\frac{3}{2}+m+v+w;-\frac{b^2}{4}\right)}{(1+2 m+2 v+2 w) \Gamma (1+v)}dw
\end{multline}
where $|Re(b)|<1$.
where $|Re(b)|<1$. We are able to switch the order of integration and summation over $w$ using Tonellii's theorem for  integrals and sums see page 177 in \cite{gelca}, since the summand is of bounded measure over the space $\mathbb{C} \times [0,\infty)$
\begin{theorem}\label{eq:bl4}
The finite Mellin transform involving the Bessel and logarithmic functions.
\begin{multline}
\int_0^b x^{2 m+v} J_v(x) \log ^k\left(a x^2\right) \, dx\\
=\sum _{j=0}^{\infty } \frac{(-1)^{-1+j-k} 2^{-2
   j+k-v} \Gamma \left(1+k,-\left(\left(\frac{1}{2}+j+m+v\right) \log \left(a b^2\right)\right)\right)}{j! \Gamma
   (1+j+v) (1+2 j+2 m+2 v)^{k+1} a^{\frac{1}{2}+j+m+v}}
\end{multline}
where $Re(\mu+v+1)>0$.
\end{theorem}
\begin{proof}
Since the right-hand sides of (\ref{eq:bl2}) and (\ref{eq:bl3}) are equivalent relative to (\ref{eq:bl1}) we may equate the left-hand sides and simplify the gamma function to yield the stated result.
\end{proof}
\begin{example}
Derivation of equation (2.3.1) in \cite{luke}, (10.22.10) in [DLMF,\href{https://dlmf.nist.gov/10.22.E10}{10.22.10}]. Use equation (\ref{eq:bl4}) and set $k=0,m\to \frac{m-v}{2},m\to \mu, x\to t,b\to x$ and simplify.
\begin{multline}\label{eq:bl5}
\int_0^x t^{\mu } J_v(t) \, dt
=\frac{2^{-v} x^{1+v+\mu } }{(1+v+\mu ) \Gamma (1+v)}\, _1F_2\left(\frac{1}{2}+\frac{v}{2}+\frac{\mu
   }{2};1+v,\frac{3}{2}+\frac{v}{2}+\frac{\mu }{2};-\frac{x^2}{4}\right)\\
=\frac{x^{\mu }
   \Gamma \left(\frac{v}{2}+\frac{\mu }{2}+\frac{1}{2}\right) }{\Gamma \left(\frac{v}{2}-\frac{\mu }{2}+\frac{1}{2}\right)}\sum _{k=0}^{\infty } \frac{\left((v+2 k+1) \Gamma
   \left(\frac{v}{2}-\frac{\mu }{2}+\frac{1}{2}+k\right)\right) J_{v+2 k+1}(x)}{\Gamma \left(\frac{v}{2}+\frac{\mu
   }{2}+\frac{3}{2}+k\right)}
\end{multline}
where $Re(\mu+v+1)>0$.
\end{example}
\begin{example}
Derivation of equation (2.3.2) in \cite{luke}. Use equation (\ref{eq:bl5}) and form a second equation by replacing $\mu\to -\mu$ and set $v=0$. Next take the difference of the integral $\int_{0}^{x} t^{-\mu}dt=-\frac{x^{1-\mu}}{1-\mu}$ and simplify.
\begin{multline}
\int_0^x t^{-\mu } (1-J_0(t)) \, dt=x^{1-\mu } \left(\frac{1}{1-\mu }-\frac{\,
   _1F_2\left(\frac{1}{2}-\frac{\mu }{2};1,\frac{3}{2}-\frac{\mu }{2};-\frac{x^2}{4}\right)}{1-\mu }\right)
\end{multline}
where $|Re(x)|<1$.
\end{example}
Next apply l'Hopital's rule as $\mu \to 1$ and simplify to get;
\begin{multline}
\int_0^x \frac{1-J_0(t)}{t} \, dt=\frac{\frac{\partial }{\partial \mu }\left(x^{1-\mu }-x^{1-\mu } \, _1F_2\left(\frac{1}{2}-\frac{\mu
   }{2};1,\frac{3}{2}-\frac{\mu }{2};-\frac{x^2}{4}\right)\right)}{\frac{\partial (1-\mu )}{\partial \mu }}\Bigr|_{\mu=1}
\end{multline}
where $|Re(x)|<1$.
\begin{example}
 Use equation (\ref{eq:bl5}) and set $a=b=1$ and apply l'Hopital's rule as $k\to -1$ and simplify.
\begin{equation}
\int_0^1 \frac{\left(x^m-x^s\right) J_v(x)}{\log \left(\frac{1}{x}\right)} \, dx=\sum _{j=0}^{\infty }
   \frac{(-1)^j 2^{-2 j-v}}{j! \Gamma (1+j+v)} \log \left(\frac{1+2 j+s+v}{1+2 j+m+v}\right)
\end{equation}
where $0< Re(m)<1,0< Re(s)<1$.
\end{example}
\begin{example}
 Use equation (\ref{eq:bl5}) and set $k=a=1$. Next take the $n$-th derivative with respect to $m$ and simplify by replacing $m\to \frac{m-v-1}{2,n\to \beta}$ and simplify.
\begin{multline}
\int_0^1 x^{-1+m} J_v(x) \log ^{\beta }(x) \, dx=\frac{(-1)^{\beta } \Gamma (1+\beta )}{2^v}\sum
   _{j=0}^{\infty } \frac{\left(-\frac{1}{4}\right)^j}{\Gamma (1+j) \Gamma (1+j+v) (2 j+m+v)^{\beta
   +1}}
\end{multline}
where $0< Re(m)<1$.
\end{example}
\subsection{Bessel function with generalized second parameter}
The contour integral representation form involving the Bessel function is given by;
\begin{multline}\label{eq:bp1}
\frac{1}{2\pi i}\int_{C}\int_{0}^{b}2^{-v-2} a^w w^{-k-1} \alpha ^v (\alpha  x)^{-v} x^{m+v+w} J_v(2 x \alpha )dxdw\\
=\frac{1}{2\pi i}\int_{C}\frac{2^{-v-2} a^w w^{-k-1}\alpha ^v b^{m+v+w+1} }{\Gamma (v+1) (m+v+w+1)}\\ \times \, _1F_2\left(\frac{m}{2}+\frac{v}{2}+\frac{w}{2}+\frac{1}{2};v+1,\frac{m}{2}+\frac{v}{2}+\frac{w}{2}+\frac{3}{2};-b^2 \alpha ^2\right)dw
\end{multline}
where $|Re(b)|<1$.
\subsubsection{Left-hand side contour integral representation}
Using a generalization of Cauchy's integral formula \ref{intro:cauchy}, we form the definite integral by replacing $y$ by $\log{ax}$ and multiply both sides by $\frac{(-1)^j 2^{-v-2} \alpha ^{2 j+v} x^{2 j+m+v}}{j! \Gamma (j+v+1)}$ then take the infinite sum of both sides over $j \in [0,\infty)$ to get;
\begin{multline}\label{eq:bp2}
\int_{0}^{b}\frac{2^{-v-2} \alpha ^v x^{m+v} (\alpha  x)^{-v} \log ^k(a x) J_v(2 x \alpha )}{k!}dxdw\\
=\frac{1}{2\pi i}\sum_{j=0}^{\infty}\int_{0}^{b}\int_{C}\frac{(-1)^j 2^{-v-2}w^{-k-1} (a x)^w \alpha ^{2 j+v} x^{2 j+m+v}}{j! \Gamma (j+v+1)}dxdw\\
=\frac{1}{2\pi i}\int_{0}^{b}\int_{C}\sum_{j=0}^{\infty}\frac{(-1)^j 2^{-v-2}w^{-k-1} (a x)^w \alpha ^{2 j+v} x^{2 j+m+v}}{j! \Gamma (j+v+1)}dxdw\\
=\frac{1}{2\pi i}\int_{0}^{b}\int_{C}2^{-v-2} w^{-k-1} \alpha^v (a x)^w x^{m+v} (\alpha  x)^{-v} J_v(2 x \alpha )dxdw\\
=\frac{1}{2\pi i}\int_{C}\int_{0}^{b}2^{-v-2} w^{-k-1} \alpha^v (a x)^w x^{m+v} (\alpha  x)^{-v} J_v(2 x \alpha )dxdw
\end{multline}
where $|Re(b)|<1$.
We are able to switch the order of integration over $t$ and $w$ using Fubini's theorem for multiple integrals see page 178 in \cite{gelca}, since the integrand is of bounded measure over the space $\mathbb{C} \times [0,b]$.
\subsubsection{Right-hand side contour integral representation}
Use equation (\ref{eq:4.4}) and set $y=0$, replace $m\to 2 j+m+v+1,a\to a b$ and multiply both sides by $\frac{(-1)^j 2^{-v-2} \alpha ^{2 j+v} b^{2 j+m+v+1}}{j! \Gamma (j+v+1)}$. Next take the infinite sum of both sides over $j \in [0,\infty)$ and simplify to get;
\begin{multline}\label{eq:bp3}
\sum _{j=0}^{\infty } \frac{(-1)^j 2^{-2-v} a^{-1-2 j-m-v}  \alpha ^{2 j+v} \Gamma
   (1+k,-((1+2 j+m+v) \log (a b)))}{j! k! \Gamma (1+j+v)(-1-2 j-m-v)^{k+1}}\\
=-\frac{1}{2\pi i}\sum_{j=0}^{\infty}\int_{C}\frac{(-1)^j 2^{-v-2} w^{-k-1} (a b)^w \alpha ^{2 j+v} b^{2 j+m+v+1}}{j! \Gamma(j+v+1) (2 j+m+v+w+1)}\\
=-\frac{1}{2\pi i}\int_{C}\sum_{j=0}^{\infty}\frac{(-1)^j 2^{-v-2} w^{-k-1} (a b)^w \alpha ^{2 j+v} b^{2 j+m+v+1}}{j! \Gamma(j+v+1) (2 j+m+v+w+1)}\\
=-\frac{1}{2\pi i}\int_{C}\frac{2^{-2-v} b^{1+m+v} (a b)^w w^{-1-k} \alpha ^v \,_1F_2\left(\frac{1}{2}+\frac{m}{2}+\frac{v}{2}+\frac{w}{2};1+v,\frac{3}{2}+\frac{m}{2}+\frac{v}{2}+\frac{w}{2};-b^2 \alpha ^2\right)}{(1+m+v+w) \Gamma (1+v)}dw
\end{multline}
where $|Re(b)|<1$.
We are able to switch the order of integration and summation over $w$ using Tonellii's theorem for  integrals and sums see page 177 in \cite{gelca}, since the summand is of bounded measure over the space $\mathbb{C} \times [0,\infty)$
\begin{theorem}
The finite definite integral of the Bessel function.
\begin{multline}\label{eq:bp4}
\int_0^b x^{m-1} J_v(x \alpha ) (-\log (a x))^k \, dx\\
=\frac{1}{a^m}\sum _{j=0}^{\infty }(-1)^j
   \left(\frac{\alpha }{2 a}\right)^{2 j+v} \frac{ \Gamma (1+k,-((2 j+m+v) \log (a b)))}{j! \Gamma (1+j+v) (2 j+m+v)^{k+1}}
\end{multline}
where $Re(\alpha+v+1)>0$.
\end{theorem}
\begin{proof}
Since the right-hand sides of (\ref{eq:bp2}) and (\ref{eq:bp3}) are equivalent relative to (\ref{eq:bp1}) we may equate the left-hand sides and simplify the gamma function to yield the stated result.
\end{proof}
\begin{example}
 Use equation (\ref{eq:bp4}) and set $a=b=1$ and apply l'Hopital's rule as $k\to -1$ and simplify.
\begin{multline}
\int_0^1 \frac{\left(x^s-x^m\right) J_v(x \alpha )}{\log (x)} \, dx=-\sum _{j=0}^{\infty }
   \frac{\left(\frac{\alpha }{2}\right)^{2 j+v} (-1)^j }{j! \Gamma
   (1+j+v)}\log \left(\frac{2 j+m+v+1}{2 j+s+v+1}\right)
\end{multline}
where $0< Re(m)<1, 0< Re(s)<1$.
\end{example}
\section{Exercise}
Use the method in equations (\ref{eq:bp2}) and (\ref{eq:bp3}) and (\ref{eq:bp1}) with the appropriate coefficient functions to derive the following formulae.
\begin{example}
Bessel function and integrals of the general form
\begin{multline}\label{eq:bft}
\int_0^b x^{-1+m} \left(1+c x^p\right)^d J_v(x \alpha ) (-\log (a x))^k \, dx\\
\left(\frac{\alpha
   }{2}\right)^v \sum _{j=0}^{\infty } \sum _{l=0}^{\infty } c^l \left(-\left(\frac{\alpha
   }{2}\right)^2\right)^j \binom{d}{l}\frac{ \Gamma (1+k,-(2 j+m+l p+v) \log (a b))}{j! \Gamma (1+j+v) a^{2 j+m+l p+v}
   (2 j+m+l p+v)^{k+1}}
\end{multline}
where $Re(\alpha+v+1)>0,Re(p)>0$.
\end{example}
\begin{example}
Show that the form given in equation (\ref{eq:bft}) can be used to derive  equations (2.12.1.1), (2.12.3.1-3), (2.12.4.1-13), and the form over $x\in [1,\infty])$ for equations (2.12.1.2-4,9), (2.12.2.2), (2.12.3.4-12), (2.12.4.14-37) in \cite{prud2}.
 \end{example}

\begin{example}
Use equation (\ref{eq:bft}) and set $k=-1,a\to e^a$. Then form a second equation by replacing $a\to -a$ and take their difference and simplify.
\begin{multline}
\int_0^b \frac{x^{-1+m} \left(1+c x^p\right)^d J_v(x \alpha )}{a^2+\log ^2(x)} \, dx
=\sum _{j=0}^{\infty }
   \sum _{l=0}^{\infty } \frac{i (-1)^j c^l \alpha ^{2 j+v} \binom{d}{l} }{a j! \Gamma (1+j+v) 2^{1+2
   j+v}}\\ \times \left(e^{i a (2 j+m+l p+v)} \Gamma
   \left(0,-\left((2 j+m+l p+v) \log \left(b e^{-i a}\right)\right)\right)\right. \\ \left.
-e^{-i a (2 j+m+l p+v)} \Gamma
   \left(0,-\left((2 j+m+l p+v) \log \left(b e^{i a}\right)\right)\right)\right)
\end{multline}
where $Re(\alpha+v+1)>0,Re(p)>0$.
\end{example}
\begin{example}
Generalized form of (6.552.4) in \cite{grad}. Use equation (\ref{eq:bft}) and set $k\to 0,m\to 1,c\to -1,p\to 2,d\to -\frac{1}{2}$ and simplify.
\begin{multline}
\int_0^b \frac{J_v(x \alpha )}{\sqrt{1-x^2}} \, dx=\sum _{j=0}^{\infty } \frac{(-1)^j 2^{-2 j-v} b^{1+2 j+v}
   \alpha ^{2 j+v} \, _2F_1\left(\frac{1}{2},\frac{1}{2}+j+\frac{v}{2};\frac{3}{2}+j+\frac{v}{2};b^2\right)}{(1+2
   j+v) \Gamma (1+j) \Gamma (1+j+v)}
\end{multline}
where $0< Re(v), 0< Re(\alpha)<1,|Re(b)|\leq 1$.
\end{example}
\begin{example}
Generalized form for (6.554.2-3) in \cite{grad}. Use equation (\ref{eq:bft}) and set $k\to 0,c\to -1,p\to 2,d\to -\frac{1}{2}$ and simplify.
\begin{multline}
\int_0^b \frac{x^{-1+m} J_v(x \alpha )}{\sqrt{1-x^2}} \, dx\\
=\sum _{j=0}^{\infty } \frac{(-1)^j 2^{-2 j-v}
   b^{2 j+m+v} \alpha ^{2 j+v} \,
   _2F_1\left(\frac{1}{2},j+\frac{m}{2}+\frac{v}{2};1+j+\frac{m}{2}+\frac{v}{2};b^2\right)}{(2 j+m+v) \Gamma (1+j)
   \Gamma (1+j+v)}
\end{multline}
where $0< Re(v), 0< Re(\alpha)<1,|Re(b)|\leq 1,0< Re(m)<1$.
\end{example}
\begin{example}
Generalized form for (6.561.1,5,9,13(7)) in \cite{grad}. Use equation (\ref{eq:bft}) and set $k=d=0,l\to 0,a\to 1,c\to 1,m\to m,b\to 1$ and simplify.
\begin{multline}
\int_0^1 x^{-1+m} J_v(x \alpha ) \, dx=\frac{2^{-v} \alpha ^v }{(m+v) \Gamma
   (1+v)}\,
   _1F_2\left(\frac{m}{2}+\frac{v}{2};1+\frac{m}{2}+\frac{v}{2},1+v;-\frac{\alpha ^2}{4}\right)
\end{multline}
where $0< Re(v), 0< Re(\alpha)<1,0< Re(m)<1$
\end{example}
\begin{example}
Generalized form for equations (6.567.1,3,6,9,13), (6.569) in \cite{grad}. Use equation (\ref{eq:bft}) and set $k\to 0,a\to 1,c\to -1,p\to 2,d\to \mu$ and simplify.
\begin{multline}
\int_0^b x^{-1+m} \left(1-x^2\right)^{\mu } J_v(x \alpha ) \, dx\\
=\sum _{j=0}^{\infty } \frac{2^{-2 j-v} b^{2
   j+m+v} \alpha ^v \left(-\alpha ^2\right)^j \, _2F_1\left(j+\frac{m}{2}+\frac{v}{2},-\mu
   ;1+j+\frac{m}{2}+\frac{v}{2};b^2\right)}{(2 j+m+v) \Gamma (1+j) \Gamma (1+j+v)}
\end{multline}
where $0< Re(v), 0< Re(\alpha)<1,0< Re(m)<1$.
\end{example}
\begin{example}
Generalized form for equations (6.592.5, 10, 11, 16) in \cite{grad}. . Use equation (\ref{eq:bft}) and set $k\to 0,c\to -1,d\to \mu -1,b=1,p=2,m\to m/2$ and replace $x\to t^{1/2}$ and simplify.
\begin{multline}
\int_0^1 (1-x)^{-1+\mu } x^{-1+m} J_v\left(\sqrt{x} \alpha \right) \, dx\\
=\sum _{l=0}^{\infty } \frac{(-1)^l
   2^{1-v} \alpha ^v \binom{-1+\mu }{l} \, _1F_2\left(l+m+\frac{v}{2};1+l+m+\frac{v}{2},1+v;-\frac{\alpha
   ^2}{4}\right)}{(2 l+2 m+v) \Gamma (1+v)}
\end{multline}
where $Re(\mu)>2$.
\end{example}
\subsection{Bessel functions, exponentials, and powers}
The contour integral representation form involving the Bessel functions, exponentials, and powers is given by;
\begin{multline}\label{eq:bfe1}
\frac{1}{2\pi i}\int_{C}\int_{0}^{b}a^w e^{f x^g} w^{-1-k} x^{v+2 (m+w)} \left(1+c x^p\right)^d \alpha ^v (x \alpha )^{-v} J_v(x \alpha )dxdw\\
=\frac{1}{2\pi i}\int_{C}\sum_{h=0}^{\infty } \sum _{l=0}^{\infty } \frac{2^{-v} a^w b^{1+g h+l p+v+2 (m+w)} c^l f^h w^{-1-k} \alpha ^v\binom{d}{l} }{(1+g h+2 m+l p+v+2 w) h! \Gamma(1+v)}\\ \times \, _1F_2\left(\frac{1}{2}+\frac{g h}{2}+m+\frac{l p}{2}+\frac{v}{2}+w;1+v,\frac{3}{2}+\frac{gh}{2}+m+\frac{l p}{2}+\frac{v}{2}+w;-\frac{1}{4} b^2 \alpha ^2\right)dw
\end{multline}
where $|Re(b)|<1$.
\subsubsection{Left-hand side contour integral representation}
Using a generalization of Cauchy's integral formula \ref{intro:cauchy}, we form the definite integral by replacing $y$ by $\log{ax}$ and multiply both sides by $\frac{(-1)^j c^l f^h 2^{-2 j-v} \binom{d}{l} \alpha ^{2 j+v} x^{g h+2 j+l p+2 m+v}}{h! j! \Gamma (j+v+1)}$ then take the infinite sum of both sides over $j \in [0,\infty),h \in [0,\infty),l \in [0,\infty)$ to get;
\begin{multline}\label{eq:bfe2}
\frac{1}{2\pi i}\int_{C}\int_{0}^{b}\frac{\alpha ^v e^{f x^g} x^{2 m+v} (\alpha  x)^{-v} \log ^k\left(a x^2\right) \left(c x^p+1\right)^d J_v(x\alpha )}{k!}dxdw\\
=\frac{1}{2\pi i}\sum_{j=0}^{\infty}\sum_{h=0}^{\infty}\sum_{l=0}^{\infty}\int_{0}^{b}\int_{C}\frac{(-1)^j c^l f^h 2^{-2 j-v} w^{-k-1} \left(a x^2\right)^w \binom{d}{l} \alpha
   ^{2 j+v} x^{g h+2 j+l p+2 m+v}}{h! j! \Gamma (j+v+1)}dwdx\\
=\frac{1}{2\pi i}\int_{0}^{b}\int_{C}\sum_{j=0}^{\infty}\sum_{h=0}^{\infty}\sum_{l=0}^{\infty}\frac{(-1)^j c^l f^h 2^{-2 j-v} w^{-k-1} \left(a x^2\right)^w \binom{d}{l} \alpha
   ^{2 j+v} x^{g h+2 j+l p+2 m+v}}{h! j! \Gamma (j+v+1)}dwdx\\
=\frac{1}{2\pi i}\int_{0}^{b}\int_{C}w^{-k-1} \alpha ^v \left(a x^2\right)^w e^{f x^g} x^{2 m+v} (\alpha  x)^{-v} \left(c x^p+1\right)^dJ_v(x \alpha )dwdx\\
=\frac{1}{2\pi i}\int_{C}\int_{0}^{b}w^{-k-1} \alpha ^v \left(a x^2\right)^w e^{f x^g} x^{2 m+v} (\alpha  x)^{-v} \left(c x^p+1\right)^dJ_v(x \alpha )dxdw
\end{multline}
where $|Re(b)|<1$.
We are able to switch the order of integration over $t$ and $w$ using Fubini's theorem for multiple integrals see page 178 in \cite{gelca}, since the integrand is of bounded measure over the space $\mathbb{C} \times [0,b]$.
\subsubsection{Right-hand side contour integral representation}
Use equation (\ref{eq:4.4}) and set $y=0$, replace $m\to \frac{g h}{2}+j+\frac{l p}{2}+m+\frac{v}{2}+\frac{1}{2},a\to a b^2$ and multiply both sides by $\frac{(-1)^j c^l f^h 2^{-2 j-v-1} \binom{d}{l} \alpha ^{2 j+v} b^{g h+2 j+l p+2 m+v+1}}{h! j! \Gamma (j+v+1)}$. Next take the infinite sums of both sides over $j \in [0,\infty),h \in [0,\infty),l \in [0,\infty)$ and simplify to get;
\begin{multline}\label{eq:bfe3}
\sum _{l=0}^{\infty } \sum _{h=0}^{\infty } \sum _{j=0}^{\infty } \frac{(-1)^j 2^{-2 j+k-v}
    c^l f^h \alpha ^{2 j+v} \binom{d}{l} \Gamma \left(1+k,-\frac{1}{2} (1+g h+2 j+2 m+l p+v) \log \left(a b^2\right)\right)}{h! j! k! \Gamma
   (1+j+v)(-1-g h-2 j-2 m-l p-v)^{1+k}a^{\frac{1}{2} (1+g h+2 j+2 m+l p+v)} }\\
=-\frac{1}{2\pi i}\sum _{l=0}^{\infty } \sum _{h=0}^{\infty }\int_{C}\frac{(-1)^j a^w c^l f^h 2^{-2 j-v} w^{-k-1} \binom{d}{l} \alpha ^{2 j+v} b^{g h+2 j+l p+2 m+v+2 w+1}}{h! j! \Gamma
   (j+v+1) (g h+2 j+l p+2 m+v+2 w+1)}dw\\
=-\frac{1}{2\pi i}\int_{C}\sum _{l=0}^{\infty } \sum _{h=0}^{\infty }\frac{(-1)^j a^w c^l f^h 2^{-2 j-v} w^{-k-1} \binom{d}{l} \alpha ^{2 j+v} b^{g h+2 j+l p+2 m+v+2 w+1}}{h! j! \Gamma
   (j+v+1) (g h+2 j+l p+2 m+v+2 w+1)}dw\\
=-\frac{1}{2\pi i}\int_{C}\sum _{l=0}^{\infty } \sum _{h=0}^{\infty } \frac{2^{-v} a^w b^{1+g h+2 m+l p+v+2
   w} c^l f^h w^{-1-k} \alpha ^v \binom{d}{l} }{(1+g h+2 m+l p+v+2 w) h! \Gamma (1+v)}\\ \times \, _1F_2\left(\frac{1}{2}+\frac{g h}{2}+m+\frac{l
   p}{2}+\frac{v}{2}+w;1+v,\frac{3}{2}+\frac{g h}{2}+m+\frac{l p}{2}+\frac{v}{2}+w;-\frac{1}{4} b^2 \alpha
   ^2\right)dw
\end{multline}
where $|Re(b)|<1$.
We are able to switch the order of integration and summation over $w$ using Tonellii's theorem for  integrals and sums see page 177 in \cite{gelca}, since the summand is of bounded measure over the space $\mathbb{C} \times [0,\infty)$
\begin{theorem}
\begin{multline}\label{eq:bfe4}
\int_0^b e^{f x^g} x^m \left(1+c x^p\right)^d J_v(x \alpha ) (-\log (a x))^k \, dx\\
=\frac{1}{a^{m+1+v}}\sum _{j=0}^{\infty} \sum _{h=0}^{\infty } \sum _{l=0}^{\infty }\left(\frac{\alpha }{2}\right)^{2 j+v}  \frac{(-1)^j c^l f^h \binom{d}{l} \Gamma \left(1+k,- (1+g h+2 j+m+l p+v) \log \left(a b\right)\right)}{h! j!\Gamma (1+j+v) a^{g h+2 j+l p} (1+g h+2 j+m+l p+v)^{k+1}}
\end{multline}
where $0< Re(v), 0< Re(\alpha)<1,0< Re(m)<1$.
\end{theorem}
\begin{proof}
Since the right-hand sides of (\ref{eq:bfe2}) and (\ref{eq:bfe3}) are equivalent relative to (\ref{eq:bfe1}) we may equate the left-hand sides and simplify the gamma function to yield the stated result.
\end{proof}
\begin{example}
General form for equations (6.625.1-2) in \cite{grad}. Use equation (\ref{eq:bfe4}) and set $k\to 0,a\to 1,b\to 1,g\to 1,c\to -1,d\to \mu -1,p\to 1,f\to -\beta$ and simplify.
\begin{multline}
\int_0^1 e^{-x \beta } (1-x)^{-1+\mu } x^m J_v(x \alpha ) \, dx
=\sum _{h=0}^{\infty } \frac{2^{-v} \alpha ^v(-\beta )^h \Gamma (1+h+m+v) \Gamma (\mu ) }{\Gamma (1+h) \Gamma (1+v) \Gamma (1+h+m+v+\mu )}\\ \times \,_2F_3\left(\frac{1}{2}+\frac{h}{2}+\frac{m}{2}+\frac{v}{2},1+\frac{h}{2}+\frac{m}{2}+\frac{v}{2};\right. \\ \left.
1+v,\frac{1}{2}+\frac{h}{2}+\frac{m}{2}+\frac{v}{2}+\frac{\mu }{2},1+\frac{h}{2}+\frac{m}{2}+\frac{v}{2}+\frac{\mu}{2};-\frac{\alpha ^2}{4}\right)
\end{multline}
where $|Re(\alpha)|<1$.
\end{example}
\begin{example}
Generalized form for Erd\'{e}yli Bessel definite integrals for equations (6.719.1-2) in \cite{grad}.  Use equation (\ref{eq:bfe4}) and set $g\to 1,k\to 0,a\to 1,d\to -\frac{1}{2},p\to 2$ then form a second equation by replacing $f\to -f$ and take their difference. Then replace $f\to i\beta, c\to -c^2$ then $c\to 1/c$ and simplify.
\begin{multline}
\int_0^b \frac{x^m J_v(x \alpha ) \sin (x \beta )}{\sqrt{c^2-x^2}} \, dx\\
=\sum _{j=0}^{\infty } \sum
   _{h=0}^{\infty } \frac{(-1)^j i^{1+h} 2^{-1-2 j-v} \left(-1+(-1)^h\right) b^{1+h+2 j+m+v} \alpha ^{2 j+v} \beta^h }{c (1+h+2 j+m+v) \Gamma (1+h) \Gamma (1+j) \Gamma (1+j+v)}\\ \times \,_2F_1\left(\frac{1}{2},\frac{1}{2}+\frac{h}{2}+j+\frac{m}{2}+\frac{v}{2};\frac{3}{2}+\frac{h}{2}+j+\frac{m}{2}+\frac{v}{2};\frac{b^2}{c^2}\right)
\end{multline}
where $|Re(b)|<1$.
\end{example}
\begin{example}
Generalized form for Table (6.674) in \cite{grad}. Use equation (\ref{eq:bft}) and set $k\to 0,a\to 1,g\to 1,p\to 0,d\to 0,c\to 1$. Next form a second equation by replacing $f\to -f$ and take their difference. Next take the derivative with respect to $f$ and set $m=-1$. Then set $m=0$ in the initial equation and multiply each equation by $\sinh(c)$ and $\cosh(c)$ respectively and take their difference and simplify.
\begin{multline}
\int_0^b J_v(x \alpha ) \sinh (c-x) \, dx
=\sum _{h=0}^{\infty } \left(\frac{2^{-2-v} \left(-1+(-1)^h\right)
   b^{2+h+v} \alpha ^{2+v} (b \cosh (c)-h \sinh (c))}{(3+h+v) \Gamma (1+h) \Gamma (2+v)}\right. \\ \left.  \times \,_1F_2\left(\frac{3}{2}+\frac{h}{2}+\frac{v}{2};\frac{5}{2}+\frac{h}{2}+\frac{v}{2},2+v;-\frac{1}{4} b^2 \alpha^2\right) \right. \\ \left.
+\frac{2^{-2-v} \left(-1+(-1)^h\right)b^{2+h+v} \alpha ^{2+v} (-b \cosh (c)+h \sinh (c))}{(2+h+v) \Gamma (1+h) \Gamma (2+v)}\right. \\ \left. \times \, _1F_2\left(1+\frac{h}{2}+\frac{v}{2};2+\frac{h}{2}+\frac{v}{2},2+v;-\frac{1}{4} b^2\alpha ^2\right)\right. \\ \left.
-\frac{2^{-1-v}\left(-1+(-1)^h\right) b^{h+v} \alpha ^v (-b (h+v)\cosh (c)+h (1+h+v) \sinh (c))}{(h+v) \Gamma (1+h) \Gamma (1+v)}\right. \\ \left. \times \,_1F_2\left(\frac{h}{2}+\frac{v}{2};1+\frac{h}{2}+\frac{v}{2},1+v;-\frac{1}{4} b^2 \alpha ^2\right) \right. \\ \left.
+\frac{2^{-1-v} \left(-1+(-1)^h\right) b^{h+v}\alpha ^v (-b (h+v) \cosh (c)+h (1+h+v) \sinh (c))}{(1+h+v) \Gamma (1+h) \Gamma (1+v)}\right. \\ \left. \times \, _1F_2\left(\frac{1}{2}+\frac{h}{2}+\frac{v}{2};\frac{3}{2}+\frac{h}{2}+\frac{v}{2},1+v;-\frac{1}{4}b^2 \alpha ^2\right) \right)
\end{multline}
where $Re(v)>0,Re(\alpha)>0,Re(c)>0,Re(b)>0$.
\end{example}
\begin{example}
Generalized form for Table (6.699.9-10) in \cite{grad}. Use equation (\ref{eq:bft}) and set $k\to 0,a\to 1,g\to 1,p\to 0,d\to 0,c\to 1$. Next form a second equation by replacing $f\to -f$ and take their difference and replace $f\to i\beta$ and simplify.
\begin{multline}
\int_0^b x^m J_v(x \alpha ) \sin (x \beta ) \, dx\\
=\sum _{h=0}^{\infty } \frac{i^{1+h} 2^{-1-v}
   \left(-1+(-1)^h\right) b^{1+h+m+v} \alpha ^v \beta ^h 
}{(1+h+m+v) \Gamma (1+h) \Gamma (1+v)}\\ \times \, _1F_2\left(\frac{1}{2}+\frac{h}{2}+\frac{m}{2}+\frac{v}{2};\frac{3}{2}+\frac{h}{2}+\frac{m}{2}+\frac{v}{2},1+v;-\frac{1}{4} b^2 \alpha ^2\right)
\end{multline}
where $Re(v)>0,Re(\alpha)>0,Re(\beta)>0,Re(b)>0$.
\end{example}
\begin{example}
Generalized form for Table (6.719) in \cite{grad}. Use equation (\ref{eq:bft}) and set $k\to 0,a\to 1,g\to 1,f\to i\beta ,p\to 2,c\to -c^2,m\to 0,d\to -\frac{1}{2}$ and simplify.
\begin{multline}
\int_0^b \frac{e^{i x \beta } J_v(x \alpha )}{\sqrt{1-c^2 x^2}} \, dx\\
=\sum _{h=0}^{\infty } \sum
   _{j=0}^{\infty } \frac{(-1)^j 2^{-2 j-v} b^{1+h+2 j+v} \alpha ^{2 j+v} (i \beta )^h \,
   _2F_1\left(\frac{1}{2},\frac{1}{2}+\frac{h}{2}+j+\frac{v}{2};\frac{3}{2}+\frac{h}{2}+j+\frac{v}{2};b^2
   c^2\right)}{(1+h+2 j+v) \Gamma (1+h) \Gamma (1+j) \Gamma (1+j+v)}
\end{multline}
where $Re(v)>0,Re(\alpha)>0,Re(\beta)>0,Re(b)>0$.
\end{example}
\begin{example}
Generalized form of equations (13.1.56-58) in \cite{erdt2}. Equations (13.1.59-65) can be derived when $p\neq 1$ and simplify. Use equation (\ref{eq:bft}) and set $f\to i \beta ,g\to 1,p\to 1,d\to \mu -1,k\to 0,a\to 1,c\to -c$ and simplify.
\begin{multline}
\int_0^b e^{i x \beta } x^m (1-c x)^{-1+\mu } J_v(x \alpha ) \, dx\\
=b^{1+m} \left(\frac{b \alpha }{2}\right)^v
   \sum _{h=0}^{\infty } \sum _{j=0}^{\infty } \frac{\left(\frac{\alpha  b i}{2}\right)^{2 j} (i b \beta )^h
  }{(1+h+2 j+m+v) \Gamma (1+h) \Gamma (1+j) \Gamma(1+j+v)}\\ \times  \, _2F_1(1+h+2 j+m+v,1-\mu ;2+h+2 j+m+v;b c)
\end{multline}
where $Re(\mu)>0,Re(v)<-1/2$.
\end{example}
\begin{example}
The Hurwitz-Lerch zeta function. Use equation (\ref{eq:bft}) and set $a=b=1,d=-1$ and simplify using [DLMF, \href{https://dlmf.nist.gov/25.14.E1}{25.14.1}].
\begin{multline}\label{eq:lerch1}
\int_0^1 \frac{e^{x^{\gamma } \beta } x^{-1+m} J_v(x \alpha ) \log ^{-1+k}\left(\frac{1}{x}\right)}{1+x^s z}
   \, dx\\
=\sum
   _{h=0}^{\infty } \sum _{j=0}^{\infty } \frac{(-1)^j 2^{-2 j-v} s^{-k} \alpha ^{2 j+v} \beta ^h \Gamma (k) \Phi
   \left(-z,k,\frac{2 j+m+v+h \gamma }{s}\right)}{h! j! \Gamma (1+j+v)}
\end{multline}
where $Re(\mu)>0,Re(v)<-1/2,Re(s)>0$.
\end{example}
\begin{example}
The Hurwitz zeta function. Use equation (\ref{eq:lerch1}) and set $z=-1$ and form a second equation by replacing $m\to r$ and take their difference and simplify using [DLMF, \href{https://dlmf.nist.gov/25.14.E2}{25.14.2}].
\begin{multline}
\int_0^1 \frac{e^{x^{\gamma } \beta } \left(x^{m-1}-x^{r-1}\right) J_v(x \alpha ) \log
   ^{-1+k}\left(\frac{1}{x}\right)}{1-x^s} \, dx\\
=\sum _{h=0}^{\infty } \sum _{j=0}^{\infty } \frac{(-1)^j 2^{-2
   j-v} s^{-k} \alpha ^{2 j+v} \beta ^h \Gamma (k) \left(\zeta \left(k,\frac{2 j+m+v+h \gamma }{s}\right)-\zeta
   \left(k,\frac{2 j+r+v+h \gamma }{s}\right)\right)}{h! j! \Gamma (1+j+v)}
\end{multline}
where $Re(\mu)>0,Re(v)<-1/2,Re(s)>0$.
\end{example}
\section{More single and double integrals}
In this section we expand definite integral derivations listed in current volumes and books using equation(s) stated in the latter. Here we use equation (\ref{eq:8.4}) and make simple variable changes to get the following theorem.
\begin{theorem}
\begin{multline}\label{eq:marian}
\int_0^b x^m \left(1+x^{\alpha } \beta \right)^{\lambda } \log ^k\left(\frac{1}{a x}\right) \, dx,\\
=\sum
   _{f=0}^{\infty } a^{-1-m-f \alpha } (1+m+f \alpha )^{-1-k} \beta ^f \binom{\lambda }{f} \Gamma \left(1+k,(1+m+f
   \alpha ) \log \left(\frac{1}{a b}\right)\right)
\end{multline}
where $Re(b)>0$.
\end{theorem}
\begin{example}
 Derivation of equation (1) in \cite{marian}.  Here we use equation (\ref{eq:marian}) and set $a\to 1,b\to e^{-a},m\to 2 n-1,\alpha \to 2,\beta \to 1,\lambda \to -n,k\to m$ and simplify;
 \begin{multline}
2^n \int_a^{\infty } \frac{x^m}{(1+\exp (2 x))^n} \, dx=2^n \int_0^{\exp (-a)} \frac{x^{2 n-1} \log
   ^m\left(\frac{1}{x}\right)}{\left(1+x^2\right)^n} \, dx\\
=2^{n-m-1} \sum _{f=0}^{\infty } \frac{\binom{-n}{f}
   \Gamma (1+m,2 a (f+n))}{(f+n)^{m+1}}
\end{multline}
where $Re(n)>0$.
\end{example}
\section{Derivation of Table 4.626 in Gradshteyn and Ryzhik page 616 (2015)}
\begin{example}
Derivation of entry (4.626.1) in \cite{grad}. Here we use equation (\ref{eq:marian}) and set $\lambda \to -1,\alpha \to 1,\beta \to -y z,a\to y,b\to 1,m\to u-1,k\to s$ and simplify;
\begin{multline}
\int _0^1\int _0^1\frac{x^{-1+u} y^{-1+v} (-\log (x y))^s}{1-x y z}dydx
=\frac{-\Gamma (1+s) \Phi (z,1+s,u)+\Gamma (1+s) \Phi (z,1+s,v)}{u-v}
\end{multline}
where $u,v,s,z\in\mathbb{C}$.
\end{example}
\begin{example}
Derivation of entry (4.626.2) in \cite{grad}. Here we use equation (\ref{eq:marian}) and set $\lambda \to -1,\alpha \to 1,\beta \to -y z,a\to y,b\to 1,m\to u-1,k\to s$ and simplify;
\begin{equation}
\int _0^1\int _0^1\frac{x^{-1+u} y^{-1+u} (-\log (x y))^s}{1-x y z}dydx=\Gamma (2+s) \Phi (z,2+s,u)
\end{equation}
where $u,v,s,z\in\mathbb{C}$.
\end{example}
\begin{example}
Derivation of entry (4.626.3) in \cite{grad}. Here we use equation (\ref{eq:marian}) and set $\lambda \to -1,\alpha \to 2,\beta \to y^2,a\to y,b\to 1,m\to 1,k\to 1$ and simplify;
\begin{equation}
\int _0^1\int _0^1\frac{x \log (x y)}{1+x^2 y^2}dydx=\frac{1}{48} \left(-48 C+\pi ^2\right)
\end{equation}
\end{example}
\begin{example}
Derivation of entry (4.626.4) in \cite{grad}. Here we use equation (\ref{eq:marian}) and set $\lambda \to -1,\alpha \to 2,\beta \to -y^2,a\to y,b\to 1,m\to 1,k\to 1$ and simplify;
\begin{equation}
\int _0^1\int _0^1\frac{x \log (x y)}{1-x^2 y^2}dydx=-\frac{\pi ^2}{12}
\end{equation}
\end{example}
\begin{example}
Derivation of entry (4.626.5) in \cite{grad}. Here we use equation (\ref{eq:marian}) and set $\lambda \to -1,\alpha \to 2,\beta \to y^2,a\to y,b\to 1,m\to 0,k\to 0$ and simplify;
\begin{equation}
\int _0^1\int _0^1\frac{1}{1+x^2 y^2}dydx=C
\end{equation}
\end{example}
\begin{example}
Derivation of entry (4.626.6) in \cite{grad}. Here we use equation (\ref{eq:marian}) and set $\lambda \to -1,\alpha \to 1,\beta \to -y z,a\to y,b\to 1,m\to 0,k\to -1$ and simplify;
\begin{equation}
\int _0^1\int _0^1\frac{1}{(1-x y z) \log (x y)}dydx=\frac{\log (1-z)}{z}
\end{equation}
\end{example}
\begin{example}
Derivation of entry (4.626.7) in \cite{grad}. Here we use equation (\ref{eq:marian}) and set $\lambda \to -1,\alpha \to 1,\beta \to -y,a\to y,b\to 1,m\to 0,k\to 0$ and simplify;
\begin{equation}
\int _0^1\int _0^1\frac{1}{1-x y}dydx=\frac{\pi ^2}{6}
\end{equation}
\end{example}
\begin{example}
Derivation of entry (4.626.8) in \cite{grad}. Here we use equation (\ref{eq:marian}) and set $\lambda \to -1,\alpha \to 1,\beta \to -y,a\to y,b\to 1,m\to 0,k\to 1$ and simplify;
\begin{equation}
\int _0^1\int _0^1\frac{\log (x y)}{1-x y}dydx=-2 \zeta (3)
\end{equation}
\end{example}
\begin{example}
Derivation of entry (4.626.9) in \cite{grad}. Here we use equation (\ref{eq:marian}) and set $\lambda \to 0,\alpha \to 0,\beta \to 1,a\to y,b\to 1,m\to u-1,k\to -1$ and simplify;
\begin{equation}
\int _0^1\int _0^1\frac{x^{-1+u} y^{-1+v}}{\log (x y)}dydx=-\frac{\log \left(\frac{v}{u}\right)}{-u+v}
\end{equation}
\end{example}
\begin{example}
Derivation of entry (4.626.10) in \cite{grad}. Here we use equation (\ref{eq:marian}) and set $\lambda \to 0,\alpha \to 0,\beta \to 1,a\to y,b\to 1,m\to u-1,k\to -1$ and simplify;
\begin{equation}
\int _0^1\int _0^1\frac{x^{-1+u} y^{-1+u}}{\log (x y)}dydx=-\frac{1}{u}
\end{equation}
\end{example}
\begin{example}
Derivation of entry (4.626.11) in \cite{grad}. Here we use equation (\ref{eq:marian}) and set $\lambda \to -1,\alpha \to 1,\beta \to y,a\to y,b\to 1,m\to 1,k\to -1$ and simplify;
\begin{multline}
\int _0^1\int _0^1-\frac{x}{(1+x y) \log (x y)}dydx=\sum _{f=0}^{\infty } (-1)^f \log
   \left(\frac{2+f}{1+f}\right)=\log \left(\frac{\pi }{2}\right)
\end{multline}
\end{example}
\begin{example}
Derivation of entry (4.626.12) in \cite{grad}. Here we use equation (\ref{eq:marian}) and set $\lambda \to -1,\alpha \to 2,\beta \to y^2,a\to y,b\to 1,m\to 1,k\to -1$ and simplify;
 \begin{multline}
\int _0^1\int _0^1-\frac{x}{\left(1+x^2 y^2\right) \log (x y)}dydx=\sum _{f=0}^{\infty } (-1)^f \log
   \left(\frac{2 (1+f)}{1+2 f}\right)=\log \left(\frac{\sqrt{2 \pi }}{\Gamma
   \left(\frac{3}{4}\right)^2}\right)
\end{multline}
\end{example}
\begin{example}
 Derivation of entry (4.626.13) in \cite{grad}. Here we use equation (\ref{eq:marian}) and set $\lambda \to -1,\alpha \to 2,\beta \to y^2,a\to y,b\to 1,m\to 0,k\to -1$ and simplify;
 \begin{equation}
\int _0^1\int _0^1\frac{1}{\left(1+x^2 y^2\right) \log (x y)}dydx=-\frac{\pi }{4}
\end{equation}
\end{example}
\begin{example}
 Derivation of entry (4.626.14) in \cite{grad}. Here we use equation (\ref{eq:marian}) and set $\lambda \to -1,\alpha \to 3,\beta \to -y^3,a\to y,b\to 1,m\to 0,k\to 0$ and simplify;
 \begin{equation}
\int _0^1\int _0^1\frac{y}{1-x^3 y^3}dydx=\frac{\pi }{3 \sqrt{3}}
\end{equation}
\end{example}
\begin{example}
 Derivation of entry (4.626.15) in \cite{grad}. Here we use equation (\ref{eq:marian}) and set $\lambda \to -1,\alpha \to 2,\beta \to -y^2,a\to y,b\to 1,m\to 0,k\to 0$ and simplify;
 \begin{equation}
\int _0^1\int _0^1\frac{1}{1-x^2 y^2}dydx=\frac{\pi ^2}{8}
\end{equation}
\end{example}
\begin{example}
Derivation of entry (4.626.16) in \cite{grad}. Here we use equation (\ref{eq:marian}) and set $\lambda \to -\frac{1}{2},\alpha \to 1,\beta \to y,a\to y,b\to 1,m\to 0,k\to -1$ and simplify;
\begin{equation}
\int _0^1\int _0^1\frac{1}{\sqrt{1+x y} \log (x y)}dydx=-2 \left(-1+\sqrt{2}\right)
\end{equation}
\end{example}
\begin{example}
Derivation of entry (4.626.17) in \cite{grad}. Here we use equation (\ref{eq:marian}) and set $\lambda \to -\frac{1}{2},\alpha \to 2,\beta \to -y^2,a\to y,b\to 1,m\to 0,k\to -1$ and simplify;
\begin{equation}
\int _0^1\int _0^1\frac{1}{\sqrt{1-x^2 y^2} \log (x y)}dydx=-\frac{\pi }{2}
\end{equation}
\end{example}
\begin{example}
Derivation of entry (4.626.18) in \cite{grad}. Here we use equation (\ref{eq:marian}) and set $\lambda \to -1,\alpha \to 2,\beta \to -y^2,a\to y,b\to 1,m\to 0,k\to \frac{1}{2}$ and simplify;
\begin{equation}
\int _0^1\int _0^1\frac{\sqrt{-\log (x y)}}{1-x^2 y^2}dydx=\frac{3}{16} \left(-1+4 \sqrt{2}\right)
   \sqrt{\frac{\pi }{2}} \zeta \left(\frac{5}{2}\right)
\end{equation}
\end{example}
\begin{example}
Derivation of entry (4.626.19) in \cite{grad}. Here we use equation (\ref{eq:marian}) and set $\lambda \to -1,\alpha \to 3,\beta \to -y^3,a\to y,b\to 1,k\to \frac{1}{2},m\to v$ and simplify;
\begin{multline}
\int _0^1\int _0^1\frac{x^v y^u \sqrt{-\log (x y)}}{1-x^3 y^3}dydx=\frac{\sqrt{\pi } \left(-\sqrt{3} \zeta
   \left(\frac{3}{2},\frac{1+u}{3}\right)+\sqrt{3} \zeta \left(\frac{3}{2},\frac{1+v}{3}\right)\right)}{18
   (u-v)}
\end{multline}
\end{example}
\begin{example}
Derivation of entry (4.626.20) in \cite{grad}. Here we use equation (\ref{eq:marian}) and set $\beta \to y^{\beta },a\to y,b\to 1$ and simplify;
\begin{multline}
\int _0^1\int _0^1\frac{\sqrt{y} \log \left(\log \left(\frac{1}{x y}\right)\right)}{\sqrt{x}+x
   \sqrt{y}}dydx=\log ^2(2)+\gamma  (-1+\log (4))-\gamma _1+\gamma _1\left(\frac{3}{2}\right)
\end{multline}
\end{example}
\section{Entries in Table 3.1.5 in Prudnikov volume I (1986)}
\begin{example}
Derivation of entry (3.1.5.27) in \cite{prud1}. Here we use equation (\ref{eq:marian}) and set $k\to 0,a\to 1,\lambda \to \frac{1}{2},\alpha \to 2,m\to 1,x\to y,\beta \to -\sin ^2(x)$ and simplify;
  \begin{multline}
\int _0^{\frac{\pi }{2}}\int _0^ky \sqrt{1-y^2 \sin ^2(x)}dydx\\
=\frac{1}{4} k^2 \pi  \,
   _2F_1\left(-\frac{1}{2},\frac{1}{2};2;k^2\right)\\
   =\frac{1}{3} \left(\left(1+k^2\right)
   E\left(k^2\right)-\left(1-k^2\right) K\left(k^2\right)\right)
\end{multline}
\end{example}
\begin{example}
Derivation of entry (3.1.5) in \cite{prud1} additional formula. Here we use equation (\ref{eq:marian}) and set $k\to 0,a\to 1,\lambda \to -\frac{1}{2},\alpha \to 2,m\to 1,x\to y,\beta \to -\sin ^2(x)$ and simplify;
\begin{equation}
\int _0^{\frac{\pi }{2}}\int _0^k\frac{y}{\sqrt{1-y^2 \sin ^2(x)}}dydx=E\left(k^2\right)+\left(-1+k^2\right)
   K\left(k^2\right)
\end{equation}
\end{example}
\begin{example}
Derivation of entry (3.1.5) in \cite{prud1} additional formula. Here we use equation (\ref{eq:marian}) and set $k\to 0,a\to 1,\lambda \to -\frac{1}{2},\alpha \to 2,m\to 1,x\to y,\beta \to -\sin ^2(x)$ and simplify;
\begin{equation}
\int _0^{\frac{\pi }{2}}\int _0^k\frac{y}{\sqrt{1-y^2 \sin ^2(x)}}dydx=E\left(k^2\right)+\left(-1+k^2\right)
   K\left(k^2\right)
\end{equation}
\end{example}
\begin{example}
Derivation of entry (3.1.5) in \cite{prud1} additional formula. Here we use equation (\ref{eq:marian}) and set $k\to 0,a\to 1,\lambda \to -\frac{3}{2},\alpha \to 2,m\to 1,x\to y,\beta \to -\sin ^2(x)$ and simplify;
\begin{equation}
\int _0^{\frac{\pi }{2}}\int _0^k\frac{y}{\left(1-y^2 \sin
   ^2(x)\right)^{3/2}}dydx=K\left(k^2\right)-E\left(k^2\right)
\end{equation}
where $0< Re(k) <1$
\end{example}
\begin{example}
 Derivation of entry (3.1.5.28) in \cite{prud1} generalized form.  Here we use equation (\ref{eq:marian}) and set $k\to 0,a\to 1,\lambda \to \frac{1}{2},\alpha \to 2,m\to 0,x\to y,\beta \to -\sin ^2(x)$ and simplify;
 \begin{equation}
\int _0^{\frac{\pi }{2}}\int _0^b\sqrt{1-y^2 \sin ^2(x)}dydx=\frac{b \pi }{2}  \,
   _3F_2\left(-\frac{1}{2},\frac{1}{2},\frac{1}{2};1,\frac{3}{2};b^2\right)
\end{equation}
where $|Re(b)| <1$.
\end{example}
\begin{example}
Derivation of entry (3.1.5.29) in \cite{prud1} generalized form.  Here we use equation (\ref{eq:marian}) and set $k\to 0,a\to 1,\lambda \to \frac{1}{2},\alpha \to 2,m\to 0,x\to y,\beta \to \tan ^2(x)$ and simplify;
\begin{multline}
\int _0^{\frac{\pi }{2}}\int _0^b\cos (x) \sqrt{1+y^2 \tan ^2(x)}dydx=\frac{2 b \pi 
   E\left(1-b^2\right)+G_{3,3}^{3,2}\left(\frac{1}{b^2} \left|
\begin{array}{c}
 -\frac{1}{2},\frac{1}{2},1 \\
 0,0,0 \\
\end{array}
\right.\right)}{4 \pi }
\end{multline}
where $Re(b) >0$.
\end{example}
\begin{example}
Derivation of entry (3.1.5.30) in \cite{prud1} infinite series involving the hypergeometric function form.  Here we use equation (\ref{eq:marian}) and set $k\to 0,a\to 1,\lambda \to \frac{1}{2},\alpha \to 2,m\to 1,x\to y,\beta \to -\sin ^2(x)$ and simplify;
\begin{multline}
\int _0^{\varphi }\int _0^ky \sqrt{1-y^2 \sin ^2(x)}dydx\\
=\sum _{f=0}^{\infty } \frac{(-1)^f k^{2+2 f}
   \binom{\frac{1}{2}}{f} }{(1+2 f) (2+2 f)}\, _2F_1\left(\frac{1}{2},\frac{1}{2}+f;\frac{3}{2}+f;\sin ^2(\varphi )\right) \sin ^{1+2
   f}(\varphi )\\
=\frac{1}{3} \left(\left(1+k^2\right) E\left(\varphi
   \left|k^2\right.\right)-\left(1-k^2\right) F\left(\varphi \left|k^2\right.\right)-\left(1-\sqrt{1-k^2 \sin
   ^2(\varphi )}\right) \cot (\varphi )\right)
\end{multline}
where $0< Re(k)<1, Re(\varphi)>0$.
\end{example}
\begin{example}
Derivation of entry (3.1.5.31) in \cite{prud1} infinite series involving the hypergeometric function form. Here we use equation (\ref{eq:marian}) and set $k\to 0,a\to 1,\lambda \to \frac{1}{2},\alpha \to 2,m\to 1,x\to y,\beta \to -\sin ^2(x)$ and simplify;
\begin{multline}
\int _0^{\varphi }\int _0^1y \sqrt{1-y^2 \sin ^2(x)}dydx\\
=\sum _{f=0}^{\infty } \frac{(-1)^f \binom{\frac{1}{2}}{f} \,
   _2F_1\left(\frac{1}{2},\frac{1}{2}+f;\frac{3}{2}+f;\sin ^2(\varphi )\right) \sin ^{1+2 f}(\varphi )}{2+6 f+4 f^2}\\
=\frac{\sin ^2(\varphi )-\cos (\varphi )+1}{3 \sin (\varphi )}
\end{multline}
\end{example}
\begin{example}
Derivation of entry (3.1.5.32) in \cite{prud1} infinite series involving the hypergeometric function form. Here we use equation (\ref{eq:marian}) and set $k\to 0,a\to 1,\lambda \to \frac{1}{2},\alpha \to 2,m\to 0,x\to y,\beta \to -\sin ^2(x)$ and simplify;
\begin{equation}
\int _0^{\frac{\pi }{2}}\int _0^b\sqrt{1-y^2 \sin ^2(x)}dydx=\frac{1}{2} b \pi  \, _3F_2\left(-\frac{1}{2},\frac{1}{2},\frac{1}{2};1,\frac{3}{2};b^2\right)
\end{equation}
\end{example}
\begin{example}
Derivation of entry (3.1.5.33) in \cite{prud1} infinite series involving the hypergeometric function form. Here we use equation (\ref{eq:marian}) and set $k\to 0,a\to 1,\lambda \to -\frac{1}{2},\alpha \to 2,m\to 0,x\to y,\beta \to \tan ^2(x)$ and simplify;
\begin{equation}
\int _0^{\frac{\pi }{2}}\int _0^1\frac{\sec (x)}{\sqrt{1+y^2 \tan ^2(x)}}dydx=\frac{\pi ^2}{4}
\end{equation}
\end{example}
\begin{example}
Derivation of entry (3.1.5.34) in \cite{prud1} infinite series involving the hypergeometric function form. Here we use equation (\ref{eq:marian}) and set $k\to 0,a\to 1,\lambda \to -\frac{1}{2},\alpha \to 2,m\to 0,x\to y,\beta \to -\sin ^2(x),b\to 1$ and simplify;
\begin{equation}
\int _0^{\frac{\pi }{2}}\int _0^1\frac{1}{\sqrt{1-y^2 \sin ^2(x)}}dydx=2 C
\end{equation}
\end{example}
\begin{example}
Derivation of entry (3.1.5.35) in \cite{prud1} infinite series involving the hypergeometric function form. Here we use equation (\ref{eq:marian}) and set $k\to 0,a\to 1,\lambda \to -\frac{1}{2},\alpha \to 2,m\to 1,x\to y,\beta \to -\sin ^2(x)$ and simplify;
\begin{equation}
\int _0^{\frac{\pi }{2}}\int _0^k\frac{y}{\sqrt{1-y^2 \sin ^2(x)}}dydx=E\left(k^2\right)+\left(-1+k^2\right) K\left(k^2\right)
\end{equation}
\end{example}
\begin{example}
Derivation of entry (3.1.5.36) in \cite{prud1} infinite series involving the hypergeometric function form. Here we use equation (\ref{eq:marian}) and set $k\to 0,a\to 1,\lambda \to -\frac{1}{2},\alpha \to 2,m\to 1,x\to y,\beta \to -\sin ^2(x)$ and simplify;
\begin{multline}
\int _0^{\varphi }\int _0^1\frac{y}{\sqrt{1-y^2 \sin ^2(x)}}dydx\\
=\sum _{f=0}^{\infty } \frac{(-1)^f \binom{-\frac{1}{2}}{f} \,
   _2F_1\left(\frac{1}{2},\frac{1}{2}+f;\frac{3}{2}+f;\sin ^2(\varphi )\right) \sin ^{1+2 f}(\varphi )}{2+6 f+4 f^2}\\
=\frac{1-\cos (\varphi )}{\sin (\varphi )}
\end{multline}
where $\varphi \in\mathbb{C},|Im(\varphi)|<1$.
\end{example}
\begin{example}
Derivation of entry (3.1.5.37) in \cite{prud1} infinite series involving the hypergeometric function form. Here we use equation (\ref{eq:marian}) and set $k\to 0,a\to 1,\lambda \to -\frac{1}{2},\alpha \to 2,m\to 1,x\to y,\beta \to -\sin ^2(x)$ and simplify;
\begin{multline}
\int _0^{\varphi }\int _0^k\frac{y}{\sqrt{1-y^2 \sin ^2(x)}}dydx\\
=\sum _{f=0}^{\infty } \frac{k^{2+2 f} \binom{-\frac{1}{2}}{f} \cos (\varphi ) \,
   _2F_1\left(\frac{1}{2}\frac{1}{2}+f;\frac{3}{2}+f;\sin ^2(\varphi )\right) (-1)^f \sin ^{2 f}(\varphi ) \tan (\varphi )}{2+6 f+4 f^2}\\
=E\left(\varphi
   \left|k^2\right.\right)-\left(1-k^2\right) F\left(\varphi \left|k^2\right.\right)+\left(\sqrt{1-(k \sin (\varphi ))^2}-1\right) \cot (\varphi )
\end{multline}
where $|Im(\varphi)|<1,0< Re(k)<1$.
\end{example}
\begin{example}
Derivation of entry (3.1.5.38) in \cite{prud1} infinite series involving the hypergeometric function form. Errata. Here we use equation (\ref{eq:marian}) and set $k\to 0,a\to 1,\lambda \to -\frac{1}{2},\alpha \to 2,m\to 1,x\to y,\beta \to -\sin ^2(x)$ and simplify;
\begin{multline}
\int _0^{\varphi }\int _0^k\frac{y}{\sqrt{1-y^2 \sin ^2(x)} \left(1-\alpha ^2 \sin ^2(x)\right)}dydx\\
=\sum _{f=0}^{\infty } \frac{(-1)^f k^{2+2 f}
   \binom{-\frac{1}{2}}{f} \sec ^{1+2 f}(\varphi ) \sin ^{1+2
   f}(\varphi )}{2+6 f+4 f^2}\\ \times
F_1\left(\frac{1}{2}+f;f,1;\frac{3}{2}+f;-\tan ^2(\varphi ),\left(-1+\alpha ^2\right) \tan ^2(\varphi )\right) \\
\neq\left(k^2-\alpha ^2\right) \Pi \left(\varphi ;\alpha ^2|k^2\right)+E\left(\varphi \left|k^2\right.\right)-F\left(\varphi
   \left|k^2\right.\right)+\left(\sqrt{1-(k \sin (\varphi ))^2}-1\right) \cot (\varphi )
\end{multline}
where $|Im(\varphi)|<1,0< Re(k)<1$.
\end{example}
\begin{example}
Derivation of entry (3.1.5.39) in \cite{prud1} infinite series involving the hypergeometric function form. Errata. Here we use equation (\ref{eq:marian}) and set $k\to 0,a\to 1,\lambda \to -\frac{1}{2},\alpha \to 2,m\to 1,x\to y,\beta \to -\sin ^2(x)$ and simplify;
\begin{multline}
\int _0^{\frac{\pi }{2}}\int _0^k\frac{y}{\sqrt{1-y^2 \sin ^2(x)} \left(1-\alpha ^2 \sin ^2(x)\right)}dydx\\
=\sum _{f=0}^{\infty } \frac{1}{4} (-1)^f k^{2+2 f} \sqrt{\pi }
   \binom{-\frac{1}{2}}{f} \left(\frac{\sqrt{\pi } \alpha ^{-2 f}}{(1+f) \sqrt{1-\alpha ^2}}\right. \\ \left.
+\frac{\Gamma \left(\frac{1}{2}+f\right) }{\left(-1+\alpha ^2\right) \Gamma (2+f)}\left(-\alpha ^2 \,
   _2F_1\left(1,\frac{1}{2}+f;-\frac{1}{2};1-\alpha ^2\right)\right.\right. \\ \left.\left.
+\left(-2 (1+f)+(3+2 f) \alpha ^2\right) \, _2F_1\left(1,\frac{1}{2}+f;\frac{1}{2};1-\alpha
   ^2\right)\right)\right)\\
\neq\left(k^2-\alpha ^2\right) \Pi \left(\frac{\pi }{2};\alpha
   ^2|k^2\right)-K\left(k^2\right)+E\left(k^2\right)
\end{multline}
where $0< Re(k)<1$.
\end{example}
\begin{example}
Derivation of entry (3.1.5.40) in \cite{prud1} infinite series involving the hypergeometric function form. Errata. Here we use equation (\ref{eq:marian}) and set $k\to 0,a\to 1,\lambda \to -\frac{1}{2},\alpha \to 2,m\to 1,x\to y,\beta \to -\sin ^2(x),b\to 1$ and simplify;
\begin{multline}
\int _0^{\varphi }\int _0^1\frac{y}{\sqrt{1-y^2 \sin ^2(x)} \left(1-\alpha ^2 \sin ^2(x)\right)}dydx\\
=\sum _{f=0}^{\infty } \frac{(-1)^f
    \binom{-\frac{1}{2}}{f} \sec ^{1+2 f}(\varphi ) \sin ^{1+2
   f}(\varphi )}{2+6 f+4 f^2}\\ \times
F_1\left(\frac{1}{2}+f;f,1;\frac{3}{2}+f;-\tan ^2(\varphi ),\left(-1+\alpha ^2\right) \tan ^2(\varphi )\right)\\
=\left(1-\alpha ^2\right) \Pi \left(\varphi ;\left.\alpha ^2\right|1\right)-\alpha ^2 \Pi \left(\varphi ;\left.\alpha ^2\right|0\right)+\frac{1-\cos
   (\varphi )}{\sin (\varphi )}-\log (\tan (\varphi )+\sec (\varphi ))
\end{multline}
where $Re(\varphi)>0$.
\end{example}
\begin{example}
Derivation of entry (3.1.5.41) in \cite{prud1} infinite series involving the hypergeometric function form. Here we use equation (\ref{eq:marian}) and set $k\to 0,a\to 1,\lambda \to -\frac{1}{2},\alpha \to 2,m\to 0,x\to y,\beta \to -\sin ^2(x)$ and simplify;
\begin{equation}
\int _0^{\frac{\pi }{2}}\int _0^b\frac{1}{\sqrt{1-y^2 \sin ^2(x)}}dydx=\frac{1}{2} b \pi  \,
   _3F_2\left(\frac{1}{2},\frac{1}{2},\frac{1}{2};1,\frac{3}{2};b^2\right)
\end{equation}
where $Re(b)>0$.
\end{example}
\begin{example}
Here we use equation (\ref{eq:marian}) and set $\lambda \to -1,\alpha \to 1$ and simplify;
\begin{multline}
\int _0^b\int _0^1\frac{x^{-1+u} y^{-1+v} \log ^s\left(\frac{b}{x y}\right)}{1-x y z}dydx=\frac{b^u \Gamma
   (1+s) (\Phi (b z,1+s,v)-\Phi (b z,1+s,u))}{u-v}
\end{multline}
where $Re(b)>0$.
\end{example}
\section{Conclusion}
In this paper, we have presented derivations using a definite integral transforms along with some interesting special cases, using contour integration. We will be using this approach to further these integral formulae in future work. The results presented were numerically verified for both real and imaginary and complex values of the parameters in the integrals using Mathematica by Wolfram.
\end{document}